\documentclass[aap, preprint]{imsart}

\RequirePackage[OT1]{fontenc}
\RequirePackage[colorlinks,citecolor=blue,urlcolor=blue]{hyperref}

\setattribute{journal}{name}{}
\usepackage[numbers]{natbib}
\usepackage{amsmath, amsthm, amsfonts, amssymb}
\usepackage{algorithm}
\usepackage{algorithmic}
\usepackage{xcolor}
\usepackage{hyperref}
\usepackage{tikz}
\usepackage{enumitem}
\usepackage[normalem]{ulem}

\usepackage{graphicx}
\usepackage{dsfont}
\usepackage{mathrsfs}
\usepackage{rotating}

\usepackage{fullpage}

\usepackage{constants}

\setlist{
listparindent=\parindent,
parsep=0pt,
}

\newconstantfamily{Const}{symbol=C}
\newconstantfamily{Asm}{symbol=\pmb{A}}

\usetikzlibrary{calc}
\allowdisplaybreaks[2]

\arxiv{arXiv:1807.11018}

\startlocaldefs
\newtheorem{theorem}{Theorem}[section]
\newtheorem{corollary}[theorem]{Corollary}
\newtheorem{lemma}[theorem]{Lemma}
\newtheorem{definition}[theorem]{Definition}
\newtheorem{remark}[theorem]{Remark}
\newtheorem{proposition}[theorem]{Proposition}
\newtheorem{fact}[theorem]{Fact}

\numberwithin{equation}{section}
\numberwithin{theorem}{section}

\newcommand{\set}[1]{\langle #1\rangle}

\newcommand{\conNu}{C^{\nu}}

\newcommand{\delL}{\delta^L}
\newcommand{\conN}{C^N}
\newcommand{\delN}{\delta^N}

\newcommand{\den}{\textnormal{den}}

\newcommand{\lk}{\textnormal{lk}}
\newcommand{\st}{\textnormal{st}}
\newcommand{\tr}{\top}

\newcommand{\rk}{\textnormal{rank}}
\newcommand{\eye}{\mathbb{I}}

\newcommand{\Real}{\mathbb R}

\newcommand{\dReal}{\mathbb{R}^d}
\newcommand{\pReal}{\mathbb{R}_{> 0}}
\newcommand{\pSeq}{\pReal^\infty}
\newcommand{\Idx}{\Gamma}
\renewcommand{\d}[1]{\mathrm{d #1}}

\newcommand{\Gp}{G^\prime}
\newcommand{\Vp}{V^\prime}
\newcommand{\Ep}{E^\prime}
\newcommand{\Grid}{\mathbb{Z}}

\newcommand{\dGrid}{\Grid^d}
\newcommand{\dGridCrDir}{\dGrid \times \tbinom{\langle d \rangle}{b_k}}
\newcommand{\dGridCrPk}{\dGrid \times \cP_k}
\newcommand{\dGridCrNk}{\dGrid \times \cN_k}
\newcommand{\IdxnCrDir}{\Idx_n \times \tbinom{\langle d \rangle}{b_k}}
\newcommand{\IdxnCrPk}{\Idx_n \times \cP_k}
\newcommand{\IdxnCrNk}{\Idx_n \times \cN_k}
\newcommand{\setd}{\langle d \rangle}
\newcommand{\bGrp}{\mathcal{B}}
\newcommand{\cGrp}{\mathcal{Z}}
\newcommand{\hGrp}{\mathcal{H}}

\newcommand{\covM}{W}
\newcommand{\cov}{\textnormal{Cov}}
\newcommand{\var}{\textnormal{Var}}
\newcommand{\ones}{\vec{\mathbf{1}}}
\newcommand{\onesM}{\mathbf{1}}
\newcommand{\indc}{\mathds{1}}

\renewcommand{\Pr}{\mathbb{P}}
\newcommand{\alphap}{\alpha^\prime}
\newcommand{\omegap}{\omega^\prime}
\newcommand{\gammap}{\gamma^\prime}

\newcommand{\rhoa}{\tilde{\rho}}
\newcommand{\rhoaN}{\tilde{\rho}^N}
\newcommand{\rhoaL}{\tilde{\rho}^L}

\newcommand{\rhop}{\rho^\prime}
\newcommand{\mup}{\mu^\prime}
\newcommand{\vecp}[1]{\bar{#1}^\prime}

\newcommand{\tOm}{(\vec{t}, \, \vec{\Omega})}

\newcommand{\tpOmp}{(\vecp{t}, \vecp{\Omega})}

\newcommand{\tG}{(\vec{t}, \, G)}
\newcommand{\tpGp}{(\vecp{t}, \, G^\prime)}

\newcommand{\kmin}{\kappa^{\min}}
\newcommand{\kmax}{\kappa^{\max}}

\newcommand{\Nxik}{\indc^{^{N}}_{\vec{0}, k}}
\newcommand{\Nxitk}{\indc^{^{N}}_{\vec{t}, k}}

\newcommand{\Lxitk}{\indc^{^{L}}_{\vec{t}, k}}

\newcommand{\Lxik}{\indc^{^{L}}_{\vec{0}, k}}

\newcommand{\xit}{\indc_{\vec{t}}}
\newcommand{\xitp}{\indc_{\vecp{t}}}

\newcommand{\xitOm}{\indc_{\vec{t}, \vec{\Omega}}}

\newcommand{\xitpOmp}{\indc_{\vecp{t}, \vecp{\Omega}}}

\newcommand{\fC}{\mathfrak{C}}
\newcommand{\cG}{\mathcal{G}}

\newcommand{\cP}{\mathcal{P}}
\newcommand{\tv}[2] {\|#1 - #2\|_{\textnormal{TV}}}
\newcommand{\vsupp}{\textnormal{vsupp}}
\newcommand{\supp}{\textnormal{supp}}
\newcommand{\cC}{\mathcal{C}}
\newcommand{\cCp}{\cC^\prime}
\newcommand{\hcK}{\hat{\cK}}

\newcommand{\bdr}{\partial}
\newcommand{\bF}{\mathbb{F}}

\newcommand{\cF}{\mathcal{F}}
\newcommand{\cX}{\mathscr{X}}
\newcommand{\cV}{\mathcal{V}}
\newcommand{\cU}{\mathcal{U}}

\newcommand{\sB}{\mathscr{B}}

\newcommand{\sL}{\mathscr{L}}
\newcommand{\cK}{\mathcal{K}}
\newcommand{\cI}{\mathcal{I}}
\newcommand{\cA}{\mathcal{A}}
\newcommand{\Poi}{\textnormal{Poi}}
\newcommand{\Gau}{N}

\newcommand{\ExP}{\mathbb{E}}
\newcommand{\CiD}{\Rightarrow}
\newcommand{\Cech}{\v{C}ech }
\newcommand{\Rips}{Vietoris-Rips }
\newcommand{\EP}{Euler-Poincar\'{e} }

\newcommand{\LKC}{Lipschitz-Killing }

\newcommand{\Bnk}{\beta_{n, k}}
\newcommand{\Dnk}{D_{n, k}}
\newcommand{\Spnk}{S^\prime_{n, k}}
\newcommand{\cSnk}{\check{S}_{n, k}}

\newcommand{\cSk}{\check{S}_k}

\newcommand{\Nnk}{N_{n, k}}
\newcommand{\cNnk}{\check{N}_{n, k}}

\newcommand{\cNk}{\check{N}_k}
\newcommand{\Lnk}{L_{n, k}}
\newcommand{\cLnk}{\check{L}_{n, k}}

\newcommand{\cLk}{\check{L}_k}

\renewcommand{\vec}[1]{\bar{#1}}

\newcommand{\vecXtG}{\vec{X}_{\vec{t}, G}}
\newcommand{\vecXtpGp}{\vec{X}_{\vecp{t}, G^\prime}}

\newcommand{\GtOm}{G_{\vec{t}, \vec{\Omega}}}

\newcommand{\vecXtOm}{\vec{X}_{\vec{t}, \vec{\Omega}}}

\newcommand{\vecXtpOmp}{\vec{X}_{\vecp{t}, \vecp{\Omega}}}

\newcommand{\Snk}{S_{n, k}}
\newcommand{\Snku}{S_{n, k}(u)}
\newcommand{\lbnk}{\lambda_{n, k}}
\newcommand{\lbnku}{\lambda_{n, k}(u)}

\newcommand{\cN}{\mathcal{N}}
\newcommand{\Var}{\textnormal{Var}}

\newcommand{\DGFX}{\{X_{\vec{t}}\}_{\vec{t} \in \Grid^d}}
\newcommand{\tNeigh}{\mathfrak{N}}
\newcommand{\mN}{\tNeigh^{^{-}}} 
\newcommand{\pN}{\tNeigh^{^{+}}} 
\newcommand{\mv}{\vec{v}_{_-}}
\newcommand{\pv}{\vec{v}_{_+}}

\newcommand{\un}{u_n}
\newcommand{\Term}{T}

\newcommand{\remove}[1]{}
\endlocaldefs

\begin{document}

\begin{frontmatter}
\title{Betti Numbers of Gaussian Excursions in the Sparse Regime}
\runtitle{Topology of Gaussian Excursions}

\begin{aug}
\author{\fnms{Gugan} \snm{Thoppe}\thanksref{t1, m1} \ead[label=e1]{gugan.thoppe@gmail.com}}
\and
\author{\fnms{Sunder Ram} \snm{Krishnan}\thanksref{t1, m2}
\ead[label=e2]{eeksunderram@gmail.com}}

\thankstext{t1}{Both GT and SRK were previously supported by URSAT, European Research Council's Advanced Grant 320422. GT is now supported by grants NSF IIS-1546331, NSF DMS-1418261, and NSF DMS-1613261, while SRK is supported by the Viterbi Postdoctoral Fellowship. A portion of this work was done when GT was a postdoc at Technion.}
\runauthor{Thoppe and Krishnan}

\affiliation{Duke University\thanksmark{m1} and Technion--Israel Institute of Technology\thanksmark{m2}}

\address{Gugan Thoppe\\
Dept. of Statistical Science \\
Duke University \\
Durham, NC 27708 \\
USA \\
\printead{e1}}

\address{Sunder Ram Krishnan\\
Faculty of Electrical Engineering\\
Technion - Israel Institute of Technology \\
Haifa 32000 \\
Israel\\
\printead{e2}
}
\end{aug}

\begin{abstract}
Random field excursions is an increasingly vital topic within data analysis in medicine, cosmology, materials science, etc. This work is the first detailed study of their Betti numbers in the so-called `sparse' regime. Specifically, we consider a piecewise constant Gaussian field whose covariance function is positive and satisfies some local, boundedness, and decay rate conditions. We model its excursion set via a \Cech complex. For Betti numbers of this complex, we then prove various limit theorems as the window size and the excursion level together grow to infinity. Our results include asymptotic mean and variance estimates, a vanishing to non-vanishing phase transition with a precise estimate of the transition threshold, and a weak law in the non-vanishing regime. We further obtain a Poisson approximation and a central limit theorem close to the transition threshold. Our proofs combine extreme value theory and combinatorial topology tools.
\end{abstract}



\end{frontmatter}

\section{Introduction}
\label{sec:intro}




Key insights into the behaviour of a random field can be inferred from its excursion set, the sub-domain where the field value exceeds some level. Hence, functionals of excursion sets have generated considerable interest both in theory \cite{adler2009random, azais2009level} and in practice \cite{adler2017applications, marinucci2011random, torquato2013random}. In this paper, we provide limit theorems, together with rates of convergence, for  Betti numbers $\{\beta_k\}$ of Gaussian excursions in the so-called `sparse' regime. This work is the first study of the sparse regime behaviour of Betti number of random field excursions of any kind. Intuitively, for a topological space, $\beta_0$ is the number of components while $\beta_k,$ for $k \geq 1,$ is the number of $(k + 1)-$dimensional `holes'. A $(k + 1)-$dimensional hole is loosely the hollow region that opens up when the interior of a solid $(k + 1)-$dimensional object is removed. Specifically, in three dimensions, $\beta_1$ is the number of `tunnels' (regions through which one can poke one's hand), $\beta_2$ is the number of `voids' (regions that look like the interior of a tennis ball, doughnut, etc.), while other higher order Betti numbers are all zero. Thus, our results statistically quantify the topology of high-level Gaussian excursions.

The majority of recent work on Gaussian excursions concerns what is known as the `thermodynamic' regime. Even within these, the major focus of most papers (\cite{Estrade, kratz2016central, muller2017central}, etc.) has been on a different set of functionals called the \LKC curvatures (LKCs), which include as special cases the \EP characteristic (EPC) and the $d-$dimensional volume. LKCs do provide useful topological information, but not at the level of individual Betti numbers. In relation to Betti numbers, the only paper of which we are aware is \cite{Reddy2018}. There, a weak law and a (multivariate) CLT have been shown for a generic class of (quasi-) local statistics of spin models on Cayley graphs. By using a suitable random cubical complex to model Gaussian excursions on $\dGrid,$ one can apply these results to Betti numbers in the thermodynamic regime. As can be seen in Section 2.2.3 there, it is required that the underlying spin model be subcritical and the covariance function decay exponentially. We point out that, unlike this fast decay assumption, the results in \cite{Estrade, kratz2016central, muller2017central} and related papers require only that the covariance function be integrable (see also Remark~\ref{rem:tightphase}).

In \cite{adler2014existence, adler2017climbing}, existence probabilities of connecting paths and holes in high level excursions of Gaussian fields on the continuum have been studied. However, as noted in those papers, it is difficult to handle such questions in the continuum and those works obtained estimates for these probabilities only at the level of large deviations. Separately, in \cite{nazarov2016asymptotic, sarnak2016topologies}, existence results on topology of level sets of stationary Gaussian fields have been provided. It is not clear whether any of those results can easily be translated to excursion sets.

In contrast, and as highlighted previously, we study Gaussian excursions in the sparse regime (see Remark~\ref{rem:SparseRegime} for a formal description of sparse/thermodynamic regime). Our results are explicit and we obtain them under significantly weaker assumptions; in fact, in some cases, these assumptions are tight. We achieve these by focussing on Gaussian fields on discrete parameter spaces and using a novel approach of modelling their excursion sets via simplicial complexes (generalization of graphs). The latter enables use of suitable combinatorial approximators to study the Betti numbers of these excursions, a trick motivated by studies in the random simplicial complex literature \cite{linial2006homological, meshulam2009homological, Kahle2009, kahle2011random, kahle2013}. We then obtain precise distributional limit theorems for these approximators by employing the Stein-Chen approach \cite{holst1990poisson}, together with Slepian's lemma and Savage's multivariate Gaussian tail estimates \cite{savage1962mills}. We finally transfer these results to Betti numbers, building upon the topological ideas from \cite{kahle2011random, kahle2013}. A detailed description of all these notions is given in Section~\ref{sec:BgPrelim}.

\begin{figure}
\centering
\begin{minipage}{0.45\textwidth}
\centering
\includegraphics[width=\textwidth]{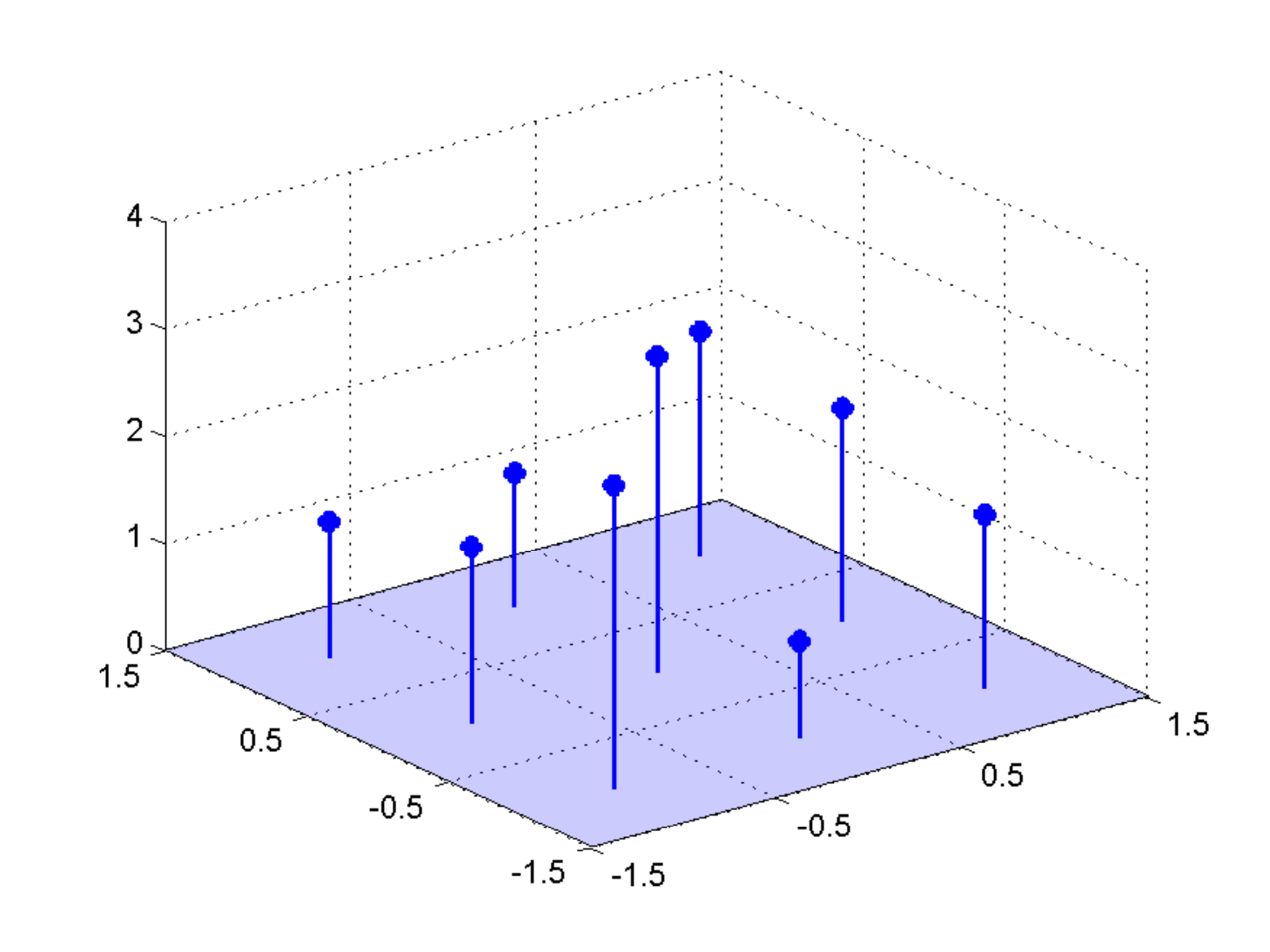} 
\end{minipage}\hfill
\begin{minipage}{0.45\textwidth}
\centering
\includegraphics[width=1.05\textwidth]{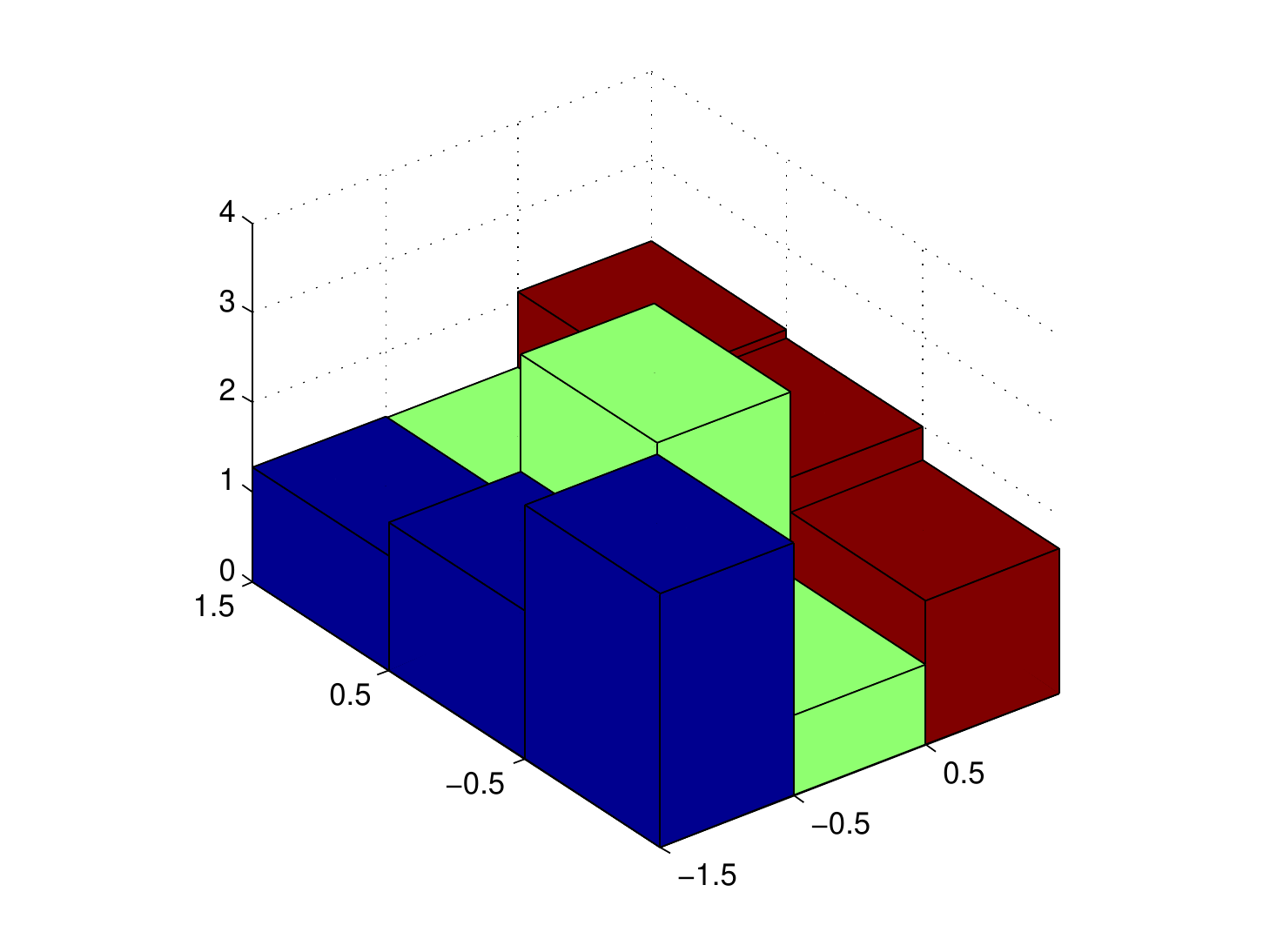} 
\end{minipage}
\caption{\label{fig:GaussianField} Possible realization of our Gaussian field $X$ on $\Grid^2$ and its piecewise constant extension to $\Real^2$ (colors are used only to provide a clearer illustration).}
\end{figure}

\begin{figure}
\centering
\begin{minipage}{0.45\textwidth}
\centering
\includegraphics[width=0.9\textwidth]{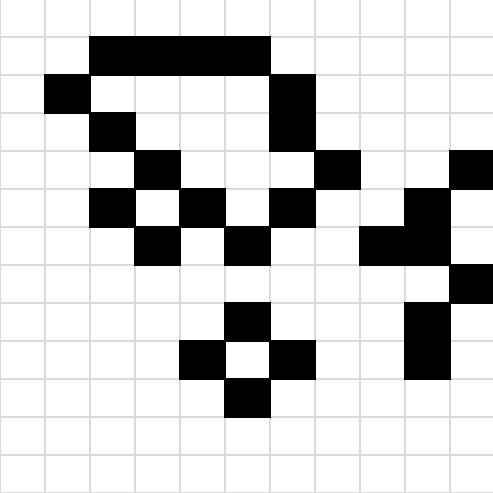} 
\end{minipage}\hfill
\begin{minipage}{0.45\textwidth}
\centering
\includegraphics[width=0.9\textwidth]{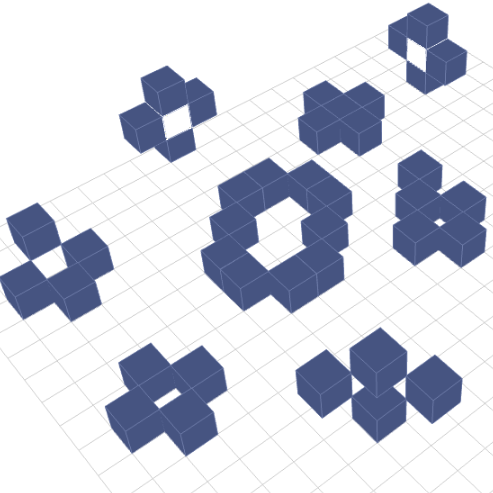} 
\end{minipage}
\caption{\label{fig:GaussianExcursions} Possible realizations of the excursion set $\cA(n; u)$ in two and three dimensions with non-trivial first Betti numbers(created using voxelbuilder).}
\end{figure}

Formally, we consider a piecewise constant random field $X : \dReal \to \Real$ with the below properties.
\begin{enumerate}
\item $\DGFX$ is a zero mean, stationary, $\|\cdot\|_1-$isotropic, discrete parameter Gaussian field with the covariance sequence  $\{\rho_q\}_{q \geq 0}.$ That is,
\begin{enumerate}
\item  $\ExP[X_{\vec{t}}] = 0,$ for all $\vec{t} \in \dGrid,$ and

\item $\cov[X_{\vec{t}}, X_{\vecp{t}}] = \rho_{_{\|\vec{t} - \vecp{t}\|_1}},$ for all $\vec{t}, \vecp{t} \in \dGrid.$
\end{enumerate}

\item For $\vec{t} \in \dReal \backslash \dGrid,$ $X_{\vec{t}}$ equals the value of the above discrete field at the lattice point closest to $\vec{t}$ in the $\|\cdot\|_\infty$ norm  (if there is more than one closest point, we pick the one that is the largest in lexicographical ordering).
\end{enumerate}
Its excursion set of interest to us is
\begin{equation}
\label{defn:ExcursionSet}
\cA(n; u) := \{\vec{t} \in [-n-1/2, n + 1/2)^d: X_{\vec{t}}  \geq u\}
\end{equation}
where $n \geq 0$ is the window size parameter and $u \in \pReal$ is the excursion level; here, $\pReal$ is the set of positive real numbers. See Figures~\ref{fig:GaussianField} and \ref{fig:GaussianExcursions} for an illustration.

To study the topology of $\cA(n; u),$ as mentioned above, we model it using a simplicial complex defined as follows. Let  $\Idx_n := \{\vec{t} \in \Grid^d: \|\vec{t}\|_\infty \leq n\}$ be the discrete window of side length $(2n + 1)$ and let $\sB_\infty(\vec{t}, r):= \{\vecp{t} \in \Real^d: \|\vec{t} - \vecp{t}\|_\infty \leq r\}.$

\begin{definition}
\label{Defn:Cech}
The random \Cech complex $\cK(n; u)$ on the excursion set $\cA(n; u)$ is the simplicial complex with vertex set $\cF_0:= \{\vec{t} \in \Idx_n:  X_{\vec{t}} \geq u\},$ and $\sigma \subset \cF_0$ is a face of $\cK(n; u)$ if
\[
\bigcap_{\vec{t} \in \sigma} \sB_\infty \left( \vec{t}, \tfrac{1}{2} \right) \neq \emptyset.
\]
\end{definition}

Observe that the vertex set of $\cK(n; u)$ is random, while its faces are decided using a deterministic rule based on the $\|\cdot\|_\infty-$distance between the vertex pairs. This places it in the family of random geometric complexes. It is called a \Cech complex since each of its face is chosen based on the mutual intersection of suitable balls centered at the vertices in that face.

An alternative way to model $\cA(n; u)$ could have been the complex $\hat{\cK}(n; u)$ defined as follows.
\begin{definition}
\label{Defn:Rips}
The random \Rips complex $\hat{\cK}(n; u)$ on the excursion set $\cA(n; u)$ is the simplicial complex with vertex set $\cF_0 := \{\vec{t} : \vec{t} \in \Idx_n, X_{\vec{t}} \geq u\},$ and $\sigma \subset \cF_0$ is a face of $\hat{\cK}(n; u)$ if
\[
\sB_\infty \left( \vec{t}_i, \tfrac{1}{2} \right) \cap \sB_\infty \left( \vec{t}_j, \tfrac{1}{2} \right)\neq \emptyset
\]
for every pair $\vec{t}_i, \vec{t}_j \in \sigma.$
\end{definition}

This is also a geometric complex. It is called the \Rips complex because its faces are decided based on pairwise intersections. Usually, a \Cech and \Rips complex are different as the rules for defining them differ. However, in our setup, using the fact that they are defined on a lattice, it turns out that $\cK(n; u) = \hat{\cK}(n; u);$ see Proposition~\ref{prop:EqCechRips} for details. Thus, according to need, we shall view $\cK(n; u)$ sometimes as a \Cech and at other times as a \Rips complex. Taking one of the views, let $\beta_{n, k}(u)$ be the $k-$th Betti number of $\cK(n; u);$ Figures~\ref{fig:GaussianExcursions} and ~\ref{fig:Betti_Illus} show a few representative examples. With details given in Section~\ref{sec:BgPrelim}, we point out that the coefficients to define these Betti numbers can be from either a field or $\mathbb{Z}.$
\begin{figure}
\centering
\begin{minipage}{0.45\textwidth}
\centering
\includegraphics[width=0.75\textwidth]{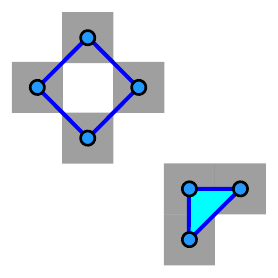} 
\end{minipage}\hfill\vline\hfill
\begin{minipage}{0.45\textwidth}
\centering
\includegraphics[width=0.75\textwidth]{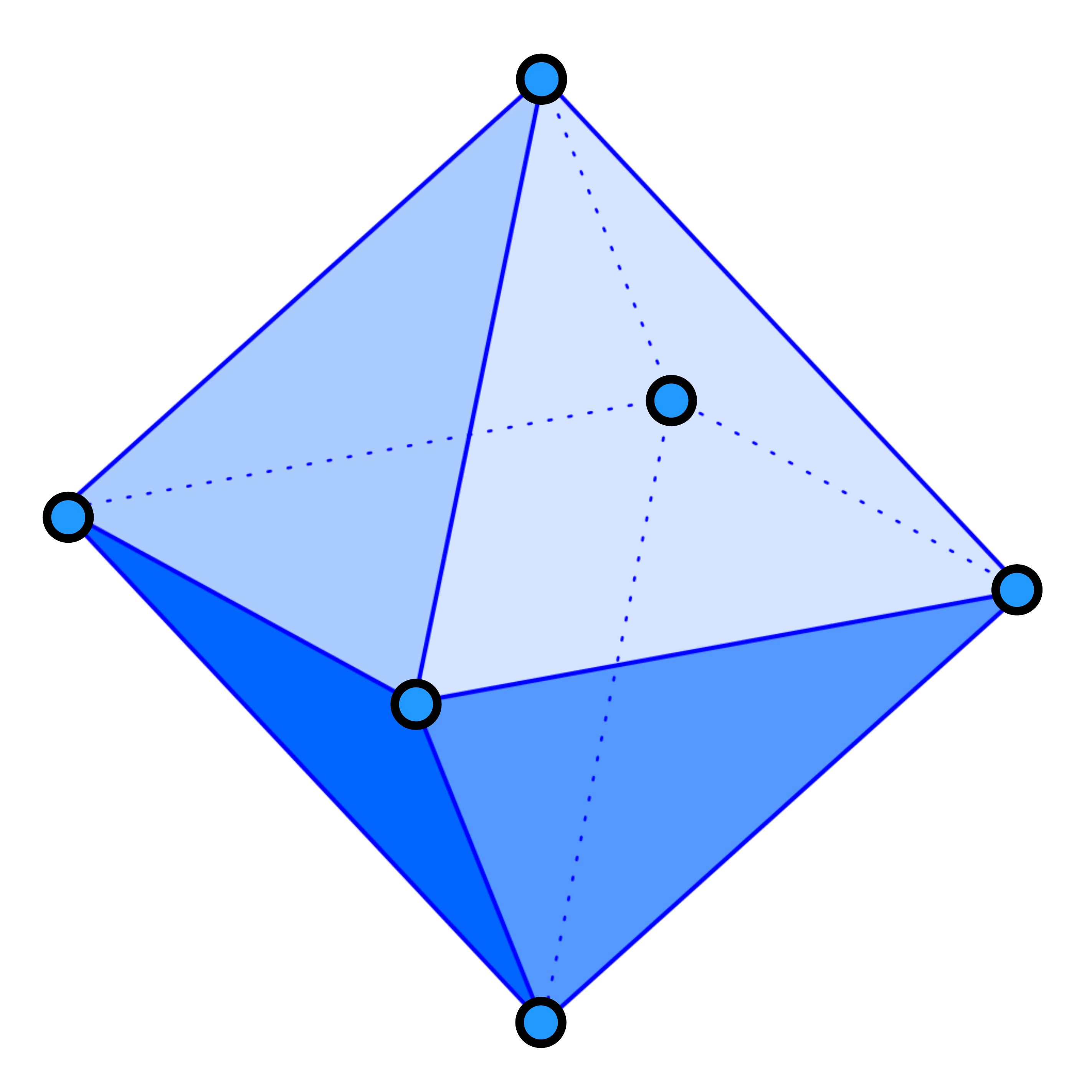}  
\end{minipage}
\caption{\label{fig:Betti_Illus} Possible subcomplexes of $\cK(n; u).$ The left image shows how a part of our excursion set may appear (grey color) when $d = 2$   and the simplicial complex associated with it; because of the top left formation, the first Betti number of this subcomplex is $1.$ The right image shows a subcomplex, the second Betti number of which equals $1;$ this may appear only when $d \geq 3.$ (Source: Wikipedia).}
\end{figure}

\begin{table}[t]
\begin{tabular}{c | l | c | c}
\hline

Id & Assumption & Applicable when & Type \\

\hline

\setcounter{cst@counter@Asm}{-1}

& & & \\[-2ex]

$\Cl[Asm]{a:r0Asm}$ & $\rho_0 = 1$ & & Local \\[2ex]
$\Cl[Asm]{a:r1Asm}$ & $0 \leq \rho_1 < \rho_0$ & & Local \\[2ex]
$\Cl[Asm]{a:rCAsm}$ & $
\begin{cases}
0 \leq \sup\limits_{q \geq 2}\rho_q \leq \rho_1\\[2ex]
0 \leq \sup\limits_{q \geq 4}\rho_q \leq \rho_3 < \rho_2 < \rho_1
\end{cases}
$
&
$
\parbox{7em}{
\vspace{-1.25ex}
$k = 0$\\[2.75ex]
$1 \leq  k < d$
}
$ & Boundedness \\[6ex]
$\Cl[Asm]{a:ComC}(k)$ & $
\begin{cases}
1 = 1\\[2ex]
1 + (2k + 1) \rho_2 > 2(k + 1) \rho_1
\end{cases}
$
&
$
\parbox{7em}{
\vspace{-1.25ex}
$k = 0$\\[2.75ex]
$1 \leq  k < d$
}
$ & Local \\[6ex]
$\Cl[Asm]{a:PoiG}$ & $\lim_{q \to \infty} \rho_q \log q = 0$ & & Decay Rate \\[2ex]
$\Cl[Asm]{a:CLTG}$ & $\sum_{q = 0}^{\infty} q^{d - 1} \rho_q < \infty$ & & Decay Rate \\[1.5ex]
\hline
\end{tabular}\\[2ex]
\caption{\label{tab:Assumptions} Assumptions for different results.}
\end{table}

Our key results can now be summarized as follows. The different assumptions are listed in Table~\ref{tab:Assumptions}; these are in addition to the conditions on $\{\rho_q\}$ imposed by the fact that the covariance function of $X$ must be positive definite. For $k$ such that $1 \leq k < d,$ define the constants $\tau_k,$ $a_k,$ and $b_k$ using
\begin{equation}
\label{Defn:Tauk}
\tau_{k} := \frac{ \tbinom{d}{k + 1}}{(2\pi)^{k + 1}}  \frac{[1 + (2k + 1) \rho_2]^{2k + \tfrac{1}{2}}}{[1 - \rho_2]^{k + \tfrac{1}{2}}}, \; \; a_k := \frac{k + 1}{1 + (2k + 1) \rho_2}, \text{ and } b_k := k + 1.
\end{equation}
Separately, let
\begin{equation}
\label{Defn:Tau0}
\tau_0 := \frac{1}{\sqrt{2\pi}}, \; \; a_0 := \tfrac{1}{2}, \text{ and } b_0 := \tfrac{1}{2}.
\end{equation}
For $n \geq 0,$ $k$ such that $0 \leq k < d,$ and $u \in \pReal,$ define $\lbnk(u)$ using
\begin{equation}
\label{Defn:Lambdanku}
\lambda_{n, k}(u) := \tau_{k} \;   (2n + 1)^d u^{-2b_k} \exp[-a_k u^2].
\end{equation}
Finally, let $\pSeq$ be the set of sequences in $\pReal.$

Our first result gives a precise asymptotic formula for $\ExP[\beta_{n, k}(u)].$ Here, and elsewhere, $O(\cdot)$ denotes the Big-O notation and it characterizes a function's behaviour as $n \to \infty.$
\begin{theorem}[Asymptotic Mean of $\Bnk(\un)$]
\label{thm:ExPMainResult}
Let $k$ be such that $0 \leq k < d$ and $\{\un\} \in \pSeq$ be such that $\lim_{n \to \infty} \un = \infty.$ Suppose $\Cr{a:r0Asm}, \Cr{a:r1Asm}, \Cr{a:rCAsm},$ and $\Cr{a:ComC}(k)$ hold. Then,
\[
\left|\frac{\ExP[\beta_{n, k}(\un)]}{\lambda_{n, k}(\un)} - 1\right| = O\left(\frac{1}{\un^2} + \frac{1}{n}\right).
\]
The constants involved in the $O$ notation depend on $k.$
\end{theorem}

From \eqref{Defn:Lambdanku}, note that the limit of the sequence $\{\lbnk(\un)\}$ can be either zero, a constant, or infinity, depending on the growth rate of  $\{\un\}.$ An alternative way to view this is to let
\begin{equation}
\label{eqn:unForm}
\un = \sqrt{\frac{1}{a_k} \left[\log[\tau_k (2n + 1)^d] - b_k \log\left[\frac{\log[\tau_k (2n + 1)^d]}{a_k}\right] + \nu_n \right]},
\end{equation}
where $\{\nu_n\}$ is any real sequence that ensures $\{\un\} \in \pSeq$ and $\lim_{n \to \infty} \un = \infty.$ Then, $\{\lambda_{n, k}(\un)\}$ has the above three limits precisely when $\lim_{n \to \infty}\nu_n$ is either $\infty,$ some constant, or $-\infty.$ Our next result gives the asymptotic behaviour of the $k-$th Betti number in these three scenarios.

\begin{theorem}[Asymptotic Distributional Behaviour of $\Bnk(\un)$]
\label{thm:BehMainResult}
Let $k$ be such that $0 \leq k < d$ and $\{\un\} \in \pSeq$ be such that $\lim_{n \to \infty} \un = \infty.$ In the different regimes dictated by the limit of  $\{\nu_n\},$ or equivalently of $\{\lbnk(\un)\},$ the $k-$th Betti number behaves as follows.
\begin{enumerate}
\item \label{itm:Vanishing}\textit{Vanishing regime $(\nu_n \to \infty$ or $\lambda_{n, k}(\un) \to 0$):} If $\Cr{a:r0Asm}, \Cr{a:r1Asm}, \Cr{a:rCAsm},$ and $\Cr{a:ComC}(k)$ hold, then
\[
1 - \Pr\{\beta_{n, k}(\un) = 0\} = O\left(e^{-\nu_n} \left[\frac{\log n} {\log n+ \nu_n}\right]^{b_k} \right).
\]

\item \label{itm:Poisson} \textit{Poisson regime $(\nu_n \to \log[\tfrac{1}{\lambda}]$ or $\lambda_{n, k}(\un) \to \lambda,$ for some $\lambda \in \pReal$):} If $\Cr{a:r0Asm}, \Cr{a:r1Asm}, \Cr{a:rCAsm}, \Cr{a:ComC}(k),$ and $ \Cr{a:PoiG}$ hold,  then $\Bnk(\un) \CiD \Poi(\lambda);$ in particular, we have
\[
\tv{\Bnk(\un)}{\Poi(\lambda)} = O\left(\frac{\log n}{n} \sum_{q = 1}^{2 dn} \rho_q + \frac{\log \log n}{\log n} + |e^{-\nu_n} - \lambda|\right).
\]

\item \label{itm:NonVanishing} \textit{Non-vanishing regime ($\nu_n \to - \infty$ or $\lambda_{n, k}(\un) \to \infty$):}  If  $\Cr{a:r0Asm}, \Cr{a:r1Asm}, \Cr{a:rCAsm}, \Cr{a:ComC}(k),$ and $ \Cr{a:CLTG}$ hold, then
\begin{equation}
\label{eqn:BnkVarEst}
\lim_{n \to \infty} \frac{\Var[\Bnk(\un)]}{\lbnk^2(\un)} = 0,
\end{equation}
\begin{equation}
\label{eqn:NonVanishing}
\lim\limits_{n \to \infty} \Pr\{\beta_{n, k}(\un) = 0\}  = 0,
\end{equation}
and
\begin{equation}
\label{eqn:WeakLaw}
\frac{\beta_{n, k}(\un) - \ExP[\beta_{n, k}(\un)]}{\lambda_{n, k}(\un)} \overset{i.p.}{\longrightarrow} 0.
\end{equation}
Additionally, there exists a constant $\conNu_k > 0,$ depending on $k,$ such that, if $C \in (0, \conNu_k)$ and $\nu_n \geq - C \log (2n + 1)^d,$ then
\begin{equation}
\label{eqn:BnkCLT}
\dfrac{\beta_{n, k}(\un) - \ExP[\beta_{n, k}(\un)]}{\sqrt{\lambda_{n, k}(\un)}} \CiD \Gau(0, 1),
\end{equation}
where $\Gau$ denotes the Gaussian distribution. The explicit formula for $\conNu_k$ is given in \eqref{Defn:conNu}.
\end{enumerate}
\end{theorem}
This result proves that the $k-$th Betti number undergoes a vanishing to non-vanishing phase transition as its asymptotic mean changes from zero to infinity. It also shows a weak law in the non-vanishing regime. In addition, close to the transition threshold, the result gives a Poisson approximation and a CLT. Loose rates of convergence for \eqref{eqn:BnkVarEst} and \eqref{eqn:BnkCLT} are given in their proofs.

We now provide some remarks on our assumptions, model setup,  and results.
{\renewcommand*\theenumi{$\pmb{\mathcal{R}_{\arabic{enumi}}}$}
\begin{enumerate}

\item \label{rem:L1Iso} We assume $\|\cdot\|_1-$isotropy in our model mainly in order to simplify our estimates and computations. With other forms of isotropy, we strongly believe that results with a similar flavor to ours should hold, but will require some more involved calculations; see Section~\ref{sec:Discussion}.

\item For all our results, we need the covariance function to be positive. This is because it is unclear at present how to apply our Stein-Chen approach when negative covariances are involved; in particular, we are then unable to establish a result such as Theorem~\ref{thm:ProbSpaceExistence}. For the specific case of $\beta_0,$ however, this assumption can be relaxed; see \cite{holst1990poisson} for the outline.

\item \label{rem:LocRest} We also need that $\Cr{a:ComC}(k)$ holds. For Gaussian fields on the continuum with covariance functions given by $\exp[-\|\vec{t} - \vecp{t}\|_2^q],$ $\|\vec{t} - \vecp{t}\|_2^{-q},$ etc. with $q > 0,$ this condition places a lower bound on the distance at which these fields need to be sampled so that our results hold. However, note that this condition is natural in ARMA models. It may be possible to eliminate this assumption, but the computations would then become involved. See Section~\ref{sec:Discussion} for details.

\item For the vanishing and Poisson regime results, we require the covariance function to also satisfy $\Cr{a:PoiG},$ the so-called `Berman' condition. This condition is known to be tight for Poisson approximation; see \cite[Remark 1]{holst1990poisson} for example. It is worth noting that this condition holds even for Gaussian fields that are long-range dependent and also for Gaussian free fields on $\dGrid$ with $d \geq 3$ \cite{chiarini2015extremes}.

\item \label{rem:tightphase} The non-vanishing regime results including the CLT need the stronger $\Cr{a:CLTG}$ condition, which basically means that covariances should be summable in the domain dimension $d.$ This is a common assumption in the Gaussian excursions literature \cite{Estrade, kratz2001central, kratz2016central, berman1989central}. However, it is unclear whether it is tight.

\item \label{rem:SparseRegime} In all the three regimes determined by the limit of $\nu_n,$ the average vertex degree (i.e., $\ExP[\mbox{number of edges}]/\ExP[\mbox{number of vertices}]$) of our \Cech complex asymptotically vanishes. In random graphs terminology, this places all our key results in the `sparse' regime. The asymptotics in \cite{Estrade, kratz2016central, muller2017central, Reddy2018},  from the perspective of our setup, loosely translates to keeping $u$ fixed and letting only $n$ increase to infinity. It is easy to see that the average vertex degree would asymptotically then be a constant. This positions the results of these studies concerning LKCs and Betti numbers in what is usually referred to as the thermodynamic regime.

\item \label{rem:TransitionThreshold} From Theorem~\ref{thm:BehMainResult} and \eqref{eqn:unForm}, in terms of $\un^2,$
\begin{equation}
\label{eqn:TransitionThreshold}
\frac{1}{a_k} \left[\log[\tau_k (2n + 1)^d] - b_k \log\left[\tfrac{\log[\tau_k(2n + 1)^d]}{a_k}\right]\right]
\end{equation}
is the transition threshold for non-triviality of the $k-$th Betti number. The lower order $\log\log$ term in the above expression is due to the sharp multivariate Gaussian tail estimates given in \cite{savage1962mills}. By using only the large deviation approach, such precision cannot be obtained.

\item In \cite{Werman2016} and \cite{hiraoka2016limit}, studies of random cubical complexes have been carried out. In \cite{hiraoka2016limit}, a strong law and a CLT for Betti numbers have been derived; while the strong law holds for generic distributions, the CLT needs independence. It may be possible to translate these results to Gaussian excursions, but then they would again apply in the thermodynamic regime. As opposed to cubical complexes, we model $\cA(n; u)$ here using simplicial complexes, since a richer set of results are available for the latter class.

\item From the point of view of the random simplicial complexes literature, our work extends recent advances. We introduce a new random simplicial complex model. While being in the family of geometric complexes, it facilitates the study of random field excursions. However, the in-built dependence amongst the faces in our model contrasts it with most existing random complexes such as Linial-Meshulam complexes \cite{linial2006homological, meshulam2009homological}, clique complexes \cite{Kahle2009, kahle2014sharp}, and random geometric complexes \cite{bobrowski2014topology}. In all the latter models, independence plays a crucial role in both their definitions and their study.

\item Despite the dependence, however, our results share several similarities with the sparse regime phenomena in random geometric complexes studied in \cite{kahle2011random, kahle2013}; the presence of faces in their model is decided based on the Euclidean proximity of the vertices that are generated as an IID sequence, while the asymptotics is in terms of the coupling between the number of vertices and the proximity thresholds. As in these works, the behaviour of our Betti numbers is dictated by the simplest, minimal subcomplexes that generate them. Furthermore, for each Betti number, we sequentially observe a vanishing, Poisson, and non-vanishing behaviour. We also see that the lower order Betti numbers appear earlier than the higher order ones as the growth rate of the excursion levels is lowered. Mathematically, we mean that, as $k$ increases from $0$ to $d - 1,$ the phase transition threshold decreases; this follows from Remark~\ref{rem:TransitionThreshold} and the fact that $a_k$ monotonically increases with $k.$ In fact, a consequence of Theorem~\ref{thm:ExPMainResult} is that
\[
\lim_{n \to \infty} \frac{\ExP[\beta_{n, k_1}(\un)]}{\ExP[\beta_{n, k_2}(\un)]} \to \infty,
\]
whenever $k_1 < k_2$ and $\{\un\}$ is such that $\lim_{n \to \infty} \un = \infty.$

\item \label{rem:Ind} If we assume independence in our model ($\rho_q = 0 \; \forall q \geq 1$), our calculations simplify significantly. However, the basic nature of our results will remain the same, except that the associated constants, and hence regimes, will be different. For a more detailed discussion, see Section \ref{sec:Discussion}.

\item \label{rem:VarEstBnk} As compared to the estimate for variance of Betti numbers in Theorem~\ref{thm:BehMainResult}, we conjecture that
\[
\lim_{n \to \infty} \frac{\Var[\Bnk(\un)]}{\lbnk(\un)} = 1.
\]
The evidence for this follows from Lemma~\ref{lem:VarianceEstimates}.\ref{itm:SnkVarEst} and Remark~\ref{rem:VarNL}.

\item \label{rem:CLTlimit} In our CLT result, the additional requirement that $\nu_n \geq - C \log (2n + 1)^d$ restricts the validity of the theorem to a regime close to the phase transition threshold. Outside this regime, the CLT for Betti numbers does not follow from that of the approximators that we use throughout this work. To obtain better approximators necessitates, as of now, an equivalent of Theorem~\ref{thm:ProbSpaceExistence} when negative covariances are also involved. But, presently, it is unclear how to obtain such a result. Some more details about this are given in Section \ref{sec:Discussion}.

\end{enumerate}
}

We end this section by describing how the rest of the paper is structured.

\emph{Structure of the Paper}: The following section provides all the requisite background material concerning the topological and probabilistic aspects of this work. In Section~\ref{sec:ProofOutline}, we outline our proofs for the key results. Specifically, we first state all our major intermediate results and then, assuming them to true, prove our key results. The topological portions of these intermediate results are proved in Section~\ref{sec:BettiBounds}, while the remaining ones are proved in Sections~\ref{sec:IntResults},~\ref{sec:ConclResults}, and the Appendix. We end with a discussion on future directions in Section~\ref{sec:Discussion}.

\section{Background}
\label{sec:BgPrelim}

This section begins with a brief overview of relevant notions from simplicial homology with a focus on Betti numbers; this is based on \cite{Kahle2009} and the references therein. After that, we give the description of a non-trivial cycle and, in the context of a special simplicial complex called the clique complex, recall some useful results concerning it from \cite{Kahle2009, kahle2011random}. These results are crucially used in Section~\ref{sec:BettiBounds} later to obtain the approximators for our Betti numbers, mentioned before. Following all of that, we discuss few pertinent results from probability theory. First, we give Savage's multivariate Gaussian tail estimates from \cite{savage1962mills}. Here, we also provide the Savage condition under which this bound holds. We then describe Slepian's lemma which relates tail probabilities to covariance relations. We finally state a special case of the Stein-Chen method, which is the same one that was used in \cite{holst1990poisson}. This method gives a bound on the total variation distance between a sum of indicator random variables and a Poisson random variable having the same mean.






\subsection{Topological Background}

A key object of study across this work is a simplicial complex---a generalization of a graph to higher dimensions. Specifically, an (abstract) simplicial complex $\cK$ on a vertex set $V$ is a collection of non-empty subsets of $V$ such that for any $\sigma \in \cK,$ if $\sigma^\prime \neq \emptyset$ and $\sigma^\prime \subset \sigma,$ then $\sigma^\prime \in \cK$ as well. That is, $\cK$ is closed under the subset operation. The elements of $\cK$ are called faces and the dimension of a face $\sigma$ is $| \sigma | - 1,$ where $|\cdot|$ denotes cardinality. The dimension of $\cK$ itself is the maximum over the dimension of all its faces. The $k-$skeleton of $\cK$ is the simplicial complex made up of all the faces of $\cK$ with dimension $k$ or less.

Given a simplicial complex $\cK,$ one way to study its topology (or shape) is via its Betti numbers $\{\beta_k(\cK)\}_{k \geq 0}.$ These are described next, first intuitively and then formally.

\textit{Intuitive Description of Betti Numbers:} Imagine the $k-$dimensional faces of $\cK,$ or $k-$faces in short, to be solid $k-$dimensional objects. Then, a $k-$cycle in $\cK$ is a collection of its $k-$faces whose union is `topologically equivalent to' the boundary of a solid $(k + 1)-$dimensional object. If it is not the boundary of any subset of $(k + 1)-$faces in $\cK,$ then that $k-$cycle, in fact, represents a $(k + 1)-$dimensional hole. Finally, $\beta_0(\cK)$ is one more than the number of `independent' $1-$holes in $\cK,$ while, for $k \geq 1,$ $\beta_k(\cK)$ is exactly the number of `independent' $(k + 1)-$dimensional holes. In this sense, the simplicial complex in the left image of Figure~\ref{fig:Betti_Illus} has two $1-$cycles, but only the top left cycle represents a $2-$hole. This description extends the intuitive picture given at the outset.

\textit{Formal Description of Betti Numbers:} Let $\sigma$ be a $k-$face made up of the vertices $v_0, \ldots, v_k.$ An \emph{orientation} of $\sigma$ is an ordering of its vertices and is denoted by $(v_0, \ldots, v_k ).$ Two orderings induce the same orientation if and only if they differ by an even permutation of the vertices. We shall assume henceforth that each face in $\cK$ is assigned a specific orientation, i.e., ordering. Let $\bF$ be $\Grid$ or some field. Then, a simplicial $k$-chain is a formal sum of oriented $k-$faces, i.e., $\sum_i c_i \sigma_i,$ with $c_i \in \bF.$ The $k-$chain group $\fC_k(\cK)$ is the free Abelian group generated by all $k$-chains, i.e.,
\[
\fC_k(\cK) := \left\{ \sum_ic_i \sigma_i : c_i \in \bF, \sigma_i \in \cF_k(\cK) \right\},
\]
where $\cF_k(\cK)$ is the set of all $k-$faces in $\cK.$ Clearly, if $\bF = \Grid,$ then $\fC_k(\cK)$ is a $\Grid$-module, and if $\bF$ is a field, then $\fC_k(\cK)$ is a $\bF$-vector space. Separately, set $\fC_{-1} = \bF.$ Now, for $k \geq 0,$ define the boundary operator $\bdr_k : \fC_k \to \fC_{k - 1}$ first on each $k-$simplex using
\[
\bdr_k\left( (v_0, \ldots, v_k) \right) = \sum_{i=0}^k(-1)^i (v_0,\ldots,\hat{v}_i,\ldots, v_k )
\]
and then extend it linearly on $\fC_k.$ Here, $\hat{v}_i$ implies that the vertex $v_i$ is to be omitted. It is easy see that  $\bdr_{k - 1} \circ \bdr_k = 0$ for all $k \geq 1,$ i.e., boundary of a boundary is zero. The $k-$th {\em boundary space}, denoted by $\bGrp_k,$ is the image of  $\bdr_{k + 1}$ and the $k-$th {\em cycle space} $\cGrp_k$ is the kernel of $\bdr_k$. The elements of $\cGrp_k$ are called {\em $k-$cycles}, while the elements of $\bGrp_k$ are called {\em $k-$boundaries}. The $k$-th homology group $\hGrp_k$ is defined to be the quotient group $\hGrp_k = \cGrp_k / \bGrp_k.$ Clearly, $\hGrp_k$ is also a $\bF$-module or a $\bF$-vector space depending on whether $\bF = \Grid$ or $\bF$ is a field. Finally, $\beta_0(\cK) = \rk(\hGrp_0) + 1,$ while $\beta_k(\cK) = \rk(\hGrp_k)$ when $k \geq 1.$

We now give a couple of definitions including that of a non-trivial cycle (NTC) and then state two of its useful properties from \cite{Kahle2009, kahle2011random} in the context of what is known as a clique complex.

Continuing with the above notions, for a chain $\gamma \in \fC_k,$  let $[\gamma] := \{\gammap \in \fC_k: \gammap - \gamma \in \bGrp_k\}$ be its equivalence class with respect to $\bGrp_k.$ Then, it is easy to see that $\hGrp_k = \{[\gamma] : \gamma \in \cGrp_k\}.$

For $\sigma \in \cK,$ let $\vsupp(\sigma)$ be its vertex support. For $\gamma = \sum_i c_i \sigma_i,$ let $\vsupp(\gamma) := \bigcup_{i: c_i \neq 0} \vsupp(\sigma_i).$ A chain $\gamma$ is a $k-$NTC if $\gamma \in \cGrp_k \setminus \bGrp_k,$ and it has minimal vertex support if $|\vsupp(\gamma)| \leq |\vsupp(\gammap)|$ for each $\gammap \in [\gamma].$ Clearly, $\beta_0(\cK)$ is one more than the maximal number of independent $0-$NTCs, while, for $k \geq 1,$ $\beta_k(\cK)$ is the maximal number of independent $k-$NTCs. Based on this, given a set $\{\cC\}$ of isolated induced subcomplexes of $\cK,$ one can show that $\beta_k(\cK) = \sum_{\cC} \beta_k(\cC).$

Separately, for a vertex $v \in \cK,$ its link $\lk(v) := \{\sigma \in \cK: v \notin \sigma, \text{ but } \sigma \cup \{v\} \in \cK\},$ while its star, denoted $\st(v),$ is the smallest simplicial complex containing $\{\sigma \in \cK: v \in \sigma\}.$ Clearly, both the link and star of a vertex are themselves simplicial complexes.

The $(k + 1)-$dimensional cross-polytope is the convex hull of the $2k + 2$ points $\{\pm e_1, \ldots, \pm e_{k + 1}\},$ where $\vec{e}_1, \vec{e}_2$ etc. are the standard basis vectors. The boundary of this polytope can be represented using a $k$-dimensional simplicial complex. Let $O_k$ denote the $1-$skeleton of this complex. This is the graph on the above $2k + 2$ points where an edge is present between a pair if and only if the $\|\cdot\|_\infty-$distance between them is $1.$

Given a graph $G \equiv (V, E),$  a set of vertices $v_1, \ldots, v_k \in V$ is said to form a clique if $\{v_i, v_j\} \in E$ for all $v_i, v_j.$ The associated clique complex $\cX(G)$ is the simplicial complex made up of all the subsets of $V$ that form a clique in $G.$

\begin{lemma}
\label{lem:CardNTC}
\cite[Lemma 3.4]{kahle2011random}
Let $\cX(G)$ be the clique complex associated with a graph $G.$ Let $k \geq 0$ and $\gamma$ be a $k-$NTC of $\cX(G).$ Then, $|\vsupp(\gamma)| \geq 2(k + 1).$ If $\vsupp(\gamma) = 2(k + 1),$ then the $1-$skeleton of the induced subcomplex of $\cX(G)$ restricted to $\vsupp(\gamma)$ is isomorphic to $O_k.$
\end{lemma}

\begin{lemma}
\label{lem:RelLinkNTC}
\cite[Lemma 5.2]{Kahle2009}
Let $\cK$ be a simplicial complex. For $k \geq 1$ and $F$ being some set of $k-$faces, let $\gamma = \sum_{\sigma \in F} c_\sigma \sigma$ be a $k-$NTC in $\cK$ with minimal vertex support and $c_\sigma \neq 0$ for all $\sigma \in F.$ Then, for any $v \in \vsupp(\gamma),$ $\gamma \cap \lk(v) := \sum_{\sigma \in F \cap \st(v)} c_\sigma  \bdr_k (\sigma)$ is a $(k - 1)-$NTC in $\lk(v) \cap \{\sigma \in \cK: \sigma \subseteq \vsupp(\gamma)\}.$
\end{lemma}

\begin{remark}
\label{rem:RelLinkNTC}
Consider the setup as in Lemma~\ref{lem:RelLinkNTC}. For  $v \in \vsupp(\gamma),$ define $\gamma \cap \st(v)$ to be $\sum_{\sigma \in F : v \in F} c_\sigma \sigma.$ Then, it is easy to see that $\gamma \cap \lk(v) = \bdr_k(\gamma \cap \st(v)).$ Also, since $\gamma$ is a $k-$cycle, it follows that $v \notin \vsupp(\gamma \cap \lk(v)).$

\end{remark}

\subsection{Probabilistic Background}
%



Another key element across our computations is the tail probability of  a multivariate Gaussian random vector. A tight estimate for this has been given in \cite{savage1962mills}. We state this result below after introducing some relevant notations.

Henceforth, we  use the bar notation for vectors such as $\vec{Y}, \vec{t},$ etc. To refer to their $l$-th coordinate we use $\vec{Y}(l), \vec{t}(l),$ etc. All our vectors are row vectors and we denote their transpose using {\scriptsize $\tr.$} All vector inequalities mean that they hold coordinate wise. Specifically, for any random vector $\vec{X} \in \Real^i$ and $\vec{u} \in \Real^i,$ by $\{\vec{X} \geq \vec{u}\},$ we mean the event $\{\vec{X}(1) \geq \vec{u}(1), \ldots, \vec{X}(i) \geq \vec{u}(i)\}.$  The determinant of a matrix $M$ is denoted by $|M|,$ while $|r|$ denotes the absolute value for any $r \in \Real.$ Lastly, we use $\indc$ to denote both the indicator random variable as well as the indicator function.

\begin{lemma} \cite[(I), (II)]{savage1962mills}
\label{lem:GaussianTailBounds}
Let $i \geq 1$ and let $\vec{Y} \in \Real^i$ be a zero mean multivariate Gaussian random vector with positive definite covariance matrix $M.$ Then, for each $\vec{u} \in \Real^i$ such that $\vec{\Delta} := \vec{u} M^{-1} > 0$ holds, the so called `Savage condition', we have
\begin{equation}
\label{eqn:GaussianTailEst}
1 -  \frac{1}{2}\sum_{j, \ell = 1}^{i} \frac{M_{j\ell} [1 + \delta_{j \ell}]}{\vec{\Delta}(j) \; \vec{\Delta}(\ell) }  \leq \dfrac{\Pr\{\vec{Y} \geq \vec{u}\}}{(2 \pi)^{- i /2}  |M|^{-1/2} \exp[-[\vec{u} M^{-1} \vec{u}^\tr]/2]\left[\prod_{j = 1}^{i} \vec{\Delta}(j) \right]^{-1}  } \leq 1,
\end{equation}
where $M_{j \ell}$ is $j \ell-$th entry of $M^{-1},$ and $\delta_{j \ell} = 1$ if $j = \ell,$ and $0$ otherwise.
\end{lemma}

To get elegant closed form expressions for tail probabilities using the above result, it is important that the covariance matrix be `nice'. Often, this will not be the case in our computations. To deal with the same, we shall be using the following comparison inequality.

\begin{lemma} [Slepian's lemma]
\label{lem:SlepainInequality}
Let $i \geq 0.$ Let $\vec{Y}, \vec{Z} \in \Real^i$ be multivariate Gaussian random vectors such that $\ExP[\vec{Y}] = \ExP[\vec{Z}],$ $\Var[\vec{Y}(j)] = \Var[\vec{Z}(j)]$ for all $ j,$ and $\cov[\vec{Y}(j), \vec{Y}(\ell)] \geq \cov[\vec{Z}(j), \vec{Z}(\ell)]$ for all $j, \ell$ with $j \neq \ell.$ Then, for any $\vec{u} \in \Real^i,$ we have $\Pr\{\vec{Y} \geq \vec{u}\} \geq \Pr\{\vec{Z} \geq \vec{u}\}.$
\end{lemma}

The proof of this result follows from \cite[Theorem 2.3]{adler1990introduction} by substituting $h(\vec{x}(1), \vec{x}(2), \ldots, \vec{x}(i)) = \prod_{j = 1}^{i} \indc[\vec{x}(j) \geq \vec{u}(j)],$ where $h$ is as defined there. The only catch is that this function is not smooth. But this can be easily overcome by approximating the indicators with smooth increasing functions.

We prove both the Poisson as well the CLT portions of Theorem~\ref{thm:BehMainResult} using the Stein-Chen method. As is the case here, this method is useful when we are dealing with limit distributions of a sum of dependent indicator random variables. To apply this method, a bound on the total variation distance between the above mentioned sum and a Poisson random variable having the same mean is required. We give one such bound in Theorem~\ref{thm:SuffCondPoissonConv}. In order to be able to apply this bound to excursions of Gaussian random vectors, an additional technical result is required. This is given in Theorem~\ref{thm:ProbSpaceExistence}, the proof of which is along the lines given in \cite[Theorem 2.2]{holst1990poisson}.

\begin{definition}
\label{Defn:DistTV}
Let $Y, Z$ be $\Real-$valued random variables. The total variation distance between $Y$ and $Z$ is $\tv{Y}{Z} := \sup_{R \subseteq \Real}|\Pr\{Y \in R\} - \Pr\{Z \in R\}|.$
\end{definition}

\begin{remark}
\label{rem:DistTVInt}
If $Y$ and $Z$ are integer valued random variables defined on the same probability space, then $\tv{Y}{Z} \leq \Pr\{Y \neq Z\};$ see \cite[pg 129]{grimmett2001probability}.
\end{remark}

\begin{remark}
\label{rem:DistTVPoiDiffMean}
$\tv{\Poi(\lambda_1)}{ \Poi(\lambda_2)} \leq |\lambda_1 - \lambda_2|;$ see
\cite[Corollary 3.1]{adell2005sharp}.
\end{remark}

\begin{theorem} \cite[Theorem 2.1]{holst1990poisson}
\label{thm:SuffCondPoissonConv} \textnormal{(Stein-Chen method)}
Let $\Idx$ be a finite index set. For $i \in \Idx,$ let $\Idx_i := \Idx \setminus \{i\}.$ Let $S = \sum_{i \in \Idx}\indc_i$ and $\lambda = \ExP[S],$ where $\{\indc_i\}_{i \in \Idx}$ are some indicator random variables. For each $i,$ let there be a probability space with indicator random variables $\{\zeta_{ji}\}_{j \in \Idx}$ and $\{\eta_{ji}\}_{j \in \Idx_i}$ defined on  it such that
\[
\sL(\zeta_{ji} : j \in \Idx) = \sL(\indc_j : j \in \Idx),
\]
and
\[
\sL(\eta_{ji}: j \in \Idx_i) = \mathscr{L}(\indc_j: j \in \Idx_i \;|\; \indc_i = 1),
\]
where $\mathscr{L}(\cdot)$ denotes the distribution function. Then, for $\cU_i = \sum_{j \in \Idx} \zeta_{ji}$ and $\cV_{i} = \sum_{j \in \Idx_i} \eta_{ji},$
\[
\tv{S}{\Poi(\lambda)} \leq \frac{1 - e^{-\lambda}}{\lambda} \sum_{i \in \Idx} \ExP[\indc_i] \; \ExP|\cU_i - \cV_i|.
\]
In addition, if there exists a partition $\Idx_i = \Idx_i^{+} \cup \Idx_i^{-} \cup \Idx_i^{0}$ with $\zeta_{ji} \leq \eta_{ji}$ a.s. for $j \in \Idx_i^{+},$ and $\zeta_{ji} \geq \eta_{ji}$ a.s. for $j \in \Idx_i^{-},$ then
\[
\tv{S}{\Poi(\lambda)} \leq \frac{1 - e^{-\lambda}}{\lambda} \left[\sum_{i \in \Idx}[\ExP[\indc_i]]^2 + \sum_{i \in \Idx} \sum_{j \in \Idx_i^{+} \cup \Idx_i^{-}} |\cov[\indc_i, \indc_j]| + \sum_{i \in \Idx} \sum_{j \in \Idx_i^{0}} \left[\ExP[\indc_i \indc_j] + \ExP[\indc_i] \ExP[\indc_j]\right] \right].
\]
\end{theorem}

\begin{remark}
For ease of use here, the above result is stated slightly differently from the original version given in \cite[Theorem 2.1]{holst1990poisson}.
\end{remark}

\begin{theorem}
\label{thm:ProbSpaceExistence}
Let $j, k \geq 0.$ Let $\vec{Y}_i,$ for $0 \leq i \leq j,$ be a random vector in $\Real^k$ so that $(\vec{Y}_0, \ldots,  \vec{Y}_j)$ is multivariate Gaussian with $\cov[\vec{Y}_0(\ell_1), \vec{Y}_i(\ell_2)] \geq 0$ for all $0 \leq i \leq j$ and $1 \leq \ell_1, \ell_2 \leq k.$ Then, for any $\vec{u} \in \Real^k,$ there exists a probability space with $\Real^k-$valued random vectors $\{\vec{Z}_{i}: 1 \leq i \leq j\}$ and $\{\vecp{Z}_{i}: 1 \leq i \leq j\}$ such that
\[
\mathscr{L}(\vec{Z}_{i}: 0 \leq i \leq j) = \mathscr{L}(\vec{Y}_{i}: 0 \leq i \leq j),
\]
\[
\mathscr{L}(\vecp{Z}_{i}: 0 \leq i \leq j) = \mathscr{L}(\vec{Y}_{i}: 0 \leq i \leq j \;|\; \vec{Y}_{0} \geq \vec{u}),
\]
and, for all $0 \leq i \leq j, 1 \leq \ell \leq k,$
\[
\vec{Z}_{i}(\ell) \leq \vecp{Z}_{i}(\ell) \text{ a.s.}
\]
The latter implies that $\indc[\vec{Z}_i \geq \vec{u}] \leq \indc[\vecp{Z}_i \geq \vec{u}]$ a.s. for all $0 \leq i \leq j.$
\end{theorem}
\begin{proof}
Let $g_{\vec{u}}: \Real^k \to \Real$ be the map $g_{\vec{u}}(\vec{y}) = \indc[\vec{y} \geq \vec{u}].$ Clearly, $g_{\vec{u}}$ is monotonically increasing in its arguments (with respect to coordinate-wise partial order). Let $h: \Real^{(j + 1)k} \to \Real$ be an arbitrary increasing function in the above sense.
Then, from \cite[Corollary 3]{joag1983association},
\begin{multline}
\label{eqn:ExPprodToProdExp}
\ExP[g_{\vec{u}}(\vec{Y}_{0}(1), \ldots, \vec{Y}_{0}(k)) \,  h(\vec{Y}_0(1), \ldots, \vec{Y}_0(k), \ldots, \vec{Y}_{j}(1), \ldots, \vec{Y}_j(k)) ] \geq \\
\ExP[g_{\vec{u}}(\vec{Y}_{0}(1), \ldots, \vec{Y}_{0}(k))] \ExP[h(\vec{Y}_0(1), \ldots, \vec{Y}_0(k), \ldots, \vec{Y}_{j}(1), \ldots, \vec{Y}_j(k))].
\end{multline}
Hence, we have
\begin{multline*}
\ExP[h(\vec{Y}_0(1), \ldots, \vec{Y}_0(k), \ldots, \vec{Y}_{j}(1), \ldots, \vec{Y}_j(k)) \;|\; \vec{Y}_0 \geq \vec{u}] \geq \\
\ExP[h(\vec{Y}_0(1), \ldots, \vec{Y}_0(k), \ldots, \vec{Y}_{j}(1), \ldots, \vec{Y}_j(k))].
\end{multline*}
The desired result now holds from the equivalence of conditions (i) and (iv) in \cite[Theorem 1]{kamae1977stochastic} (under the coordinate-wise partial order).
\end{proof}

\section{Outline of Proof of Key Results}
\label{sec:ProofOutline}
Here we first state all our major intermediate results and then prove Theorems~\ref{thm:ExPMainResult} and \ref{thm:BehMainResult}. Proofs of these intermediate results are given later in Sections~\ref{sec:BettiBounds}, \ref{sec:IntResults}, \ref{sec:ConclResults}, and the Appendix; the page numbers are noted near the statements. The result here either needs none or only a subset of the assumptions from Table~\ref{tab:Assumptions}. While $\Cr{a:rCAsm}$ is never needed in its entirety, various weaker implications of it are used at different times. These are stated in Table~\ref{tab:MoreAssumptions}.

\begin{table}
\begin{tabular}{c | l | c | c | c}
\hline

Id & Assumption & Applicable when & Type & Implied by\\

\hline

& & & & \\[-2ex]

$\Cl[Asm]{a:r2Asm}$ & $\begin{cases}
1 = 1\\[2ex]
0 \leq \rho_2 < \rho_0
\end{cases}
$
&
$
\parbox{7em}{
\vspace{-1.25ex}
$k = 0$\\[2.75ex]
$1 \leq  k < d$
}
$ & Local & $\Cr{a:r1Asm}$ and $\Cr{a:r2LAsm}$ \\[5ex]
$\Cl[Asm]{a:rqAsm}$ & $\begin{cases}
1 = 1\\[2ex]
0 \leq \sup\limits_{q \geq 4}\rho_q \leq \rho_3 < \rho_2
\end{cases}
$
&
$
\parbox{7em}{
\vspace{-1.25ex}
$k = 0, d - 1$\\[2.75ex]
$1 \leq  k < d - 1$
}
$ & Boundedness & $\Cr{a:rCAsm}$ \\[6ex]
$\Cl[Asm]{a:r2LAsm}$ & $
\begin{cases}
0 \leq \sup\limits_{q \geq 2}\rho_q \leq \rho_1\\[2ex]
0 \leq \sup\limits_{q \geq 3}\rho_q \leq \rho_2 < \rho_1
\end{cases}
$
&
$
\parbox{7em}{
\vspace{-1.25ex}
$k = 0$\\[2.75ex]
$1 \leq  k < d$
}
$ & Boundedness & $\Cr{a:rCAsm}$  \\[6ex]
\hline
\end{tabular}\\[2ex]
\caption{\label{tab:MoreAssumptions} Weaker Implications of Assumption~$\Cr{a:r1Asm}$ and $\Cr{a:rCAsm}.$}
\end{table}

Understanding the statistical behaviour of Betti numbers is not straightforward. This is because they, being rank of some space, are not nice enough combinatorial objects to be handled directly. Hence, as remarked in Section~\ref{sec:intro}, the trick is to use good approximators. We provide these in Theorem~\ref{thm:RBettiBounds}. But to see the motivation behind them, we initially discuss few relevant properties of the Betti numbers of our \Cech complex $\cK(n; u),$ which we establish separately.

First we show that the \Cech and \Rips complexes from Definitions~\ref{Defn:Cech} and \ref{Defn:Rips} are equivalent. The statement is given below and the proof follows from Lemma~\ref{lem:EqGenCechRips}.

\begin{proposition}
\label{prop:EqCechRips}
Let $n \geq 0$ and $u \in \pReal.$ Then, $\cK(n; u) = \hat{\cK}(n; u).$
\end{proposition}

An immediate and an important consequence of the above result and Definition~\ref{Defn:Rips} is that its $1-$skeleton, the underlying graph, completely characterizes $\cK(n; u).$ In that, $\sigma$ is a face in $\cK(n; u)$ if and only if the vertices in $\sigma$ form a clique in the $1-$skeleton of $\cK(n; u).$ In other words, $\cK(n; u)$ is the clique complex associated with its $1-$skeleton.

Henceforth, we shall say that the $k-$th Betti number is non-trivial if it is at least $2$ when $k = 0,$ and at least $1$ when $k \geq 1.$ Then, from Proposition~\ref{prop:EqCechRips}, and Lemma~\ref{lem:CardNTC} along with the discussions above it on the relationship between Betti numbers and NTCs, we right away have that every induced subcomplex of $\cK(n; u)$ with non-trivial $k-$th Betti number has at least $2k + 2$ vertices. Note that the induced subcomplex whose $1-$skeleton is isomorphic to $O_k$ has non-trivial $\beta_k.$ Hence, it follows that the minimal induced subcomplex having non-trivial $\beta_k$ must be the one that is isolated, has $2k + 2$ vertices, and whose $1-$skeleton is isomorphic to $O_k.$ Extending these facts to our setup, we prove the following additional characteristic of induced subcomplexes with non-trivial $\beta_k.$ Its proof is via induction and follows from Theorem~\ref{thm:DistantVertices}.

\begin{theorem}
\label{thm:RDistantVertices}
In every induced subcomplex $\cC \subseteq \cK(n; u)$ with non-trivial $k-$th Betti number, where $0 \leq k < d,$ there exist at least $2k + 2$ vertices with pairwise $\|\cdot\|_1$-distances at least $2.$
\end{theorem}

In general, the $1-$skeleton of a minimal induced subcomplex with non-trivial $k-$th Betti number can either be $\|\cdot\|_1-$isometric to $O_k$ or not; Figure~\ref{fig:BettiConfig} gives an example for each of these cases when $d = 3$ and $k = 1.$ Keeping in mind this observation, Figure~\ref{fig:GaussianExcursions},  and the well known topological fact, that the $k-$th Betti number is bounded from above by the number of $k-$faces, we now introduce the different terms with which we approximate the Betti numbers. Let
\begin{align*}
\cSnk(u)  & :=
\begin{cases}
\text{number of isolated vertices in $\cK(n; u),$} & \hspace{1.5em} \text{if $k = 0,$} \\[1ex]
\parbox{27em}{
number of subgraph components in the $1-$skeleton of $\cK(n; u)$ that are isomorphic and $\|\cdot\|_1-$isometric to $O_k,$
}
& \hspace{1.5em}  \text{if $1 \leq k < d;$}
\end{cases}
\end{align*}
and
\begin{align*}
\cNnk(u) & :=
\begin{cases}
0,                  & \hspace{1.5em}  \text{if $k = 0,$} \\[1ex]
\parbox{27em}{
number of subgraph components in the $1-$skeleton of
$\cK(n; u)$ that are isomorphic but not $\|\cdot\|_1$-isometric to $O_k,$
}
& \hspace{1.5em}  \text{if $1 \leq k < d.$}
\end{cases}
\end{align*}\\[1ex]
Further, for $0 \leq k < d,$ to deal with non-minimal subcomplexes that contribute to the $k-$th Betti number, let $\cLnk(u)$ be the number of subgraphs in the $1-$skeleton of $\cK(n; u)$ that are isomorphic to a graph $G \in \cP_k,$ where $\cP_k$ is as specified below. With $\Idx_{2k + 1}$ defined as above Definition~\ref{Defn:Cech}, let $\cG_k$ be the geometric graph  on $\Idx_{2k + 1}$ with respect to $\sB_\infty(\cdot, \tfrac{1}{2});$ i.e., an edge is present between two vertices whenever the $\|\cdot\|_\infty-$distance between them equals $1.$ Now, if $k \geq 2,$ let
\begin{multline}
\label{Defn:cPk}
\cP_k := \{\text{Subgraphs } G \equiv (V, E) \subseteq \cG_k : G \text{ is connected, }  \text{ elements of $V$ are ordered } \\
\text{lexicographically, } |V| \geq 2k + 3, \; \exists \text{ $2k+2$ vertices in $V$ with pairwise} \\ \text{$\|\cdot\|_1$-distances $\geq 2,$ $\exists$ \text{unique $(k + 1)-$sized subset of $V$ that forms a clique in $G$}} \};
\end{multline}
whereas, if $k = 1,$ let
\begin{multline}
\label{Defn:cP1}
\cP_1 := \{\text{Subgraphs } G \equiv (V, E) \subseteq \cG_1 : G \text{ is connected, }  \text{ elements of $V$ are ordered } \\
\text{lexicographically, } |V| \geq 5, \; \text{ $\exists \; 4$ vertices in $V$ with pairwise $\|\cdot\|_1$-distances $\geq 2$} \};
\end{multline}
and, lastly, if $k = 0,$ let
\begin{multline}
\label{Defn:cP0}
\cP_0 := \{\text{Subgraphs } G \equiv (V, E) \subseteq \cG_0 : G \text{ is connected, }  \\ \text{ elements of $V$ are ordered }
\text{lexicographically, } |V| = 2 \}.
\end{multline}

Loosely, $\cP_k$ is the set of subgraphs to which each $k-$face, of an induced subcomplex with non-trivial $\beta_k,$ can be extended into. For $1 \leq k < d,$ the idea is similar  to \cite[Fig. 1]{kahle2011random}; the difference is, while there the fact that each subgraph has $2k + 3$ vertices is harnessed, we additionally exploit the property that there is a further subset of $2k + 2$ vertices with pairwise $\|\cdot\|_1-$distances at least $2.$

\begin{figure}[t]
\centering
\includegraphics[width=0.5\textwidth]{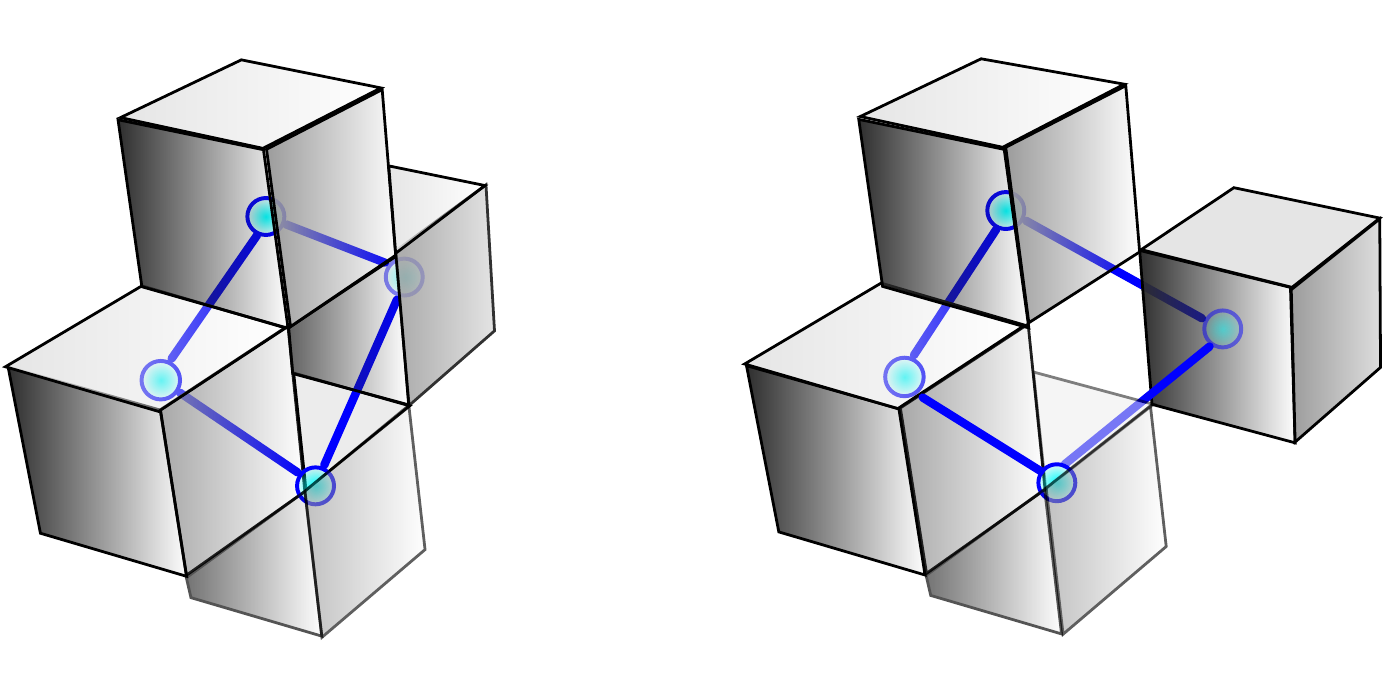} \caption{\label{fig:BettiConfig} A few examples of minimal induced subcomplexes with non-trivial $\beta_1$ when $d = 3.$}
\end{figure}

Using the above notions, the aforementioned approximators for Betti numbers of $\cK(n; u)$ are the lower and upper bounds given below. Its proof is via induction and  follows from Theorem~\ref{thm:BettiBounds}.

\begin{theorem}
\label{thm:RBettiBounds}
Let $u \in \pReal$ and $k$ be such that $0 \leq k < d.$ Then,
\[
\cSnk(u) \leq \beta_{n, k}(u) \leq \cSnk(u) + \cNnk(u) + {\Cl[Const]{c:Betti}}_{, k} \, \cLnk(u),
\]
where ${\Cr{c:Betti}}_{, k} = 1 $ when $k = 0$ or $2 \leq k < d,$ while  ${\Cr{c:Betti}}_{, 1}$ is an universal upper bound on the number of edges in a graph $G \in \cP_1.$
\end{theorem}

\begin{remark}
The constant ${\Cr{c:Betti}}_{, 1}$ in the above result differs for the same reasons as discussed below \cite[Fig.~3]{kahle2011random}. But, roughly, $\check{L}_{n, 1}(u)$ needs to be scaled suitably for it be a valid upper bound for the number of edges and subsequently for the first Betti number.
\end{remark}

We now elaborate on how we use these approximators or bounds to prove our key results. The core idea is to show that, as $n$ and $u$ become large, $\cNnk(u)$ and $\cLnk(u)$ become negligible compared to $\cSnk(u)$ and that the latter more or less determines $\beta_{n, k}(u).$ At this point, it is worth noting that the two terms $\cSnk(u)$ and $\cNnk(u)$ are not simple combinatorial objects themselves. In that, these terms count components and the isolation condition, enforced thereby, is not easy to deal with directly. We overcome this difficulty by further approximating these terms with even simpler expressions. Specifically, instead of $\cSnk(u),$ we actually deal with
\begin{equation}
\label{Defn:Snk}
S_{n, k}(u) =
\begin{cases}
\sum_{\vec{t} \in \Idx_n}\indc[X_{\vec{t}} \geq u], & \text{ if } k = 0,\\[1ex]
\sum_{\tOm \in \IdxnCrDir} \indc[\vecXtOm \geq u \ones_{2b_k}], & \text{ if }1 \leq k < d.
\end{cases}
\end{equation}
Here, for $i \geq 1,$ $\set{i}$ is the increasingly ordered set $(1, \ldots, i);$ $\tbinom{\set{d}}{b_k}$ is the collection of all increasingly ordered subsets of $\setd$ with size $b_k;$ for $\vec{\Omega} \equiv (\omega_1, \ldots, \omega_{k + 1}) \in \tbinom{\set{d}}{b_k},$
\begin{equation}
\label{Defn:vecX}
\vecXtOm = (X_{\vec{t} + \vec{e}_{\omega_1}}, X_{\vec{t} - \vec{e}_{\omega_1}}, \ldots, X_{\vec{t} + \vec{e}_{\omega_{b_k}}}, X_{\vec{t} - \vec{e}_{\omega_{b_k}}});
\end{equation}
and $\ones_i$ is the all ones $i-$dimensional row vector. The set $\tbinom{\set{d}}{b_k}$ captures the different orientations in which the minimal Betti generating complexes can occur; see the image on the right in Figure~\ref{fig:GaussianExcursions} for an illustration. Working with $\Snku$ is much easier than $\cSnk(u)$ as the former does not involve conditions enforcing isolation and, for $k$ such that $1 \leq k < d,$ it also has no condition on $X_{\vec{t}}$ itself.

In the same spirit, as against $\cNnk(u),$ we work with $\Nnk(u)$ which is defined as follows. For each $k$ such that $1 \leq k < d - 1,$ let
\begin{multline}
\label{Defn:cNk}
\cN_k := \{\text{Subgraphs } G \equiv (V, E) \subseteq \cG_0 : |V| = 2k + 2,  \text{ elements of $V$ are ordered} \\
\text{lexicographically, } G \text{ is isomorphic but not $\|\cdot\|_1-$isometric to } O_k \},
\end{multline}
where $\cG_0$ is as defined above \eqref{Defn:cPk}. For $\vec{t} \in \dGrid$ and $G \equiv (V, E),$ where $V \subseteq \dGrid$ is ordered, let $\vecXtG$ be the ordered vector $(X_{\vec{t} + \vec{v}} : \vec{v} \in V).$ Then,
\begin{equation}
\label{Defn:Nnk}
\Nnk(u) :=
\begin{cases}
0, & \text{ if $k = 0, d - 1,$} \\
\sum_{\tG \in \IdxnCrNk} \indc[\vecXtG \geq u \ones_{2b_k}], & \text{ if $1 \leq k < d - 1.$}
\end{cases}
\end{equation}

Separately, for $0 \leq k < d,$ let
\begin{equation}
\label{Defn:Lnk}
\Lnk(u) :=
\sum_{\tG \in \IdxnCrPk} \indc[\vecXtG \geq u \ones_{|V|}],
\end{equation}
where $\cP_k$ is as in \eqref{Defn:cPk}, \eqref{Defn:cP1}, or \eqref{Defn:cP0}; $V$ is the vertex set of $G;$ and $\vecXtG$ is as defined above. While $\cLnk(u)$ does not count components, nevertheless, the above expression is the easier to handle.

The next result compares $\Snk(u), \Nnk(u),$ and $\Lnk(u)$ with the approximators in Theorem~\ref{thm:RBettiBounds}; its proof is given in Section~\ref{sec:IntResults}, p.~\pageref{Pf:BettiApprox}. As we shall see in Remark~\ref{rem:SnkImp} later, the additional $\Dnk(u)$ term is due to boundary conditions.
\begin{lemma}
\label{lem:BettiApprox}
Fix $k$ so that $0 \leq k < d.$ Let $n \geq 0$ and $u \in \pReal.$ Then, the below statements hold:
\begin{enumerate}
\item \label{itm:BdSnku} $\cSnk(u) \leq \Snku \leq \cSnk(u) + {\Cr{c:Betti}}_{,k} \, \cLnk(u) + \Dnk(u),$
where ${\Cr{c:Betti}}_{, k}$ is as in Theorem~\ref{thm:RBettiBounds}, and
\begin{equation}
\label{Defn:Dnk}
\Dnk(u) :=
\begin{cases}
0, & \text{ if $k = 0,$} \\[1ex]
\sum\limits_{\substack{\tOm \in \IdxnCrDir:\\ \|\vec{t}\|_\infty = n}} \indc[\vecXtOm \geq u \ones_{2b_k}], & \text{ if $1 \leq k < d.$}
\end{cases}
\end{equation}

\item \label{itm:BdNnku} $\cNnk(u) \leq \Nnk(u).$

\item \label{itm:BdLnku} $\cLnk(u) \leq \Lnk(u).$

\end{enumerate}
\end{lemma}

An immediate consequence of Theorem~\ref{thm:RBettiBounds} and Lemma~\ref{lem:BettiApprox} is the following result, which bounds the difference between $\Snk(u)$ and $\Bnk(u).$

\begin{lemma}
\label{lem:BSDiffBd}
Fix $k$ such that $0 \leq k < d.$ Let $n \geq 0$ and $u \in \pReal.$ Then,
\[
|\Bnk(u) - \Snk(u)| \leq \Nnk(u) +  {\Cr{c:Betti}}_{,k} \, \Lnk(u) +  \Dnk(u),
\]
where ${\Cr{c:Betti}}_{, k}$ is as in Theorem~\ref{thm:RBettiBounds}.
\end{lemma}

Because of the above result, in order to prove Theorems~\ref{thm:ExPMainResult} and \ref{thm:BehMainResult}, it suffices to establish the corresponding results first for $\Snk(u)$ and then show that $\Nnk(u),$ $\Lnk(u),$ and $\Dnk(u)$ are all negligible relative to $\Snk(u).$ Indeed, this is precisely what we do.

In the above line of thought, our first result describes the asymptotic behaviour of the mean value, along with the rate of convergence, of the different terms in Lemma~\ref{lem:BSDiffBd}. Before stating it formally, we introduce few notations. For $m \geq 1,$ and $\rho, \mup \in [0, 1),$ let
\begin{equation}
\label{Defn:phiCons}
\phi_m(\rho, \mup) = \dfrac{m + 1 + (m - 1) \rho - 2 m \mup}{2(1 + (m - 1) \rho - m [\mup]^2)}.
\end{equation}
Now, define constants
\begin{equation}
\label{Defn:VarRho}
\varrho_k := \phi_{2b_k -1}(\rho_2, \rho_3), \quad \text{ if $1 \leq k < d,$}
\end{equation}
and
\begin{equation}
\label{Defn:VarPhi}
\varphi_k := \phi_{2b_k}(\rho_2, \rho_1) = \frac{1 - 2a_k [2\rho_1 - 1]}{2[1 - 2a_k \rho_1^2]}, \quad \text{ if $0 \leq k < d,$}
\end{equation}

\begin{lemma}
\label{lem:limRates}
Fix $k$ such that $0 \leq k < d.$ Let $\lbnku$ be as in \eqref{Defn:Lambdanku} and let $\{\un\} \in \pSeq$ be such that $\lim_{n \to \infty} \un = \infty.$ Then, the following statements are true.
\begin{enumerate}[leftmargin=*, labelindent=1.35em]
\itemsep2ex

\item \label{itm:ExPSnku} If $\Cr{a:r0Asm}$ and  $\Cr{a:r2Asm}$ hold, then $\left|\dfrac{\ExP[S_{n, k}(\un)]}{\lambda_{n, k}(\un)} - 1\right| = O\left(\dfrac{1}{\un^2}\right).$
\item \label{itm:ExPNnku} If $\Cr{a:r0Asm}, \Cr{a:r2Asm},$ and $\Cr{a:rqAsm}$ hold, then $\varrho_k > a_k,$ for $1\leq k < d - 1,$ and
\[
\dfrac{\ExP[\Nnk(\un)]}{\lambda_{n, k}(\un)} =
\begin{cases}
0, & \text{ if $k = 0, d - 1,$}\\
O(\exp[-(\varrho_k - a_k)\un^2]), & \text{ if $1 \leq k < d - 1.$}
\end{cases}
\]
\item \label{itm:ExPLnku} If $\Cr{a:r0Asm}, \Cr{a:r1Asm},$ $\Cr{a:ComC}(k),$ and $\Cr{a:r2LAsm}$ hold, then $\varphi_k > a_k$ and  $\dfrac{\ExP[\Lnk(\un)]}{\lambda_{n, k}(\un)} = O(\un^{-1} \exp[-(\varphi_k - a_k) \un^2]).$
\item \label{itm:ExPDnku} If $\Cr{a:r0Asm}$ and  $\Cr{a:r2Asm}$ hold, then
\[
\dfrac{\ExP[\Dnk(\un)]}{\lambda_{n, k}(\un)} =
\begin{cases}
0, & \text{ if $k = 0,$}\\
O\left(\dfrac{1}{n}\right), & \text{ if $1 \leq k < d.$}
\end{cases}
\]
\end{enumerate}
\end{lemma}

The proof of this result is given in Section~\ref{sec:IntResults}, p.~\pageref{pf:limRates} and uses Slepian's lemma (Lemma~\ref{lem:SlepainInequality}) and Multivariate Gaussian tail estimates (Lemma~\ref{lem:GaussianTailBounds}); the Savage condition in the latter is ensured due to our assumptions in Table~\ref{tab:Assumptions}.

It is now straightforward to establish Theorem~\ref{thm:ExPMainResult} and the Statement~\ref{itm:Vanishing} of Theorem~\ref{thm:BehMainResult}.

\begin{proof}[Proof of Theorem~\ref{thm:ExPMainResult}] By a simple triangle inequality and Lemma~\ref{lem:BSDiffBd},
\[
\left|\frac{\ExP[\beta_{n, k}(\un)]}{\lambda_{n, k}(\un)} - 1\right| \leq \left|\dfrac{\ExP[S_{n, k}(\un)]}{\lambda_{n, k}(\un)} - 1\right| + \dfrac{\ExP[\Nnk(\un)]}{\lambda_{n, k}(\un)} +  \dfrac{\ExP[\Lnk(\un)]}{\lambda_{n, k}(\un)} + \dfrac{\ExP[\Dnk(\un)]}{\lambda_{n, k}(\un)}.
\]
Because $1/\un^2 + \exp[-(\varphi_k - a_k)\un^2] + \exp[-(\varrho_k - a_k)\un^2] = O(1/\un^2),$ the desired result is now easy to see from Lemma~\ref{lem:limRates}.
\end{proof}

\begin{proof}[Proof of Theorem~\ref{thm:BehMainResult}.\ref{itm:Vanishing}]
From Markov's inequality, we have
\[
1 - \Pr\{\Bnk(\un) = 0\} = \Pr\{\Bnk(\un) \geq 1\} \leq \ExP[\Bnk(\un)].
\]
On the other hand,
\[
\ExP[\Bnk(\un)] = \lbnk(\un) \frac{\ExP[\Bnk(\un)]}{\lbnk(\un)} = \lbnk(\un) O\left(1 + \frac{1}{\un^2} + \frac{1}{n} \right) = \lbnk(\un) O(1) = O(\lbnk(\un)),
\]
where the second relation follows from Theorem~\ref{thm:ExPMainResult}, while the third one follows because both $\un$ and $n$ grow to $\infty.$ Finally, from \eqref{Defn:Lambdanku} and \eqref{eqn:unForm}, and since $\nu_n \to \infty,$ we have
\[
\lbnk(\un) = O\left(e^{-\nu_n} \left[\frac{\log n}{\log n + \nu_n}\right]^{b_k} \right).
\]
The desired result is now easy to see.
\end{proof}

Moving on, our next result shows that $\Snk(\un)$ has a Poisson behaviour when $\lbnk(\un)$ converges to a constant. Let %
\begin{equation}
\label{Defn:VarTheta}
\vartheta_k := \frac{[1 + (2b_k - 1)\rho_2 - 2b_k \rho_1]^2}{1 + (2b_k - 1) \rho_2 - 2b_k \rho_1^2} = \frac{b_k [1 - 2a_k \rho_1]^2}{a_k [1 - 2a_k \rho_1^2]}, \quad \text{ if $0 \leq k < d,$}
\end{equation}

\begin{theorem}
\label{thm:PoiResSnk}
Fix $k$ such that $0 \leq k < d.$ Let $\lbnku$ be as in \eqref{Defn:Lambdanku} and $\{\un\} \in \pSeq$ be as in \eqref{eqn:unForm} with $\lim_{n \to \infty} \nu_n = \log \left[\tfrac{1}{\lambda}\right]$ for some fixed $\lambda \in \pReal.$ Suppose $\Cr{a:r0Asm}, \Cr{a:r1Asm}, \Cr{a:ComC}(k), \Cr{a:PoiG},$ and $\Cr{a:r2LAsm}$ hold. Then, $0 < \vartheta_k < 1;$ further,
\[
\tv{\Snk(\un)}{\Poi(\ExP[\Snk(\un)])} = O\left(n^{-d \vartheta_k/(2 b_k)} [\log n]^{-(1 - \vartheta_k)/2} + \frac{\log n}{n} \sum_{q = 1}^{2dn} \rho_q\right).
\]
\end{theorem}

The $k = 0$ case of this result has been shown in \cite[Theorem 3.6]{holst1990poisson} using Theorem~\ref{thm:SuffCondPoissonConv}. We build upon their ideas to show the multi-dimensional version here. Our proof is given in Section~\ref{sec:IntResults}, p.~\pageref{Pf:PoiResSnk}.

Using the above result and Lemma~\ref{lem:limRates}, we now establish the Poisson result for $\Bnk(\un).$

\begin{proof}[Proof of Theorem~\ref{thm:BehMainResult}.\ref{itm:Poisson}]
By the triangle inequality, and Remarks~\ref{rem:DistTVInt} and \ref{rem:DistTVPoiDiffMean}, %
\begin{eqnarray}
\tv{\Bnk(\un)}{\Poi(\lambda)} & \leq & \tv{\Snk(\un)}{\Poi(\ExP[\Snk(\un)])} \nonumber \\
& & \tv{\Poi(\ExP[\Snk(\un)])}{\Poi(\lambda)} + \tv{\Bnk(\un)}{\Snk(\un)} \nonumber\\
& \leq & \tv{\Snk(\un)}{\Poi(\ExP[\Snk(\un)])}  \label{eqn:dTV_TriangleBd} \\
& & + \; |\ExP[\Snk(\un)] - \lambda| + \Pr\{\Bnk(\un) \neq \Snk(\un)\}. \nonumber
\end{eqnarray}
Separately, Markov's inequality and Lemma~\ref{lem:BSDiffBd} show that
\begin{eqnarray}
\Pr\{\Bnk(\un) \neq \Snk(\un)\} & = & \Pr\{|\Bnk(\un) - \Snk(\un)| \geq 1\} \nonumber \\
& \leq & \ExP[\Nnk(\un)] + \ExP[\Lnk(\un)] + \ExP[\Dnk(\un)]. \label{eqn:BNeqSBd}
\end{eqnarray}
Further, we have
\begin{eqnarray}
& & \ExP[\Nnk(\un)] + \ExP[\Lnk(\un)] + \ExP[\Dnk(\un)] \nonumber \\
& = & O(\exp[-(\varrho_k - a_k)\un^2]) + O(\un^{-1} \exp[-(\varphi_k - a_k) \un^2]) + O\left(\frac{1}{n}\right) \label{eqn:NLDSumBd0} \\
& = & O\left(\frac{[\log n]^{(\varrho_k - a_k)b_k/a_k}}{n^{d(\varrho_k - a_k)/a_k}}\right) +  O\left(\frac{[\log n]^{(\varphi_k - a_k)b_k/a_k - 0.5}}{n^{d(\varphi_k - a_k)/a_k}}\right)+  O\left(\frac{1}{n}\right) \label{eqn:NLDSumBd}
\end{eqnarray}
and
\begin{equation}
\label{eqn:PExpSnkBd}
|\ExP[\Snk(\un)] - \lbnk(\un)| = O\left(\frac{1}{\un^2}\right) = O\left(\frac{1}{\log n}\right),
\end{equation}
where \eqref{eqn:NLDSumBd0} and \eqref{eqn:PExpSnkBd} hold due to Lemma~\ref{lem:limRates} and the fact that $\lbnk(\un) \to \lambda,$ a constant, while \eqref{eqn:NLDSumBd} follows by observing $\exp[-(\varrho_k - a_k)\un^2] = [\tau_k (2n + 1)^d \un^{-2b_k}]^{-(\varrho_k - a_k)/a_k}[\lbnk(\un)]^{(\varrho_k - a_k)/a_k},$ and so on, and then making use of the facts that $\lbnk(\un) \to \lambda$ and $a_k \un^2/[d\log n] \to 1.$

By substituting \eqref{eqn:unForm} in \eqref{Defn:Lambdanku}, additionally observe that
\[
|\lbnk(\un) - \lambda| = \left|\left[\frac{\log[\tau_k (2n + 1)^d]}{\left[\log[\tau_k (2n + 1)^d] - b_k \log\left[\frac{\log[\tau_k (2n + 1)^d]}{a_k}\right] + \nu_n \right]}\right]^{b_k} e^{-\nu_n} - \lambda\right|.
\]
Using the triangle inequality and the fact that $\lim_{n \to \infty} \nu_n = \log[1/\lambda],$ it then follows that
\[
|\lbnk(\un) - \lambda| = O(|e^{-\nu_n} - \lambda|) + O\left(\left|\left[\frac{\log[\tau_k (2n + 1)^d]}{\log[\tau_k (2n + 1)^d] - b_k \log\left[\frac{\log[\tau_k (2n + 1)^d]}{a_k}\right] + \nu_n }\right]^{b_k} - 1\right|\right).
\]
Recall that, if $f(x) = x^{b_k},$ then, by the monotonicity of $f'(x)$ and the mean value theorem,
\[
|f(x) - f(1)| \leq b_k \left[\sup_{c \in \{x, 1\}} x^{b_k - 1}\right] |x - 1|.
\]
Applying this above and using the fact that $\lim_{n \to \infty} \nu_n$ is a constant, we get
\[
|\lbnk(\un) - \lambda| = O(|e^{-\nu_n} - \lambda|) + O\left(\frac{\log \log n}{ \log n}\right).
\]
Consequently, the triangle inequality and \eqref{eqn:PExpSnkBd} show
\begin{equation}
\label{eqn:PExpLbBd}
|\ExP[\Snk(\un)] - \lambda| = O(|e^{-\nu_n} - \lambda|) + O\left(\frac{\log \log n}{ \log n}\right).
\end{equation}

The desired result now follows from \eqref{eqn:dTV_TriangleBd}, \eqref{eqn:BNeqSBd}, \eqref{eqn:NLDSumBd}, \eqref{eqn:PExpLbBd}, and Theorem~\ref{thm:PoiResSnk}.
\end{proof}

Finally, it remains to establish the non-vanishing regime behavior in Theorem~\ref{thm:BehMainResult}. Towards that, we establish the following two results. The first one discusses second order moments of $\Snk(u), \Nnk(u), \Lnk(u),$ and $\Dnk(u),$ while the next one obtains a CLT for $\Snk(u).$ Unlike Theorem~\ref{thm:BehMainResult}, note that the CLT for $\Snk(u)$ holds for the entire non-vanishing regime (see    Remark~\ref{rem:CLTlimit}).

\begin{lemma}
\label{lem:VarianceEstimates}
Fix $k$ such that $0 \leq k < d.$ Let $\lbnku$ be as in \eqref{Defn:Lambdanku} and $\{\un\} \in \pSeq$ be as in \eqref{eqn:unForm} with $\nu_n \to -\infty.$ Then, the following statements are true.
\begin{enumerate}

\item \label{itm:SnkVarEst} If $\Cr{a:r0Asm}, \Cr{a:r1Asm}, \Cr{a:ComC}(k),$ $\Cr{a:CLTG},$ and $\Cr{a:r2LAsm}$ hold, then $\left|\dfrac{\Var[\Snk(\un)]}{\lbnk(\un)} - 1\right| = O \left(\frac{1}{\un^2}\right).$

\item \label{itm:SqNnkEst} If $\Cr{a:r0Asm}, \Cr{a:CLTG}, \Cr{a:r2Asm},$ and $\Cr{a:rqAsm}$ hold, then
\[
\dfrac{\ExP[\Nnk^2(\un)]}{\lbnk^2(\un)}  =
\begin{cases}
0, & \text{ if $k = 0, d - 1,$ } \\
O\left(\dfrac{\exp[-(\varrho_k - a_k) \un^2]}{\lbnk(\un)}\right) + O\left(\exp[-2(\varrho_k - a_k) \un^2]\right), & \text{ if $1 \leq k < d - 1.$}
\end{cases}
\]

\item \label{itm:SqLnkEst} If $\Cr{a:r0Asm}, \Cr{a:r1Asm}, \Cr{a:ComC}(k),$ $\Cr{a:CLTG},$ and $\Cr{a:r2LAsm}$ hold, then
\[
\dfrac{\ExP[\Lnk^2(\un)]}{\lbnk^2(\un)}  = O\left(\dfrac{\un^{-1} \exp[-(\varphi_k - a_k) \un^2]}{\lbnk(\un)} \right) +  O\left(\un^{-2} \exp[-2(\varphi_k - a_k) \un^2]\right).
\]

\item \label{itm:SqDnkEst} If $\Cr{a:r0Asm}, \Cr{a:r1Asm}, \Cr{a:ComC}(k),$ $\Cr{a:CLTG},$ and $\Cr{a:r2LAsm}$ hold, then
\[
\dfrac{\ExP[\Dnk^2(\un)]}{\lbnk^2(\un)} =
\begin{cases}
0, & \text{ if $k = 0,$} \\
O\left(\dfrac{1}{n\lbnk(\un)}\right)  + O\left(\dfrac{1}{n^2}\right), & \text{ if $1 \leq k < d.$}
\end{cases}
\]

\end{enumerate}
\end{lemma}
The proof of this result is given in Section~\ref{sec:IntResults}, p.~\pageref{Pf:VarianceEstimates}. Statements~\ref{itm:SnkVarEst} and \ref{itm:SqDnkEst} need the covariance bounds from Lemmas~\ref{lem:CovBounds0} and \ref{lem:CovBounds}, while Statements~\ref{itm:SqNnkEst} and \ref{itm:SqLnkEst} require the covariance type bounds from Lemma~\ref{lem:CovBdsNL}.

\begin{theorem}
\label{thm:SCLT}
Fix $k$ such that $0 \leq k < d.$ Let $\lbnku$ be as in \eqref{Defn:Lambdanku} and $\{\un\} \in \pSeq$ be as in \eqref{eqn:unForm} with $\lim_{n \to \infty} \nu_n = -\infty.$ Suppose $\Cr{a:r0Asm}, \Cr{a:r1Asm}, \Cr{a:ComC}(k), \Cr{a:CLTG},$ and $\Cr{a:r2LAsm}$  hold. Then,
\[
\frac{S_{n, k}(\un)-\ExP[S_{n, k}(\un)]}{\sqrt{\lambda_{n,k}(\un)}} \CiD \Gau(0, 1).
\]
\end{theorem}
The proof for this result is given in Section~\ref{sec:IntResults}, p.~\pageref{Pf:SCLT}. Under the new set of assumptions, the key trick, as in Theorem~\ref{thm:PoiResSnk}, again is to show $\tv{\Snk(\un)}{\Poi(\ExP[\Snk(\un)])} \to 0.$ The desired result then follows from the fact that a sequence of Poisson random variables with mean tending to $\infty,$ after suitable normalization, converge in distribution to a standard Gaussian.

We now discuss the proof of Theorem~\ref{thm:BehMainResult}.\ref{itm:NonVanishing}.

\begin{proof}[Proof of Theorem.~\ref{thm:BehMainResult}.\ref{itm:NonVanishing}]
By Chebyshev's inequality,
\[
\Pr\left\{\left|\frac{\Bnk(\un) - \ExP[\Bnk(\un)]}{\lbnk(\un)}\right| \geq \epsilon \right\} \leq \frac{\Var[\Bnk(\un)]}{\epsilon^2 \lbnk^2(\un)}
\]
for all $\epsilon > 0.$ Additionally, if $\epsilon \in (0, 1),$ then it follows from Theorem~\ref{thm:ExPMainResult} that for all large enough $n$
\[
\{\Bnk(\un) = 0\} \subseteq \left\{\left|\frac{\Bnk(\un) - \ExP[\Bnk(\un)]}{\lbnk(\un)}\right| \geq \epsilon \right\},
\]
and hence
\[
\Pr\{\Bnk(\un) = 0\} \leq \Pr\left\{\left|\frac{\Bnk(\un) - \ExP[\Bnk(\un)]}{\lbnk(\un)}\right| \geq \epsilon \right\}.
\]
Consequently, both \eqref{eqn:NonVanishing} and \eqref{eqn:WeakLaw} are simple consequences of \eqref{eqn:BnkVarEst} which we now prove.

Observe that
\begin{eqnarray*}
\Var[\Bnk(\un)] & \leq & 2 \Var[\Snk(\un)] +  2 \,  \Var[\Bnk(\un) - \Snk(\un)]\\
& \leq &  2 \Var[\Snk(\un)] + 2 \ExP|\Bnk(\un) - \Snk(\un)|^2\\
& = & O\left(\Var[\Snk(\un)] + \ExP[\Nnk^2(\un)] + \ExP [\Lnk^2(\un)] + \ExP[\Dnk^2(\un)]\right),
\end{eqnarray*}
where the first relation follows since, for any two random variables $Y_1$ and $Y_2,$ we have $2 |\cov[Y_1, Y_2]|  \leq \Var[Y_1] + \Var[Y_2];$ the second holds because $\Var[Y] \leq \ExP|Y|^2$ for any random variable $Y;$ while the third one follows from Lemma~\ref{lem:BSDiffBd}. From Lemma~\ref{lem:VarianceEstimates}, it is then easy to see that
\[
\frac{\Var[\Bnk(\un)]}{\lbnk^2(\un)} = O\left(\frac{1}{\lbnk(\un)}\right) + O\left(\exp[-2(\varrho_k - a_k)\un^2]\right) + O\left(\un^{-2}\exp[-2(\varphi_k - a_k)]\right) + O\left(\frac{1}{n^2}\right).
\]
Since $\varrho_k > a_k$ and $\varphi_k > a_k$ on account on Lemma~\ref{lem:limRates} and, also, since $\lbnk(\un) \to \infty,$ it follows that \eqref{eqn:BnkVarEst} holds, as desired.

We now turn to proving \eqref{eqn:BnkCLT}. Let
\begin{equation}
\label{Defn:conNu}
\conNu_k :=
\begin{cases}
\frac{\varphi_k - a_k}{\varphi_k - a_k/2}, & \text{ if $k = 0,$}\\[1ex]
\min\left\{\frac{\varrho_k - a_k}{\varrho_k - a_k/2}, \frac{\varphi_k - a_k}{\varphi_k - a_k/2}, \frac{2}{d}\right\}, & \text{ if $1 \leq k < d.$}
\end{cases}
\end{equation}
%
%
Since $\Cr{a:rCAsm}$ holds, note that $\rho_2 \geq 0.$ Hence, $a_k \geq 0.$ Therefore, from Lemmas~\ref{lem:limRates}.\ref{itm:ExPNnku} and \ref{lem:limRates}.\ref{itm:ExPLnku} and by the above definition, it follows that $\conNu_k \in (0, 1].$

Suppose that $1 \leq k < d.$ Using Lemma~\ref{lem:limRates}, we have
\begin{eqnarray*}
& & \Pr\left\{\left|\frac{[\Bnk(\un) - \Snk(\un)] - \ExP[\Bnk(\un) - \Snk(\un)]}{\sqrt{\lbnk(\un)}}\right| > \epsilon \right\}  \\
& \leq & \frac{2}{\epsilon}\frac{\ExP[\Nnk(\un)] + \ExP[\Lnk(\un)] + \ExP[\Dnk(\un)]}{\sqrt{\lbnk(\un)}}
\end{eqnarray*}
for any $\epsilon > 0.$ Hence, \eqref{eqn:BnkCLT} follows from Theorem~\ref{thm:SCLT} and Slutsky's Theorem once we show that the three terms on the RHS decays to zero.

Observe that
\begin{eqnarray*}
& & \frac{\ExP[\Nnk(\un)]}{\sqrt{\lbnk(\un)}} \\
& = & O\left(\frac{(2n + 1)^{d/2}}{\un^{b_k}} \exp[-(\varrho_k - a_k/2)\un^2]\right) \\
& = & O\left(\frac{(2n + 1)^{d/2}}{\un^{b_k}} \exp\left[-[\varrho_k/a_k - 1/2]\left[[1 - C] \log (2n + 1)^d - b_k \log\left[\log (2n + 1)^d\right]\right]\right]\right)\\
& = & O\left(\frac{1}{\un^{b_k}} \exp\left[-[\varrho_k/a_k - 1/2]\left[\left[\frac{\varrho_k - a_k}{\varrho_k - a_k/2} - C\right] \log (2n + 1)^d - b_k \log\left[\log (2n + 1)^d \right]\right]\right]\right) \\
& \to & 0,
\end{eqnarray*}
where the first relation follows from Lemma~\ref{lem:limRates}.\ref{itm:ExPNnku} and \eqref{Defn:Lambdanku}, the second holds  due to \eqref{eqn:unForm} and the assumptions on $\nu_n,$ the third relation follows by simple algebra, while the last relation follows from \eqref{Defn:conNu} and since $C < \conNu_k.$

Similarly, one can see that $\lim_{n \to \infty} \ExP[\Lnk(\un)]/\sqrt{\lbnk(\un)} = 0.$

Next, observe that
\begin{eqnarray*}
& & \frac{\ExP[\Dnk(\un)]}{\sqrt{\lbnk(\un)}} \\
& = & O\left(\frac{(2n + 1)^{d/2}}{n \un^{b_k}} \exp[-a_k \un^2/2]\right)\\
& = & O\left(\frac{(2n + 1)^{d/2}}{n \un^{b_k}} \exp\left[-[1/2]\left[[1 - C] \log (2n + 1)^d - b_k \log\left[\log (2n + 1)^d\right]\right]\right]\right) \\
& \to & 0,
\end{eqnarray*}
where the first relation follows from Lemma~\ref{lem:limRates}.\ref{itm:ExPDnku} and \eqref{Defn:Lambdanku}, the second one holds on account of \eqref{eqn:unForm} and the assumptions on $\nu_n,$ while the last relation follows since $\conNu_k \leq 2/d$ and $C < \conNu_k.$

This establishes \eqref{eqn:BnkCLT} for the $1 \leq k < d$ case. One can similarly argue the $k = 0$ case. This completes the proof.
\end{proof}

\section{Key Properties of Geometric Complexes on a Lattice}
\label{sec:BettiBounds}
Here we establish the main properties (Proposition~\ref{prop:EqCechRips} and Theorems~ \ref{thm:RDistantVertices} and~\ref{thm:RBettiBounds}) of our \Cech and equivalently \Rips complex that we mentioned in Section~\ref{sec:ProofOutline} earlier. We also prove some additional features that we use later in Sections~\ref{sec:IntResults} and~\ref{sec:ConclResults}. While we employ these results afterwards for our random setup from Section~\ref{sec:intro}, we emphasize that the discussion here is completely from a deterministic perspective.

Throughout this section, let $V$ be an arbitrary fixed subset of $\dGrid.$ As in Definitions~\ref{Defn:Cech} and \ref{Defn:Rips} respectively, let $\cK$ be the \Cech and $\hat{\cK}$ be the \Rips complex on $V$ with respect to $\sB_\infty(\cdot, \tfrac{1}{2}).$

Our first aim is to show that $\cK = \hat{\cK}.$ This follows from their definitions and holds since these are defined on a lattice. We begin with two trivial facts.

\begin{fact}
\label{Fact:AbsD1}
If $r_1, r_2 \in \Grid$ with $|r_1 - r_2| = 1,$ then for all $r \in \Grid$ with $\max\{|r - r_1|, |r - r_2|\} \leq 1,$ we have
$r \in \{r_1, r_2\}.$
\end{fact}

\begin{fact}
\label{Fact:AbsD2}
If $r_1, r_2 \in \Grid$ with $|r_1 - r_2| = 2,$ then for all $r \in \Grid$ with $\max\{|r - r_1|, |r - r_2|\} \leq 1,$ we have $r = (r_1 + r_2)/2.$
\end{fact}

\begin{lemma}
\label{lem:EqGenCechRips}
$\cK = \hat{\cK}.$
\end{lemma}
\begin{proof}
By definition, $\sigma \in \cK$ implies that $\sigma \in \hat{\cK}.$ It thus suffices to show only the converse. Clearly, if $\sigma \in \hcK$ and $|\sigma| \in \{1, 2\},$ then $\sigma \in \cK.$ Consider $\sigma \equiv \{\vec{t}_0, \ldots, \vec{t}_k\} \in  \hcK$ for some $k \geq 2.$ As $\sigma \in \hcK,$ $\sB_\infty(\vec{t}_i, 1/2) \cap \sB_\infty(\vec{t}_j, 1/2) \neq \emptyset,$ or equivalently $\|\vec{t}_i - \vec{t}_j\|_{\infty} = 1,$ for each pair $\vec{t}_i, \vec{t}_j.$ Without loss of generality, let $\vec{t}_0 \equiv \vec{0}.$ Then, using Fact~\ref{Fact:AbsD1}, for any coordinate index $l,$ exactly one of the following cases is true:
\begin{itemize}
\item[i)] $\vec{t}_0(l) = \vec{t}_1(l) = \vec{t}_2(l) = \cdots = \vec{t}_k(l) = 0,$ or

\item[ii)] for some $i \in \{1, \ldots, k\},$ $\vec{t}_i(l) = 1;$ hence, $\vec{t}_{i^\prime}(l) \in \{0, 1\}$ for all $i^\prime \in \{1, \ldots, k\} \setminus i,$ or

\item[ii)] for some $i \in \{1, \ldots, k\},$ $\vec{t}_i(l) = -1;$ hence, $\vec{t}_{i^\prime}(l) \in \{-1, 0\}$ for all $i^\prime \in \{1, \ldots, k\} \setminus i.$
\end{itemize}
Now define the vector $\vec{t} \in \dReal$ as follows:
\[
\vec{t}(l) =
\begin{cases}
\frac{1}{2}, & \text{ if Cases i) or ii) hold, }\\
-\frac{1}{2}, & \text{ if Case iii) holds.}
\end{cases}
\]
Trivially, $\vec{t} \in \sB_\infty(\vec{t}_i, 1/2)$ for each $i \in \{0, \ldots, k\}.$ The desired result is now easy to see.
\end{proof}

Precisely due to this equivalence, $\cK$ here and $\cK(n; u)$ from Definition~\ref{Defn:Cech} is both a \Cech complex as well as a \Rips complex.

We next establish the deterministic equivalent of Theorem~\ref{thm:RDistantVertices}.

\begin{theorem}
\label{thm:DistantVertices}
In every induced subcomplex $\cC \subseteq \cK$ with non-trivial $k-$th Betti number, where $0 \leq k < d,$ there exist at least $2k + 2$ vertices with pairwise $\|\cdot\|_1$-distances at least $2.$
\end{theorem}

Before giving its proof, we state and prove two technical lemmas. For any $\vec{t} \in \dGrid$ and $1 \leq j \leq d,$ let $\tNeigh(\vec{t}), \mN_j(\vec{t}),$ and $\pN_j(\vec{t})$ be suitable $\|\cdot\|_\infty-$neighbourhoods of $\vec{t}$ in $\dGrid$ given by
\[
\tNeigh(\vec{t}) := \{\vecp{t} \in \dGrid: \|\vec{t} - \vecp{t}\|_\infty \leq 1\},
\]
\[
\mN_j(\vec{t}) := \{\vecp{t} \in \tNeigh(\vec{t}): \vecp{t}(j) = \vec{t}(j) - 1\}, \text{ and }  \pN_j(\vec{t}) := \{\vecp{t} \in \tNeigh(\vec{t}): \vecp{t}(j) = \vec{t}(j) +1\}.
\]

\begin{lemma}
\label{lem:NeighNonEmpty}
Let $\vec{t} \in \dGrid$ and $j \in \{1, \ldots, d\}.$ Then, the following statements hold:
\begin{itemize}
\item For every induced subcomplex $\cC \subseteq \cK$ such that $\cF_{0}(\cC) \subseteq \tNeigh(\vec{t}) \setminus \mN_j(\vec{t})$ and $\vec{t} + \vec{e}_j \in \cC,$ all its Betti numbers are trivial.

\item For every induced subcomplex $\cC \subseteq \cK$ such that $\cF_{0}(\cC) \subseteq \tNeigh(\vec{t}) \setminus \pN_j(\vec{t})$ and $\vec{t} - \vec{e}_j \in \cC,$
all its Betti numbers are trivial.
\end{itemize}
\end{lemma}
\begin{proof}
The first statement is proved here; the other one follows similarly. We use induction on $k,$ the index of Betti numbers.  For any $\vecp{t} \in \tNeigh(\vec{t}) \setminus \mN_j(\vec{t}),$ we have
\begin{equation}
\label{eqn:EveryoneClose}
\|\vecp{t} - (\vec{t} + \vec{e}_j)\|_\infty \leq 1.
\end{equation}
From this, the $k = 0$ case is easy to see. Now assume the result for some $(k - 1).$ For the sake of contradiction, assume that there exists $\cC$ with $\cF_{0}(\cC) \subseteq \tNeigh(\vec{t}) \setminus \mN_j(\vec{t})$ and $\vec{t} + \vec{e}_j \in \cC$  such that $\beta_k(\cC) \geq 1.$ This then immediately implies that there exists a $k-$NTC in $\cC$ with minimal vertex support. In fact, we claim that there exists a $k-$NTC $\gamma$ in $\cC$ having minimal vertex support such that $\vec{v} \equiv \vec{t} + \vec{e}_j \in \vsupp(\gamma).$ For the time being, suppose that this latter claim is true. Then, for any $\vecp{v} \in \vsupp(\gamma)$ such that $\vecp{v} \neq \vec{v},$ from Lemma~\ref{lem:RelLinkNTC}, $\gamma \cap \lk(\vecp{v})$ is a $(k - 1)-$NTC  in $\cCp := \lk(\vecp{v}) \cap \{\sigma \in \cC: \sigma \subseteq \vsupp(\gamma)\},$ i.e., the $(k - 1)-$th Betti number of $\cCp$ is non-trivial. But $\cCp$ is an induced subcomplex of $\cK$ with $\cF_0(\cCp) \subseteq \tNeigh(\vec{t}) \setminus \mN_j(\vec{t})$ and $\vec{v} \in \cCp,$ where the latter is due to \eqref{eqn:EveryoneClose}. These statements contradict the induction hypothesis and we are done.

It only remains to establish our above claim. Let $\gamma \equiv \sum_{\sigma \in \supp(\gamma)} c_\sigma \sigma$ be a $k-$NTC in $\cC$ with minimal vertex support such that $\vec{v} \notin \vsupp(\gamma),$  where $\supp(\gamma)$ restriction in the summation means that the linear sum representation of $\gamma$ has only those faces $\sigma$ for which $c_{\sigma} \neq 0.$ Pick $\vecp{v} \in \vsupp(\gamma);$ clearly, $\vecp{v} \neq \vec{v}.$ Consider the $k-$chain
\[
\gamma_1 =  \gamma \cap \st(\vecp{v}) = \sum_{\sigma \in \st(\vecp{v})} c_{\sigma} \sigma,
\]
and let
\[
\gamma_1 * \{\vec{v}\} := \sum_{\sigma \in \st(\vecp{v})} c_{\sigma} ( \sigma, \vec{v} ),
\]
where $(\sigma, \vec{v})$ is the oriented $(k + 1)-$face whose first $(k + 1)$ vertices are those of $\sigma$ and in the same order as in $\sigma,$ and the last vertex is $\vec{v};$ and $\gamma \cap \st(\vecp{v})$ is defined as in Remark~\ref{rem:RelLinkNTC}. As $\cC$ is an induced complex, $\gamma_1 * \{\vec{v}\}$ is necessarily a $(k + 1)-$chain in $\cC.$ With $(\gamma \cap \lk(\vecp{v})) * \{v\}$ defined as in the spirit of $\gamma_1 * \{\vec{v}\},$ and using the relation $\gamma \cap \lk(\vecp{v}) = \bdr_k(\gamma_1)$ from Remark~\ref{rem:RelLinkNTC}, it is easy to see that
\[
\bdr_{k + 1}(\gamma_1 * \{\vec{v}\}) = (\gamma \cap \lk(\vecp{v})) * \{v\} + (-1)^{k +  1} \gamma_1.
\]
Let $\gamma_2 := \gamma - (-1)^{k + 1} \bdr_{k + 1}(\gammap * \{\vec{v}\}).$ Then, as $\bdr_{k} \circ \bdr_{k + 1} = 0,$  $\gamma_2 \in [\gamma].$ From Remark~\ref{rem:RelLinkNTC}, $\vecp{v} \notin \vsupp(\gamma_2),$ while, from the definition of $\gamma_2,$ $\vec{v} \in \vsupp(\gamma_2)$ and $\vsupp(\gamma_2) \setminus \{\vec{v}\} \subseteq \vsupp(\gamma).$ Consequently, because $\gamma$ is a $k-$NTC with minimal vertex support, $\gamma_2$ is also a $k-$NTC with minimal vertex support. Since $\vec{v} \in \vsupp(\gamma_2),$ our claim follows.
\end{proof}

\begin{lemma}
\label{lem:MissingFrontBackVertex}
Let $j \in \{1, \ldots, d\}$ and $k$ be such that $1 \leq k < d.$ Let $\gamma$ be a $k-$NTC in $\cK$ with minimal vertex support. Then, the following statements hold:
\begin{itemize}
\item If $\vec{t} \in \vsupp(\gamma)$ and $\mN_j(\vec{t}) \cap \vsupp(\gamma) = \emptyset,$ then $\vec{t} + \vec{e}_j \notin \vsupp(\gamma).$

\item If $\vec{t} \in \vsupp(\gamma)$ and $\pN_j(\vec{t}) \cap \vsupp(\gamma) = \emptyset,$ then $\vec{t} - \vec{e}_j \notin \vsupp(\gamma).$
\end{itemize}
\end{lemma}
\begin{proof}
We only prove the first statement here; the other one follows similarly. From Lemma~\ref{lem:RelLinkNTC}, $\gamma \cap \lk(\vec{t})$ is a $(k - 1)-$NTC in $\cC:= \lk(\vec{t}) \cap \{\sigma \in \cK: \sigma \subseteq \vsupp(\gamma)\},$ i.e., the $(k - 1)-$th Betti number of $\cC$ is non-trivial. Also, $\mN_j(\vec{t}) \cap \vsupp(\gamma) = \emptyset$ implies that $\cF_0(\cC) \subseteq \tNeigh(\vec{t}) \setminus \mN_j(\vec{t}).$ The desired result now follows from Lemma~\ref{lem:NeighNonEmpty}.
\end{proof}

We are now ready to prove Theorem~\ref{thm:DistantVertices}.

\begin{proof}[Proof of Theorem~\ref{thm:DistantVertices}]
We use induction on $k$. The result is trivially true for $k = 0.$ Suppose the result holds for some $k - 1.$ Now consider an induced subcomplex $\cC \subseteq \cK$ such that $\beta_k(C) \geq 1.$ For the sake of contradiction, let the assumption below hold.
\[
\Cl[Asm]{a:pwDist} : \text{every $2k + 2$ vertices in $\cC$ have a pair $\vec{t}, \vecp{t}$ such that $\|\vec{t} - \vecp{t}\|_1 = 1.$}
\]

Since $\beta_{k}(\cC) \geq 1,$ there exists at least one $k-$NTC in $\cC$ with minimal vertex support. Below we show that, under assumption $\Cr{a:pwDist},$ no such cycle can exist. This gives the desired contradiction.

Let $\gamma$ be an arbitrary $k-$NTC with minimal vertex support in $\cC.$ From Lemma~\ref{lem:CardNTC}, $|\vsupp(\gamma)| \geq 2k + 2.$ By Assumption $\Cr{a:pwDist}$ above, every $2k + 2$ vertices in $\vsupp(\gamma)$ have a pair $\vec{t}, \vecp{t}$ such that $\|\vec{t} - \vecp{t}\|_1 = 1;$ without loss of generality, let $\vec{0}, \vec{e}_1 \in \vsupp(\gamma).$  Since $\gamma$ is a $k-$NTC of minimal vertex support and since $\vec{0}, \vec{e}_1 \in \vsupp(\gamma),$ it follows from Lemma~\ref{lem:MissingFrontBackVertex} that $\mN_{1}(\vec{0}) \cap \vsupp(\gamma) \neq \emptyset$ and $\pN_{1}(\vec{e}_1) \cap \vsupp(\gamma) \neq \emptyset.$ Let $\mv, \pv$ respectively lie in $\mN_{1}(\vec{0}) \cap \vsupp(\gamma)$ and $\pN_{1}(\vec{0}) \cap \vsupp(\gamma).$

Through a series of claims, we first show that Assumption $\Cr{a:pwDist}$ forces all vertices in $\vsupp(\gamma)$ to be sufficiently close to either $\vec{0}$ or $\vec{e}_1.$ We then show that this violates Lemma~\ref{lem:NeighNonEmpty}. Since $\gamma$ is arbitrary, the desired result follows.

$\textbf{Claim}_1:$ For all $\vec{t} \in \vsupp(\gamma),$ $-1 \leq \vec{t}(1) \leq 2.$

For sake of contradiction, suppose that there exists $\vec{t} \in \vsupp(\gamma)$ with $\vec{t}(1) \geq 3.$ From Lemma~\ref{lem:RelLinkNTC}, $\gamma \cap \lk(\vec{t})$ is a $(k - 1)-$NTC. Hence, by induction hypothesis, there are $2k$ vertices in $\tNeigh(\vec{t}) \cap \vsupp(\gamma)$ such that their pairwise $\|\cdot\|_1-$distances are $\geq 2.$ Similarly, it follows that there are $2k$ vertices in $\tNeigh(\mv) \cap \vsupp(\gamma)$ with pairwise $\|\cdot\|_1-$distances $\geq 2.$ But then we have a contradiction to Assumption $\Cr{a:pwDist}$ above. By a symmetric argument, one can obtain a contradiction if $\exists \vec{t} \in \vsupp(\gamma)$ with $\vec{t}(1) \leq -2.$ Thus, the above claim follows.

$\textbf{Claim}_2:$ For all $\vec{t} \in \vsupp(\gamma),$ $|\vec{t}(j)| \leq 2$ for each $j \neq 1.$

To see this, suppose that there exists a $\vec{t} \in \vsupp(\gamma)$ with $|\vec{t}(j)| \geq 3$  for  some $j \neq 1.$ Then, from Lemma~\ref{lem:RelLinkNTC} and the induction hypothesis, there exist $2k$ vertices, say $\vec{t}_1, \ldots, \vec{t}_{2k},$ in $\tNeigh(\vec{t}) \cap \vsupp(\gamma)$ with pairwise $\|\cdot\|_1-$distances $\geq 2.$ If $\vec{t}(1) \geq 1,$ then it is easy see that $\{\vec{t}_1, \ldots, \vec{t}_{2k}, \mv, \vec{e}_1\}$ is a set of $2k + 2$ vertices with pairwise $\|\cdot\|_1-$distances $\geq 2.$ On the other hand, if $\vec{t}(1) \leq 0,$ then $\{\vec{t}_1, \ldots, \vec{t}_{2k}, \vec{0}, \pv\}$ is a set of $2k + 2$ vertices with pairwise $\|\cdot\|_1-$distances $\geq 2.$ Both contradict assumption $\Cr{a:pwDist}.$ The above claim follows.

In fact, we can improve upon $\textbf{Claim}_2$ as shown below.

$\textbf{Claim}_3:$ For all $\vec{t} \in \vsupp(\gamma),$ $|\vec{t}(j)| \leq 1$ for each $j \neq 1.$

Again, for sake of contradiction, assume there exists a $\vec{t} \in \supp(\gamma)$ with $|\vec{t}(j)| = 2$ for some $j \geq 1.$ For ease of exposition, first assume $\vec{t}(j) = 2$ and $\vec{t}(1) \geq 1.$ From Lemma~\ref{lem:RelLinkNTC} and the induction hypothesis, there are $2k$ vertices, say $\vec{t}_1, \ldots, \vec{t}_{2k},$ in $\tNeigh(\vec{t}) \cap \vsupp(\gamma)$ with pairwise $\|\cdot\|_1-$distances $\geq 2.$ By similar arguments as for $\textbf{Claim}_1,$ it now follows that for each $\vec{v} \in \vsupp(\gamma),$ $\vec{v}(j) \geq -1.$ In fact, we now show that it cannot even be $-1;$ not even in $\mN_1(\vec{0})$ or $\pN_1(\vec{e}_1).$ To see this, suppose there exists a $\vec{v} \in \vsupp(\gamma)$ such that $\vec{v}(j) = -1.$
\begin{itemize}[leftmargin=*]
\item Subcase $\vec{v}(1) \geq 1.$ From $\textbf{Claim}_1,$ we have $\mN_1(\mv) \cap \vsupp(\gamma) = \emptyset.$ Hence, from Lemma~\ref{lem:MissingFrontBackVertex}, $\mv + \vec{e}_1 \notin \vsupp(\gamma).$ Combining this with $\vec{t}(1) \geq 1,$ it follows that $\|\mv - \vec{t}_i\|_1 \geq 2$ for all $i \in \{1, \ldots, 2k\}.$ Also, since $\vec{t}(j) = 2$ and $\vec{v}(j) = -1,$ $\|\vec{v} - \vec{t}_i\|_1 \geq 2$ for all $i \in \{1, \ldots, 2k\}.$ As $\mv(1) = -1$ and $\vec{v}(1) \geq 1,$ we further have $\|\mv - \vec{v}\|_1 \geq 2.$ In other words, $\{\vec{t}_1, \ldots, \vec{t}_{2k}, \mv, \vec{v}\}$ is a set of $2k + 2$ vertices with pairwise $\|\cdot\|_1-$distances $\geq 2.$ This contradicts $\Cr{a:pwDist}$ above.

\item Subcase $\vec{v}(1) \leq 0.$ Again, from Lemma~\ref{lem:RelLinkNTC} and the induction hypothesis, there are $2k$ vertices, say $\vec{v}_1, \ldots, \vec{v}_{2k},$ in $\tNeigh(\vec{v}) \cap \vsupp(\gamma)$ with pairwise $\|\cdot\|_1-$distances $\geq 2.$ From $\textbf{Claim}_1,$ we have $\pN_1(\pv) \cap \vsupp(\gamma) = \emptyset.$ Hence, from Lemma~\ref{lem:MissingFrontBackVertex}, $\pv - \vec{e}_1 \notin \vsupp(\gamma).$ Combining this with $\vec{v}(1) \leq 0,$ we have $\|\pv - \vec{v}_i\|_1 \geq 2$ for all $i \in \{1, \ldots, 2k\}.$ As $\vec{v}(j) = -1$ and $\vec{t}(j) = 2,$ we additionally have $\|\vec{t} - \vec{v}_i\|_1 \geq 2$ for all $i \in \{1, \ldots, 2k\}.$ Separately, from $\textbf{Claim}_2,$ we have $\pN_j(\vec{t}) \cap \vsupp(\gamma) = \emptyset.$ Hence, from Lemma~\ref{lem:MissingFrontBackVertex}, $\vec{t} - \vec{e}_j \notin \vsupp(\gamma).$ This necessarily implies that $\|\vec{t} - \pv\|_1 \geq 2.$ Putting all the above together, it follows that $\{\vec{v}_1, \ldots, \vec{v}_{2k}, \pv, \vec{t}\}$ is a set of $2k + 2$ vertices with pairwise $\|\cdot\|_1-$distances $\geq 2.$ This contradicts Assumption $\Cr{a:pwDist}$ above.
\end{itemize}
Hence, it follows that there cannot exist $\vec{v} \in \vsupp(\gamma)$ such that $\vec{v}(j) = -1.$ In particular, $\mN_j(\vec{e}_1) \cap \vsupp(\gamma) = \emptyset.$ But from Lemma~\ref{lem:MissingFrontBackVertex}, we then immediately have $\vec{e}_1 + \vec{e}_j \notin \vsupp(\gamma).$ This means that $\|\vec{e}_1 - \vec{t}_i\|_1 \geq 2$ for all $i \in \{1, \ldots, 2k\}.$ Consequently, as in the first subcase above, $\{\vec{t}_1, \ldots, \vec{t}_{2k}, \mv, \vec{e}_1\}$ is a set of $2k + 2$ vertices with pairwise $\|\cdot\|_1-$distances $\geq 2.$ This contradicts Assumption $\Cr{a:pwDist}$ again. So there exists no $\vec{t} \in \vsupp(\gamma)$ such that $\vec{t}(1) \geq 1$ and $\vec{t}(j) = 2$ for some $j \neq 1.$ Similarly, by symmetric arguments, one can show that there cannot exist any $\vec{t} \in \vsupp(\gamma)$ such that $|\vec{t}(j)| = 2$ for some $j \neq 1.$ The desired claim now follows.

Combining the three claims above, it follows that all vertices in $\vsupp(\gamma)$ must be sufficiently close to either $\vec{0}$ or $\vec{e}_1;$ i.e., for all $\vec{t} \in \vsupp(\gamma),$  either $\|\vec{t} - \vec{0}\|_\infty \leq 1$ or $\|\vec{t} - \vec{e}_1\|_\infty \leq 1.$ Consider $\mv$ as defined earlier. From Lemma~\ref{lem:RelLinkNTC}, $\gamma \cap \lk(\mv)$ is a $(k - 1)$-NTC in $\cCp:= \lk(\mv) \cap \{\sigma \in \cC: \sigma \subseteq \vsupp(\gamma)\},$ i.e., $\beta_{k - 1}(\cCp) \geq 1.$ But observe that, since all vertices in $\vsupp(\gamma)$ are sufficiently close to either $\vec{0}$ or $\vec{e}_1$ as described above, $\cF_0(\cCp) \subseteq \tNeigh(-\vec{e}_1) \setminus \mN_1(-\vec{e}_1).$ Further, since $\cC$ is an induced subcomplex of $\cK,$ it follows that $\vec{0} \in \cCp.$ Together, these contradict Lemma~\ref{lem:NeighNonEmpty}. Thus, $\gamma$ cannot exist as desired.
\end{proof}

We next prove the deterministic variant of Theorem~\ref{thm:RBettiBounds}.
Let $\cSk(\cK),$ $\cNk(\cK),$ and $\cLk(\cK)$ be the respective deterministic analogues of $\cSnk(u),$ $\cNnk(u),$ and $\cLnk(u)$ from Section~\ref{sec:ProofOutline} obtained by replacing $\cK(n; u)$ there with $\cK.$ The statement of Theorem~\ref{thm:RBettiBounds} then translates to the following.

\begin{theorem}
\label{thm:BettiBounds}
For $0 \leq k < d,$
\[
\cSk(\cK) \leq \beta_k(\cK) \leq \cSk(\cK) + \cNk(\cK) + {\Cr{c:Betti}}_{, k} \, \cLk(\cK).
\]
\end{theorem}
\begin{proof}
The $k = 0$ case is trivially true. So suppose that $1 \leq k < d.$

Because $\cK$ is also a \Rips complex, the induced subcomplex associated with a subgraph component that is isomorphic to $O_k$ contributes exactly $1$ to the  $k-$th Betti number. From this, the lower bound is easy to see.

It now only remains to show the upper bound. Suppose that $k \geq 2.$ As discussed above Lemma~\ref{lem:CardNTC}, $\beta_k(\cK) = \sum_{\cC} \beta_k(\cC),$ where the sum is over the isolated induced subcomplexes of $\cK.$ Using Lemma~\ref{lem:CardNTC} and the standard inequality $\beta_k(\cC) \leq |\cF_k (\cC)|$ from the simplicial homology theory, we then have
\[
\beta_k(\cK) \leq \cSk(\cK) + \cNk(\cK) + \sum_{\cC: |\cF_0(\cC)| \geq 2k + 3} |\cF_k(\cC)| \; \indc[\beta_k(\cC) \geq 1].
\]
Consider the $1-$skeleton of an arbitrary $\cC$ with $|\cF_0(\cC)| \geq 2k + 3$ and $\beta_k(\cC) \geq 1.$ Using Theorem~\ref{thm:DistantVertices}, in this skeleton there exists a tree $G_{\cC}$ such that: (a) it has at least $2k + 3$ vertices, (b) it has a further subset of $2k + 2$ vertices having pairwise $\|\cdot\|_1-$distances at least two, and (c) its $\|\cdot\|_\infty-$diameter is at most $4k +3.$ The vertex properties in (a) and (b) trivially hold in a tree with diameter $4k + 3,$ justifying the $4k + 3$ limit. Now let $\sigma$ be any $k-$face in $\cC.$ Clearly, there exists a subgraph in $\cC$ connecting $\vsupp(\sigma)$ to $G_{\cC}.$ Depending on the $\|\cdot\|_\infty-$distance of $\vsupp(\sigma)$ from $G_{\cC},$ and removing edges if necessary, it is easy to see that there exists a subgraph isomorphic to one in $\cP_k$ such that $\vsupp(\sigma)$ forms the only  $(k + 1)-$sized clique in this subgraph. Therefore, $|\cF_k(\cC)| \indc[\beta_k(\cC) \geq 1] \leq \cLk(\cC).$ The desired result now follows.

By making suitable changes to the above argument, in the exact same spirit as discussed below \cite[Fig. 3]{kahle2011random}, the desired result follows for the $k = 1$ case as well.
\end{proof}





The rest of the section describes some useful properties concerning subgraph components in the $1-$skeleton of $\cK$ that are isomorphic to $O_k.$ Fix $k \geq 1$ and let $V \equiv \{\vec{t}_1, \ldots, \vec{t}_{2k + 2}\} \subseteq \dGrid$ be such that the geometric graph $G$ on these vertices, with respect to $\sB_\infty(\cdot, \tfrac{1}{2}),$ is isomorphic to $O_k.$

The following result is a consequence of Theorem~\ref{thm:DistantVertices} and the fact that the induced subcomplex of $\cK$ associated with $G$ has non-trivial $k-$th Betti number.

\begin{corollary}
\label{cor:DistVertMC}
For any $i_1, i_2 \in \{1, \ldots, 2k + 2\}$ with $i_1 \neq i_2,$ $\|\vec{t}_{i_1} - \vec{t}_{i_2}\|_1  \geq 2.$
\end{corollary}

Observe that, for every vertex in $O_k,$ there is precisely one other vertex to which it is not connected to. The following result shows that the $\|\cdot\|_\infty-$distance between the corresponding vertex pairs in $G$ is precisely $2.$

\begin{lemma}
\label{lem:LInf2}
Fix an arbitrary $i \in \{1, \ldots, 2k + 2\}.$ Then, there exists a unique $i^\prime$ such that $\|\vec{t}_i - \vec{t}_{i^\prime}\|_\infty \geq 2;$ in fact, $\|\vec{t}_i - \vec{t}_{i^\prime}\|_\infty = 2.$
\end{lemma}
\begin{proof}
Fix an arbitrary $i.$ Because $G$ is isomorphic to $O_k,$ clearly, there exists a unique $i^\prime$ such that $\|\vec{t}_i - \vec{t}_{i^\prime}\|_\infty \geq 2.$ However, for each $j \neq i, i^\prime,$ we additionally have $\|\vec{t}_j - \vec{t}_i\|_\infty = \|\vec{t}_j - \vec{t}_{i^\prime}\|_\infty = 1.$ But the latter condition can hold only if $\|\vec{t}_i - \vec{t}_{i^\prime}\|_\infty = 2.$ The desired result thus follows.
\end{proof}
From the above result and the definition of $G,$ the below statement is immediate.

\begin{corollary}
\label{cor:LInfIsometry}
$G$ is $\|\cdot\|_\infty-$isometric to $O_k.$
\end{corollary}

For a general $k,$ the graph $G$ can either be $\|\cdot\|_1-$isometric to $O_k$ or not. We now show that, if $k = d - 1,$ then it has to be.

\begin{lemma}
\label{lem:AllwaysIsoCase}
If $k = d - 1,$ then $G$ is also $\|\cdot\|_1-$isometric to $O_{d - 1}.$
\end{lemma}
\begin{proof}
First, observe that $V$ has $2d$ vertices. Separately, from Corollary~\ref{cor:LInfIsometry}, we have that, for each $i,$ there is a unique $i^\prime$ such that $\|\vec{t}_i - \vec{t}_{i^\prime}\|_\infty = 2.$ Hence, $V$ can be written as $\{\vec{t}_{1}, \vec{t}_{1^\prime}, \ldots, \vec{t}_{d}, \vec{t}_{d^\prime}\}.$ Let $\cI_i$ be the set of coordinate indices in which $\vec{t}_i$ and $\vec{t}_{i^\prime}$ differ by $2;$ this is clearly non-empty.

Fix some arbitrary $i, j.$ From Corollary~\ref{cor:LInfIsometry},
$\|\vec{t}_{j_1} - \vec{t}_{j_2}\|_\infty = 1$ for $j_1 \in \{j, j^\prime\}$ and $j_2 \in \{i, i^\prime\}.$ From Fact~\ref{Fact:AbsD2}, it then follows that, for each $l \in I_i,$ $\vec{t}_{j}(l) = \vec{t}_{j^\prime}(l) = (\vec{t}_{i}(l) + \vec{t}_{i^\prime}(l))/2,$ and viceversa. Consequently, $\cI_i \cap \cI_j = \emptyset.$  As $i,j$ were arbitrary, and since there are only $d$ distinct coordinate indices, the desired result is easy to see.
\end{proof}

The next result states that, for general $k,$ if $G$ is indeed $\|\cdot\|_1-$isometric to $O_k,$ then its vertices have a unique ordered representation.

\begin{lemma}
\label{lem:CentrePtSP}
$G$ is $\|\cdot\|_1-$isometric to $O_k$ if and only if there is a unique $\vec{t} \in \dGrid$ and a unique $\vec{\Omega} \equiv (\omega_1, \ldots, \omega_{k + 1}) \in \tbinom{\setd}{b_k}$ such that the vertices of $G$ are, up to permutation,
\[
(\vec{t} + \vec{e}_{\omega_1}, \vec{t} - \vec{e}_{\omega_1}, \ldots, \vec{t} + \vec{e}_{\omega_{k + 1}}, \vec{t} - \vec{e}_{\omega_{k + 1}}).
\]
\end{lemma}
\begin{proof}
We only show the necessary part of the statement as the other direction is obvious. Suppose $G$ is $\|\cdot\|_1-$isometric to $O_k.$ Then, using Corollary~\ref{cor:LInfIsometry}, we have that, for any $i,$ there exists a unique $i^\prime \neq i$ such that $\|\vec{t}_i - \vec{t}_{i^\prime}\|_\infty = \|\vec{t}_i - \vec{t}_{i^\prime}\|_1 =  2;$ further, for each $j_1 \in \{i, i^\prime\}$ and $j_2 \notin \{i, i^\prime\},$ $\|\vec{t}_{j_1} - \vec{t}_{j_2}\|_\infty = 1$ and $\|\vec{t}_{j_1} - \vec{t}_{j_2}\|_1 = 2.$ From these observations and Fact~\ref{Fact:AbsD2}, it is easy to see that the desired result holds with $\vec{t} = (\sum_{i = 1}^{2k + 2} \vec{t}_i)/(2k + 2).$
\end{proof}

Suppose $k < d - 1.$ Then, it is possible that $G$ is only isomorphic but not $\|\cdot\|_1-$isometric to $O_k.$ The next result show that, while there is no unique representation, there does exist a common neighbour to all its vertices with respect to the $\|\cdot\|_\infty-$distance.

\begin{lemma}
\label{lem:CentrePtNSP}
Suppose $G$ is not $\|\cdot\|_1-$isometric to $O_k.$ Then, there is a $\vec{t} \in \dGrid$ such that $\|\vec{t} - \vec{t}_i\|_\infty = 1$ for all $j.$ \end{lemma}
\begin{proof}
Fix an arbitrary $i.$ It follows from Corollary~\ref{cor:LInfIsometry} that there exists a unique $i^\prime \neq i$ such that $\|\vec{t}_i - \vec{t}_{i^\prime}\|_\infty = 2.$ Now define $\vec{t}$ using the following rule: for each $l \in \setd,$
\[
\vec{t}(l)  :=
\begin{cases}
(\vec{t}_i(l) + \vec{t}_{i^\prime}(l))/2, & \text{ if $|\vec{t}_i(l) - \vec{t}_{i^\prime}(l)| = 2,$}\\
\vec{t}_i(l), & \text{ otherwise. }
\end{cases}
\]
Keeping in mind $\|\vec{t}_i - \vec{t}_{i^\prime}\|_\infty = 2,$ it is easy to see that $\|\vec{t} - \vec{t}_i\|_\infty = \|\vec{t} - \vec{t}_{i^\prime}\|_\infty = 1.$

Now consider any arbitrary $j \notin \{i, i^\prime\}.$ Again, from Corollary~\ref{cor:LInfIsometry}, we have that $\|\vec{t}_{j} - \vec{t}_{j^\prime}\|_\infty = 1$ for all $j^\prime \in \{i, i^\prime\}.$ From these observations and Fact~\ref{Fact:AbsD2}, it is easy to see that $\|\vec{t}_{j} - \vec{t}\|_\infty = 1.$

Since $j$ is arbitrary, the desired result follows.
\end{proof}

Whenever $G$ is only isomorphic but not $\|\cdot\|_1-$isometric to $O_k,$ our final result here shows that the vertex set of $G$ can be partitioned into two non-empty subsets such that the pairwise $\|\cdot\|_1-$distances between vertices in the two subsets is at least $3.$

\begin{lemma}
\label{lem:PwL1Dist}
Suppose $G$ is not $\|\cdot\|_1-$isometric to $O_k.$ Then, its vertex set $V$ can be partitioned into non-empty subsets $V_1$ and $V_2$ so that $1 \leq |V_1| \leq k + 1;$ further, for every $\vec{t}_{j_1} \in V_1$ and $\vec{t}_{j_2} \in V_2,$ $\|\vec{t}_{j_1} - \vec{t}_{j_2}\|_1 \geq 3.$
\end{lemma}
\begin{proof}
From Corollary~\ref{cor:DistVertMC} and the fact that $G$ is not $\|\cdot\|_1-$isometric to $O_k,$ it follows that there exists a pair $i_1, i_2$ such that $\|\vec{t}_{i_1} - \vec{t}_{i_2}\|_1 \geq 3.$ With such a pair, partition $V$ into two subsets $V_1$ and $V_2$ where
\[
V_1 := \{\vec{t}_j : \|\vec{t}_{i_1} - \vec{t}_j\|_1 \leq 2\} \text{ and }  V_2 := V \setminus V_1.
\]
Clearly, $V_1$ contains $\vec{t}_{i_1}$ and $V_2$ contains $\vec{t}_{i_2};$ so, both are non-empty. We now show that for $\vec{t}_{j_1} \in V_1$ and $\vec{t}_{j_2} \in V_2$
\[
\|\vec{t}_{j_1} - \vec{t}_{j_2}\|_1 \geq 3.
\]

Fix $\vec{t}_{j_1} \in V_1$ and $\vec{t}_{j_2} \in V_2.$ From Corollary~\ref{cor:LInfIsometry}, either %
\[
\|\vec{t}_{j_1} - \vec{t}_{j_2}\|_\infty = 2 \text{ or } \|\vec{t}_{j_1} - \vec{t}_{j_2}\|_\infty = 1.
\]
Keeping this in mind, we break our arguments into several cases:
\begin{enumerate}[leftmargin=*]
\item $j_1 = i_1:$ Here the claim follows from the definition of $V_1.$

\vspace{2ex}

\item $j_1 \neq i_1$ and $\|\vec{t}_{j_1} - \vec{t}_{j_2}\|_\infty = 2:$ Clearly, there exists a coordinate index $l$ such that $|\vec{t}_{j_1}(l) - \vec{t}_{j_2}(l)| = 2.$ If there is more than one such coordinate index, then we are already done. So, consider the case when there is precisely one such $l.$ It suffices to show that there is a coordinate index, other than $l,$ where $\vec{t}_{j_1}$ and $\vec{t}_{j_2}$ differ by at least $1.$ Due to Corollary~\ref{cor:LInfIsometry}, it follows that $\|\vec{t}_{i_1} - \vec{t}_{j_1}\|_\infty = \|\vec{t}_{i_1} - \vec{t}_{j_2}\|_\infty = 1.$ Hence, using Fact~\ref{Fact:AbsD2}, $|\vec{t}_{i_1}(l) - \vec{t}_{j_1}(l)| = |\vec{t}_{i_1}(l) - \vec{t}_{j_2}(l)| = 1.$ From the definition of $V_1,$ we also have $\|\vec{t}_{i_1} -\vec{t}_{j_2}\|_1 \geq 3.$ Further, by combining the definition of $V_1$ and  Corollary~\ref{cor:DistVertMC}, we have $\|\vec{t}_{i_1} -\vec{t}_{j_1}\|_1 = 2.$ These observations imply that, other than $l,$ $\vec{t}_{j_2}$ must differ from $\vec{t}_{i_1}$ in at least two other coordinates, and the value of the difference must be exactly $1;$ while $\vec{t}_{j_1}$ must vary from $\vec{t}_{i_1}$ in only one other coordinate, with the difference being exactly one again. The desired claim is now easy to see.

\vspace{2ex}

\item $j_1 \neq i_1$ and $\|\vec{t}_{j_1} - \vec{t}_{j_2}\|_\infty = 1:$ From Lemma~\ref{lem:LInf2}, there exists a unique $j_1^\prime$ such that $\|\vec{t}_{j_1} - \vec{t}_{j_1^\prime}\|_\infty = 2.$ The subsequent arguments are broken into further subcases.

\begin{enumerate}

\item  $j_1^\prime = i_1:$ From the definition of $V_1$ and Corollary~\ref{cor:DistVertMC}, $\|\vec{t}_{i_1} - \vec{t}_{j_1}\|_1 = 2.$ This implies there is precisely one coordinate index $l$ in which $\vec{t}_{i_1}$ and $\vec{t}_{j_1}$ differ and that $|\vec{t}_{i_1}(l) - \vec{t}_{j_1}(l)| = 2.$ Due to Corollary~\ref{cor:LInfIsometry}, $\|\vec{t}_{j_1} - \vec{t}_{j_2}\|_\infty = \|\vec{t}_{i_1} - \vec{t}_{j_2}\|_\infty = 1.$ But, from the definition of $V_1,$ $\|\vec{t}_{i_1} - \vec{t}_{j_2}\|_1 \geq 3.$ Combining these observations with Fact~\ref{Fact:AbsD2}, the desired claim is now easy to see.

\item $j_1^\prime \neq i_1 :$ Again, from Corollary~\ref{cor:LInfIsometry}, $\|\vec{t}_{\ell_1} - \vec{t}_{\ell_2}\|_\infty = 1$ for each $\ell_1 \in \{j_1, j_1^\prime\}$ and $\ell_2 \in \{i_1, j_2\}.$  Along with Fact~\ref{Fact:AbsD2}, the above implies that, for each coordinate index $l$ where $|\vec{t}_{j_1}(l) - \vec{t}_{j_1^\prime}(l)| = 2,$ we have $|\vec{t}_{i_1}(l) - \vec{t}_{j_2}(l) | = 0$ and $|\vec{t}_{j_1}(l) -  \vec{t}_{j_2}(l) | = 1.$ Let $\cI$ be the collection of all such $l.$ Now fix an $l \in \cI.$ From the definition of $V_1$ and Corollary~\ref{cor:DistVertMC}, $\|\vec{t}_{i_1} - \vec{t}_{j_1}\|_1 = 2.$ Therefore, $\vec{t}_{i_1}$ and $\vec{t}_{j_2}$ differ in only one coordinate, other than $l,$ and that difference is precisely $1.$ But, from the definition of $V_1,$ $\|\vec{t}_{i_1} - \vec{t}_{j_2}\|_1 \geq 3.$ Now, whether $\|\vec{t}_{i_1} - \vec{t}_{j_2}\|_\infty$ equals $1$ or $2,$ the desired result is easy to see.
\end{enumerate}

\end{enumerate}

Since $V_1$ and $V_2$ are non-empty and $|V_1| + |V_2| = 2k + 2,$ it is easy to see that $1 \leq \min\{|V_1|, |V_2|\} \leq k + 1.$ Hence, by using the above arguments and interchanging the labels of $V_1$ and $V_2$ if necessary, we have the desired result.
\end{proof}

\section{Asymptotic Behavior of Betti Approximators}
\label{sec:IntResults}

The intermediate results concerning the asymptotic behaviour of Betti approximators from Section~\ref{sec:ProofOutline} are proved here. Often, some technical results are needed. Their proofs are given either in Section~\ref{sec:ConclResults} or the Appendix.

Using \eqref{Defn:Snk}, \eqref{Defn:Nnk}, \eqref{Defn:Lnk}, and Lemma~\ref{lem:CentrePtSP}, we first make the following observations.

\begin{remark}
\label{rem:SnkImp}
For $k = 0,$ $S_{n, 0}$ is the number of vertices in $\cK(n; u).$ On the other hand, for $k \geq 1,$ $\Snk(u)$ is roughly the number of subgraphs in $\cK(n; u)$ which are isomorphic and $\|\cdot\|_1-$isometric to $O_k.$ We say `roughly' because $\Snk(u)$ counts additional subgraphs at the boundary of $\Idx_n$ which cannot be part of $\cK(n; u).$ Separately, and as remarked earlier in Section~\ref{sec:ProofOutline}, note that $\Snk(u)$ does not require that these subgraphs be components themselves.
\end{remark}

\begin{remark}
\label{rem:NnkImp}
For $k \geq 1,$ $\Nnk(u)$ overcounts (by a constant factor) the number of subgraphs in $\cK(n; u)$ that are isomorphic but not $\|\cdot\|_1-$isometric to $O_k.$ The overcounting is due to the two facts: (i) the same subgraph is counted for more than one $\vec{t},$ and (ii) extra subgraphs are counted at the boundary of $\Idx_n$ that cannot be part of $\cK(n; u)$. But this does not matter for us, since it will only add a constant factor in the associated estimates.
\end{remark}

\begin{remark}
\label{rem:LnkImp}
In the same spirit as in Remark~\ref{rem:NnkImp}, $L_{n, k}(u)$ overestimates the number of subgraphs in $\cK(n; u)$ that is isomorphic to a graph in $\cP_k.$ In particular, we note that $L_{n, 0}(u)$ is at least twice the number of edges in $\cK(n; u).$
\end{remark}

Separately, using \eqref{Defn:cNk}  and Lemma~\ref{lem:PwL1Dist}, we note the following.

\begin{remark}
\label{rem:NonUniPartition}
For each $G \equiv (V, E) \in \cN_k,$ Lemma~\ref{lem:PwL1Dist} gives a partition of $V$ with certain properties. It is not difficult to see that such a partition is not unique. For convenience, however, we shall henceforth assume that each $G$ is uniquely associated with one such partition.
\end{remark}
%

We are now ready to prove Lemma~\ref{lem:BettiApprox}.

\begin{proof}[Proof of Lemma~\ref{lem:BettiApprox}]\label{Pf:BettiApprox}
We prove each statement separately.
\begin{itemize}[leftmargin=*]
\item Arguments for Statement~\ref{itm:BdSnku}: Consider the $k = 0$ case first. From Remark~\ref{rem:SnkImp} above, the lower bound follows trivially. It remains to show the upper bound. Note that a vertex in $\cK(n; u)$ can either be isolated or not. Separately, given a graph, recall that a connected component with $i$ vertices must have at least  $i - 1$ edges within it. Therefore, the number of vertices in a connected component, with at least two vertices, is bounded from above by twice the number of edges in it. From these observations and Remark~\ref{rem:LnkImp}, the upper bound is easy to see.

Now suppose that $1 \leq k < d.$ For $\tOm \in \dGridCrDir,$ let $\GtOm$ be the geometric graph on the vertex set $\{\vec{t} \pm \vec{e}_{\omega_1}, \ldots, \vec{t} \pm \vec{e}_{\omega_{b_k}}\}$ with respect to $\sB_\infty(\cdot, \tfrac{1}{2}).$ Then, from \eqref{Defn:Snk}, we have

\begin{multline}
\label{eqn:SnkDecom}
\Snk(u) =  \sum_{\tOm \in \IdxnCrDir} \indc[\vecXtOm \geq u \ones_{2b_k}] \; \indc[\GtOm \text{ forms a component}]  \\[2ex]
+  \sum_{\tOm \in \IdxnCrDir}  \indc[\vecXtOm \geq u \ones_{2b_k}] \; \indc[\GtOm \text{ does not form a component}].
\end{multline}
Therefore, it follows from Remark~\ref{rem:SnkImp} and the definition of $\cSnk(u)$ that
\[
\cSnk(u) \leq \sum_{\tOm \in \IdxnCrDir} \indc[\vecXtOm \geq u \ones_{2b_k}] \; \indc[\GtOm \text{ forms a component}].
\]
From this, we get the lower bound.

With regards to the upper bound, first observe that
\[
\sum_{\tOm \in \IdxnCrDir} \indc[\vecXtOm \geq u \ones_{2b_k}] \; \indc[\GtOm \text{ forms a component}] \leq \cSnk(u) + \Dnk(u).
\]
On the other hand, from the definition of $\cLnk(u),$ we have
\[
\sum_{\tOm \in \IdxnCrDir} \indc[\vecXtOm \geq u \ones_{2b_k}] \; \indc[\GtOm \text{ does not form a component}] \leq \cLnk(u).
\]
From these observations and \eqref{eqn:SnkDecom}, the upper bound is easy to see. The desired result follows.

\item Arguments for Statement~\ref{itm:BdNnku}: Consider the definition of $\cNnk(u),$ \eqref{Defn:Nnk}, Lemmas~\ref{lem:CentrePtNSP} and \ref{lem:AllwaysIsoCase}, and Remark~\ref{rem:NnkImp}. By the same argument used for deriving the lower bound above, the desired result follows.

\item Arguments for Statement~\ref{itm:BdLnku}: From the definition of $\cLnk(u),$ \eqref{Defn:Lnk}, and Remark~\ref{rem:LnkImp}, the desired result follows trivially.
\end{itemize}
This completes the proof.
\end{proof}

\label{pf:limRates}
We now aim to prove Lemma~\ref{lem:limRates}. But, for that, we first need probability estimates for the events associated with the indicators in \eqref{Defn:Snk}, \eqref{Defn:Nnk}, \eqref{Defn:Lnk}, and \eqref{Defn:Dnk}. Due to the stationarity assumption on field $X$ (see Section~\ref{sec:intro}), it suffices to obtain these estimates for $\vec{t} = 0.$ For the $k \geq 1$ case in \eqref{Defn:Snk} and \eqref{Defn:Dnk}, by additionally using the $\|\cdot\|_1-$isotropy assumption on $X,$ it in fact follows that we only need to consider $\vec{\Omega} = \set{b_k}.$ The next result gives these bounds. Its proof is in Section~\ref{sec:ConclResults}, p.~\pageref{Pf:ProbEst}, and uses Lemmas~\ref{lem:GaussianTailBounds} and~\ref{lem:SlepainInequality}; the Savage condition holds due to our assumptions in Table~\ref{tab:Assumptions}.
\begin{lemma}
\label{lem:ProbEst}
Let $\{\un\} \in \pSeq$ be such that $\lim_{n \to \infty} \un = \infty.$ Then, the following claims are true.
\begin{enumerate}
\setlength\itemsep{1em}
\item \label{itm:Sn0Ind} If $\Cr{a:r0Asm}$ holds, then $\left|\dfrac{\Pr\{X_{\vec{0}} \geq \un\}}{\tau_0 \un^{-2b_0} \exp[-a_0 \un^2]} - 1\right| = O\left(\dfrac{1}{\un^2}\right).$

\item \label{itm:SnkInd} Suppose $1 \leq k < d.$ If $\Cr{a:r0Asm}$ and  $\Cr{a:r2Asm}$ hold, then $\left|\dfrac{\Pr\{\vec{X}_{\vec{0}, \set{b_k}} \geq \un \ones_{2b_k}\}}{\left[\tau_k/\tbinom{d}{b_k}\right] \un^{-2b_k} \exp[-a_k \un^2]} - 1\right| = O\left(\dfrac{1}{\un^2}\right).$

\item \label{itm:NnkInd} Suppose $1 \leq k < d - 1.$ If $\Cr{a:r0Asm}, \Cr{a:r2Asm},$ and $\Cr{a:rqAsm}$ hold, and $G \in \cN_k,$ then
\[
\Pr\{\vec{X}_{\vec{0}, G} \geq \un \ones_{2b_k}\} = O\left(\un^{-2b_k} \exp[-\varrho_k \, \un^2]\right).
\]

\item \label{itm:Ln0Ind} If $\Cr{a:r0Asm}$ and $\Cr{a:r1Asm}$ hold, then %
\[
\Pr\{(X_{\vec{0}}, X_{\vec{e}_1}) \geq \un \ones_{2}\} = O\left(\un^{-(2b_0 + 1)} \exp[- \varphi_0 \, \un^2]\right).
\]

\item \label{itm:SnkIndWC} Suppose $1 \leq k < d.$ If $\Cr{a:r0Asm}, \Cr{a:r1Asm},$ and $\Cr{a:ComC}(k)$ hold, then
\[
\Pr\{(\vec{X}_{\vec{0}, \set{k + 1}}, X_{\vec{0}}) \geq \un \ones_{2b_k + 1}\} = O\left(\un^{-(2b_k + 1)}\exp[-\varphi_k \, \un^2]\right),
\]

\item \label{itm:LnkInd} Suppose $0 \leq k < d.$ If $\Cr{a:r0Asm}, \Cr{a:r1Asm},$ $\Cr{a:r2LAsm},$ and $\Cr{a:ComC}(k)$ hold, and $G \in \cP_k,$ then
\[
\Pr\{\vec{X}_{\vec{0}, G} \geq \un \ones_{|V|}\} = O\left(\un^{-(2b_k + 1)}\exp[-\varphi_k \, \un^2]\right),
\]
\end{enumerate}
\end{lemma}

We are now ready to prove Lemma~\ref{lem:limRates}.

\begin{proof}[Proof of Lemma~\ref{lem:limRates}]
Each statement is proved individually.
\begin{itemize}[leftmargin=*]
\item Arguments for Statement~\ref{itm:ExPSnku}: First consider the case $1 \leq k < d.$ Using \eqref{Defn:Snk}, we have
\[
\ExP[S_{n, k}(\un)] = \sum_{\tOm \in \IdxnCrDir} \Pr\{\vecXtOm \geq u \ones_{2b_k}\},
\]
where $\vecXtOm$ is as in \eqref{Defn:vecX}. For each $\tOm \in \dGridCrDir,$ note that the pairwise $\|\cdot\|_1-$distances between the indices involved in the definition of $\vecXtOm$ is exactly $2.$ Consequently, using the fact that the field $X$ is both stationary and $\|\cdot\|_1-$isotropic, it follows that
\[
\ExP[S_{n, k}(\un)] = (2n + 1)^d \binom{d}{b_k} \Pr\{\vec{X}_{\vec{0}, \set{b_k}} \geq \un \ones_{2b_k}\}.
\]
From \eqref{Defn:Lambdanku}, we then have
\begin{equation}
\label{eqn:SnkLbnkRatio}
\frac{\ExP[S_{n, k}(\un)]}{\lbnk(\un)} = \dfrac{\Pr\{\vec{X}_{\vec{0}, \set{b_k}} \geq \un \ones_{2b_k}\}}{\left[\tau_k/\tbinom{d}{b_k}\right] \un^{-2b_k} \exp[-a_k \un^2]}.
\end{equation}
The desired result now follows from Lemma~\ref{lem:ProbEst}.\ref{itm:SnkInd}.

The case $k = 0$ follows similarly using Lemma~\ref{lem:ProbEst}.\ref{itm:Sn0Ind}, except, no condition on $\rho_2$ is needed here.

\item Arguments for Statement~\ref{itm:ExPNnku}: The expectation bound is trivially true when $k = 0$ or $d - 1.$ Suppose that $1 \leq k < d -1.$ From \eqref{Defn:Nnk} and the fact that the field $X$ is stationary,
\[
\ExP[\Nnk(\un)] = (2n + 1)^d \sum_{G \in \cN_k} \Pr\{\vec{X}_{\vec{0}, G} \geq \un \ones_{2b_k}\}.
\]
Therefore, from \eqref{Defn:Lambdanku},
\[
\frac{\ExP[\Nnk(\un)]}{\lbnk(\un)} \leq \frac{|\cN_k| \left[\sup_{G \in \cN_k} \Pr\{\vec{X}_{\vec{0}, G} \geq \un \ones_{2k + 2}\}\right]}{\tau_k \un^{-2b_k} \exp[-a_k \un^2]}.
\]
Since $|\cN_k|$ is finite, by invoking Lemma~\ref{lem:ProbEst}.\ref{itm:NnkInd}, the desired bound follows.

It only remains to show that $\varrho_k > a_k$ for $1 \leq k < d - 1.$ But, from \eqref{Defn:Tauk} and \eqref{Defn:VarRho},
\[
\varrho_k - a_k = \frac{(2b_k - 1) (\rho_2 - \rho_3)(1 - \rho_3 + (b_k - 1)(\rho_2 - \rho_3))}{(1 - \rho_3^2 + 2(b_k - 1)(\rho_2 - \rho_3^2))(1 + (2b_k - 1) \rho_2)} > 0,
\]
where the positivity follows since $0 \leq \rho_3 < \rho_2 < 1,$ which itself holds on due to $\Cr{a:r0Asm},$ $\Cr{a:r2Asm},$ and $\Cr{a:rqAsm}.$ This completes the proof.

\item Arguments for Statement~\ref{itm:ExPLnku}: Using \eqref{Defn:Lnk},  Lemma~\ref{lem:ProbEst}.\ref{itm:LnkInd}, and similar arguments as in the proof of Statement~\ref{itm:ExPNnku} above, the expectation bound is easy see. It only remains to show $\varphi_k > a_k$ for $k \geq 0.$

Using \eqref{Defn:Tau0}, \eqref{Defn:Tauk}, and \eqref{Defn:VarPhi}, note that
\begin{equation}
\label{eqn:DiffPhiA}
\varphi_k - a_k = (1 - 2a_k \rho_1)^2/[2(1- 2a_k\rho_1^2)].
\end{equation}
For the $k = 0$ case, the expression on the right is positive due to $\Cr{a:r0Asm}$ and  $\Cr{a:r1Asm};$ for the $1 \leq k < d$ case, it is positive due to $\Cr{a:ComC}(k)$ and the facts that $\rho_2 \geq 0$ and $0 \leq \rho_1 < 1,$ which itself hold due to $\Cr{a:r0Asm},$ $\Cr{a:r1Asm},$ and $\Cr{a:r2LAsm}.$ We are now done.

\item Arguments for Statement~\ref{itm:ExPDnku}: The $k = 0$ case is trivial. So, suppose that $1 \leq k < d.$ A simple observation here is that $|\{\tOm \in \IdxnCrDir : \|\vec{t}\|_\infty = n\}| \leq 2d (2n + 1)^{d - 1}.$ Further, from Lemma~\ref{lem:ProbEst}.\ref{itm:SnkInd},
\[
\dfrac{\Pr\{\vec{X}_{\vec{0}, \set{b_k}} \geq \un \ones_{2b_k}\}}{\left[\tau_k/\tbinom{d}{b_k}\right] \un^{-2b_k} \exp[-a_k \un^2]} = O\left(1 + \frac{1}{\un^2}\right) = O(1),
\]
where the latter holds since $\un \to \infty.$ Arguing now as in the proof of Statement~\ref{itm:ExPSnku} above, the desired result is easy to see.
\end{itemize}
This completes the proof.
\end{proof}

\label{Pf:PoiResSnk}
Our next objective is to establish Theorem~\ref{thm:PoiResSnk} by making use of Theorem~\ref{thm:SuffCondPoissonConv}. As mentioned earlier, the $k = 0$ case of this result has been proved in \cite[Theorem 3.6]{holst1990poisson}. That proof makes use of Lemma~3.4 there, which discusses several covariance bounds concerning the indicators that add up to give $S_{n, 0}(u);$ see \eqref{Defn:Snk} for the definition of the latter. Below we restate a few of these covariance bounds in a form that is convenient to us. While these bounds have been shown in \cite{holst1990poisson}, they also follow from our proof for  Lemma~\ref{lem:CovBounds} stated below.

From now on, whenever appropriate, we shall write $\indc[X_{\vec{t}} \geq u]$ as $\xit(u).$
\begin{lemma}
\label{lem:CovBounds0}
Let $\vec{t}, \vecp{t} \in \mathbb{Z}^d$ and $\{\un\} \in \pSeq$ be such that $\lim_{n \to \infty} \un = \infty.$ Let $a_0$ and $b_0$ be as in \eqref{Defn:Tau0}. Suppose $\Cr{a:r0Asm},$ $\Cr{a:r1Asm},$ and $\Cr{a:r2LAsm}$ hold. Then, the following statements are true.
\begin{enumerate}
\setlength\itemsep{1em}
\item $\cov[\xit(\un), \xitp(\un)] \geq 0$ for all $n \geq 0.$

\item If $\|\vec{t} - \vecp{t}\|_1 = q \geq 1,$ then $\cov[\xit(\un), \xitp(\un)] = O\left(\un^{-4b_0} \exp \left[\dfrac{-2 a_0 \un^2}{1 + 2 a_0 \rho_q}\right]\right).$

\item \label{itm:StBd0} If $\|\vec{t} - \vecp{t}\|_1 = q \geq 1,$ then
\[
\cov[\xit(\un), \xitp(\un)] = O\left(\rho_q \un^{-(4b_0 - 2)} \exp \left[\dfrac{-2 a_0 \un^2}{1 + 2 a_0 \rho_q}\right] \right) = O\left(\rho_q \exp \left[\dfrac{-2 a_0 \un^2}{1 + 2 a_0 \rho_q}\right] \right).
\]
\end{enumerate}
In each statement, the constants involved in the $O$ notation are independent of both $\vec{t}$ and $\vecp{t}.$
\end{lemma}

In order to handle the $1 \leq k < d$ case in Theorem~\ref{thm:PoiResSnk}, our first aim is to obtain similar covariance bounds for the indicators that sum up to $\Snk(u);$ the latter is as defined in \eqref{Defn:Snk}. These bounds are stated next.

Henceforth, as and when convenient, we shall shorten $\indc[\vecXtOm \geq u\ones_{2b_k}]$ to $\xitOm(u).$ Separately, keeping in mind that the covariance sequence $\{\rho_q\}$ may not be monotonically decreasing, for $\tOm, \tpOmp \in \dGridCrDir,$ let
\begin{equation}
\label{Defn:rhoa}
\rhoa_{\tOm, \tpOmp} := \max\left\{\rho_{\|[\vec{t} + \alpha \vec{e}_{\omega}] - [\vecp{t} + \alpha^\prime \vec{e}_{\omegap}]\|_1} : \alpha, \alpha^\prime \in \{-1, +1\}, \omega \in \vec{\Omega}, \omegap \in \vecp{\Omega}\right\}.
\end{equation}
The triangle inequality shows that, whenever $\rhoa_{\tOm, \tpOmp} = \rho_q,$
\begin{equation}
\label{eqn:rhoaBd}
\max\{q - 2, 0\} \leq \|\vec{t} - \vecp{t}\|_1 \leq q + 2.
\end{equation}
From this, it immediately follows that there exists a $k-$dependent positive constant ${\Cl[Const]{c:upBdRho}}_{, k}$ such that, for any $\tOm \in \dGridCrDir$ and $q \geq 1,$
\begin{equation}
\label{eqn:upBdRho}
|\{\tpOmp \in \dGridCrDir: \rhoa_{\tOm, \tpOmp} = \rho_q\}| \leq {\Cr{c:upBdRho}}_{,k} \; q^{d - 1}.
\end{equation}

\begin{lemma}
\label{lem:CovBounds}
Fix $k$ such that $1 \leq k < d.$ Let $\tOm, \tpOmp \in \dGridCrDir,$ and $\{\un\} \in \pSeq$ be such that $\lim_{n \to \infty} \un = \infty.$ Let $a_k$ and $b_k$ be as in \eqref{Defn:Tauk} and let $\varphi_k$ be as in \eqref{Defn:VarPhi}. Suppose $\Cr{a:r0Asm}, \Cr{a:r1Asm}, \Cr{a:r2LAsm},$ and $\Cr{a:ComC}(k)$ hold. Then, the following statements are true.
\begin{enumerate}
\setlength\itemsep{1em}
\item \label{itm:NonNegCov} $\cov[\xitOm(\un), \xitpOmp(\un)] \geq 0$ for all $n.$
\item \label{itm:WkstBd} If $\tOm \neq \tpOmp,$ then
\[
\cov[\xitOm(\un), \xitpOmp(\un)] = O \left( \un^{-(2b_k + 1)} \exp[-\varphi_k \un^2]\right).
\]

\item \label{itm:WkBd} If $\rhoa_{\tOm, \tpOmp}= \rho_q$ with $q \geq 1,$ then
\[
\cov[\xitOm(\un), \xitpOmp(\un)] = O\left( \un^{-4b_k} \exp\left[\frac{ -2a_k \un^2}{1 + 2 a_k\rho_q}\right] \right).
\]

\item \label{itm:StBd} If $\rhoa_{\tOm, \tpOmp}= \rho_q$ with $q \geq 1,$ then
\[
\cov[\xitOm(\un), \xitpOmp(\un)] = O \left(
\rho_q \; \un^{-(4b_k - 2)} \exp\left[\frac{-2a_k\un^2}{1 + 2a_k \rho_q}\right]\right).
\]
\end{enumerate}
In each statement, while the constants involved in the $O$ notation do depend on $k,$ they are independent of both $\tOm$ and $\tpOmp.$
\end{lemma}

We prove this result in Section~\ref{sec:ConclResults}, p.~\pageref{Pf:CovBounds}. Statements~\ref{itm:NonNegCov}, \ref{itm:WkstBd}, and \ref{itm:WkBd} of this result follow by an application of Lemmas~\ref{lem:GaussianTailBounds} and \ref{lem:SlepainInequality}; these ideas are multivariate extensions of those discussed in the proof of \cite[Lemma 3.4]{holst1990poisson}. In contrast, our proof of Statement~\ref{itm:StBd} above is significantly different to the one used in \cite{holst1990poisson} to derive (the analogous version of) Lemma~\ref{lem:CovBounds0}.\ref{itm:StBd0}. There, given that the setup is much simpler, the proof proceeds via first principles. In contrast, here we first approximate the desired covariance by a definite integral and then make use of Lemma~\ref{lem:ZmuTailBeh}, which provides bounds on the integrand.

\begin{proof}[Proof of Theorem~\ref{thm:PoiResSnk}]
Because $\Cr{a:r0Asm}, \Cr{a:r1Asm},$ $\Cr{a:ComC}(k),$ and $\Cr{a:r2LAsm}$ hold, both $1 + (2b_k - 1)\rho_2 - 2b_k \rho_1$ and $\rho_1$ lie in $(0, 1),$ while $\rho_2 \in [0, 1).$ Using this, it is easy to see that $0 < \vartheta_k < 1,$ as desired.

It now remains to establish the total variation bounds. The $k = 0$ case has been proved in \cite[Theorem 3.6]{holst1990poisson}. We only deal with the $1 \leq k < d$ case here. So, fix one such $k.$

For any $n \geq 0,$ it is not difficult to see from \eqref{Defn:Snk}, Theorems~\ref{thm:SuffCondPoissonConv} and ~\ref{thm:ProbSpaceExistence}, and Lemma~\ref{lem:CovBounds}.\ref{itm:NonNegCov} that
\begin{eqnarray}
& & \tv{\Snk(\un)}{\Poi(\ExP[\Snk(\un)])} \nonumber \\
& \leq & \frac{1 - e^{-\lbnk(\un)}}{\lbnk(\un)} \Bigg[ \sum_{\tOm \in \IdxnCrDir} \hspace{-0.5em} [\ExP[\xitOm(\un)]]^2 +  \sum_{\substack{\tOm, \tpOmp \in \IdxnCrDir, \\ \tOm \neq \tpOmp}} \hspace{-0.5em} \cov[\xitOm(\un), \xitpOmp(\un)]\Bigg] \nonumber\\
& \leq & \frac{1}{\lbnk(\un)} \Bigg [ \sum_{\tOm \in \IdxnCrDir} \hspace{-0.5em} [\ExP[\xitOm(\un)]]^2  +  \sum_{\substack{\tOm, \tpOmp \in \IdxnCrDir, \\ \tOm \neq \tpOmp}} \hspace{-0.5em} \cov[\xitOm(\un), \xitpOmp(\un)]\Bigg] \label{eqn:dtvBd},
\end{eqnarray}
where the last relation follows since $1 - e^{-x} \leq 1.$ Since $\Cr{a:r1Asm}$ and $\Cr{a:r2LAsm}$ imply $\Cr{a:r2Asm}$ and also because $\vartheta_k/2b_k \in (0, 1),$ the desired result clearly follows from the following two claims:
\begin{enumerate}

\item \label{itm:SqMeanSum} If $\Cr{a:r0Asm}$ and $\Cr{a:r2Asm}$ hold, then $\dfrac{1}{\lbnk(\un)} \sum_{\tOm \in \IdxnCrDir} [\ExP[\xitOm(\un)]]^2 = O\left(\dfrac{1}{n^d}\right).$

\vspace{2ex}

\item \label{itm:CovSum} If $\Cr{a:r0Asm}, \Cr{a:r1Asm}, \Cr{a:ComC}(k), \Cr{a:PoiG},$ and $\Cr{a:r2LAsm},$ hold, then
\begin{eqnarray*}
& & \frac{1}{\lbnk(\un)} \sum_{\substack{\tOm, \tpOmp \in \IdxnCrDir \\ \tOm \neq \tpOmp}} \cov[\xitOm(\un), \xitpOmp(\un)]\\
& = & O\left(n^{-d \vartheta_k/(2 b_k)} [\log n]^{-(1 - \vartheta_k)/2} + \frac{\log n}{n} \sum_{q = 1}^{2dn} \rho_q\right).
\end{eqnarray*}

\end{enumerate}
In the remainder of this proof, we establish the above bounds.
\begin{itemize}[leftmargin=*]
\item Arguments for Claim~\ref{itm:SqMeanSum}: We have
\begin{eqnarray}
& & \frac{1}{\lbnk(\un)} \sum_{\tOm \in \IdxnCrDir} [\ExP[\xitOm(\un)]]^2 \nonumber \\
& = & (2n + 1)^d \tbinom{d}{b_k} [\Pr\{\vec{X}_{\vec{0}, \set{b_k}} \geq \un \ones_{2b_k}\}]^2 / \lbnk(\un) \nonumber \\
& = & \frac{\lbnk(\un)}{(2n + 1)^d \tbinom{d}{b_k}} \left[\frac{\Pr\{\vec{X}_{\vec{0}, \set{b_k}} \geq \un \ones_{2b_k}\}}{\left[\tau_k/\tbinom{d}{b_k}\right] \un^{-2b_k} \exp[-a_k \un^2]}\right]^2 \nonumber \\
& = &  \frac{\lbnk(\un)}{(2n + 1)^d \tbinom{d}{b_k}} O\left(1 + \frac{1}{\un^2} + \frac{1}{\un^4}\right) \label{eqn:IntSqSnkBd},
\end{eqnarray}
where the first relation is due to the stationarity and $\|\cdot\|_1-$isotropy of the field $X,$ the second is because of \eqref{Defn:Lambdanku}, while third one is a consequence of Lemma~\ref{lem:ProbEst}.\ref{itm:SnkInd}. The desired result now follows from the fact that $\un \to \infty$ and $\lbnk(\un) \to \lambda,$ a constant.

\item Arguments for Claim~\ref{itm:CovSum}: Pick a $\delta$ satisfying
\begin{equation}
\label{eqn:DeltaCond}
0 < (4a_k + 1) \delta < \frac{2 a_k \rho_1}{1 + 2a_k \rho_1} < 1.
\end{equation}
While the effectiveness of this choice will become clear below, the choice itself can be made because $\Cr{a:r0Asm}, \Cr{a:r1Asm},$ and $\Cr{a:r2LAsm}$ ensure $\rho_1 \in (0, 1)$ and $\rho_2 \in [0, 1).$ For any $\tOm, \tpOmp \in \IdxnCrDir,$ it is easy to see from \eqref{eqn:rhoaBd} that $\rhoa_{\tOm, \tpOmp} = \rho_q$ for some $q \in \{0, \ldots, 2dn + 2\}.$ Keeping this in mind, we rewrite the covariance sum in the claim as $\Term_1(n) + \ldots + \Term_4(n),$ where
\begin{equation}
\label{eqn:vCloseRangeSum}
\Term_1(n) :=  \frac{1}{\lbnk(\un)} \sum_{\substack{\tOm, \tpOmp \in \IdxnCrDir: \\[0.5ex] \tOm \, \neq \, \tpOmp, \; \rhoa_{\tOm, \tpOmp} \, = \,\rho_0}} \cov[\xitOm(\un), \xitpOmp(\un)],
\end{equation}
\begin{equation}
\label{eqn:EarlyRangeSum}
\Term_2(n) := \frac{1}{\lbnk(\un)}  \sum_{q: \rho_q \geq \delta,\;  q \neq 0} \sum_{\substack{\tOm, \tpOmp \in \IdxnCrDir: \\[0.5ex] \rhoa_{\tOm, \tpOmp} \, = \,\rho_q}}  \cov[\xitOm(\un), \xitpOmp(\un)],
\end{equation}
\begin{equation}
\label{eqn:MidRangeSum}
\Term_3(n) := \frac{1}{\lbnk(\un)}  \sum_{q: \rho_q < \delta,\;  q \leq (2n + 1)^\delta} \sum_{\substack{\tOm, \tpOmp \in \IdxnCrDir: \\[0.5ex]  \rhoa_{\tOm, \tpOmp} \, = \, \rho_q }} \cov[\xitOm(\un), \xitpOmp(\un)],
\end{equation}
and
\begin{equation}
\label{eqn:LongRangeSum}
\Term_4(n) := \frac{1}{\lbnk(\un)}  \sum_{\substack{q: \rho_q < \delta, \\[0.5ex] (2n + 1)^\delta < q \leq 2dn + 2}} \sum_{\substack{\tOm, \tpOmp \in \IdxnCrDir: \\[0.5ex] \rhoa_{\tOm, \tpOmp} \, = \, \rho_q}} \cov[\xitOm(\un), \xitpOmp(\un)].
\end{equation}
We now sequentially bound each of these four terms. Since $\lbnk(\un) \to \lambda,$ a constant, we ignore the $1/\lbnk(\un)$ scaling factor, present in each term, in the computations below.

Using the stationarity property of the field $X,$ the fact that $\{\tpOmp: \rhoa_{\tOm, \tpOmp} = \rho_0\}$ is finite for each $\tOm,$ and Lemma~\ref{lem:CovBounds}.\ref{itm:WkstBd}, it follows that
\begin{equation}
\label{eqn:T1Bd}
\Term_1(n) = O\left( (2n + 1)^d \un^{-(2b_k + 1)} \exp[-\varphi_k \un^2]\right).
\end{equation}
Hence, writing $\exp[-\varphi_k \un^2]$ as $[\tau_k (2n + 1)^d \un^{-2b_k}]^{-\varphi_k/a_k} [\lbnk(\un)]^{\varphi_k/a_k}$ and, again, using the fact that $\lbnk(\un) \to \lambda,$ it follows from \eqref{eqn:DiffPhiA} that
\[
\Term_1(n) = O\left(n^{-d\vartheta_k/2b_k} \un^{-(1 - \vartheta_k)}\right).
\]

Consider the second term. We have
\begin{eqnarray*}
\Term_2(n) & = & O\Biggl(\sum_{q: \rho_q \geq \delta,\;  q \neq 0} \sum_{\substack{\tOm, \tpOmp \in \IdxnCrDir: \\[0.5ex] \rhoa_{\tOm, \tpOmp} \, = \,\rho_q}} \un^{-4b_k} \exp\left[\frac{ -2a_k \un^2}{1 + 2 a_k\rho_q}\right]\Biggr)\\
& = & O\Biggl(\sum_{q: \rho_q \geq \delta,\;  q \neq 0} \sum_{\substack{\tOm, \tpOmp \in \IdxnCrDir: \\[0.5ex] \rhoa_{\tOm, \tpOmp} \, = \,\rho_q}} \un^{-4b_k} \exp\left[\frac{ -2a_k \un^2}{1 + 2 a_k\rho_1}\right]\Biggr) \\
& = & O\left((2n + 1)^d \un^{-4b_k} \exp\left[\frac{ -2a_k \un^2}{1 + 2 a_k\rho_1}\right]\sum_{q: \rho_q \geq \delta,\;  q \neq 0} q^{d - 1}\right) \\
& = & O\left((2n + 1)^d \un^{-4b_k} \exp\left[\frac{ -2a_k \un^2}{1 + 2 a_k\rho_1}\right]\right),
\end{eqnarray*}
where the first relation follows from Lemma~\ref{lem:CovBounds}.\ref{itm:WkBd}, the second one holds since $\rho_q \leq \rho_1$ which itself holds due to $\Cr{a:r2LAsm},$ the truth of the third one is on account of   \eqref{eqn:upBdRho} and the fact that $|\IdxnCrDir| = O((2n + 1)^d),$ while the last one follows because $\{q: \rho_q \geq \delta\}$ is a finite set which is true due to $\Cr{a:PoiG}.$ As we did for $\Term_1(n),$ by expressing the exponential term suitably so as to enable use of the fact that $\lbnk(\un) \to \lambda,$ we get
\[
\Term_2(n) = O\left(n^{-d(1 - 2a_k \rho_1)/(1 + 2a_k \rho_1)} \; \un^{-8b_k a_k \rho_1/(1 + 2a_k \rho_1)} \right).
\]

We now handle the third term. As above, using \eqref{eqn:upBdRho} and Lemma~\ref{lem:CovBounds}.\ref{itm:WkBd}, we have
\[
\Term_3(n) = O\left((2n + 1)^d \un^{-4b_k} \sum_{q: \rho_q < \delta, \; q \leq (2n + 1)^\delta} q^{d- 1} \exp\left[\frac{ -2a_k \un^2}{1 + 2 a_k\rho_q}\right] \right).
\]
Using the fact that the sum only concerns those $q$ where $\rho_q \leq \delta,$ it then follows that
\[
\Term_3(n) = O\left(n^d \un^{-4b_k} \exp\left[\frac{ -2a_k \un^2}{1 + 2 a_k\delta}\right] \sum_{q: q \leq (2n + 1)^\delta} q^{d- 1}  \right) =  O\left(n^{d(1 + \delta)} \un^{-4b_k} \exp\left[\frac{ -2a_k \un^2}{1 + 2 a_k \delta}\right] \right).
\]
The idea now is to show that our choice of $\delta$ in \eqref{eqn:DeltaCond} ensures  $\Term_2(n)$ and $\Term_3(n)$ have similar rates. Towards this, we have
\begin{eqnarray*}
\Term_3(n) & = & O\left(n^{d [(\delta + 1) - 2/(1 + 2a_k \delta)]} \; \un^{-8b_k a_k \delta/(1 + 2a_k \delta)} \right) \\
& = & O\left(n^{d [(\delta + 1) - 2/(1 + 2a_k \delta)]}\right)\\
& = & O\left(n^{-d} n^{d \delta} n^{2d[1 - 1/(1 + 2a_k \delta)]}\right)\\
& = & O\left(n^{-d} n^{d \delta} n^{4a_k \delta d/(1 + 2a_k \delta)]}\right)\\
& = & O\left(n^{-d} n^{d (1 + 4a_k)\delta} \right)\\
& = & O\left(n^{-d} n^{2a_k\rho_1 d /(1 + 2a_k \rho_1)} \right)\\
& = & O\left(n^{-d} n^{2a_k\rho_1 d /(1 + 2a_k \rho_1)} \left[n^d \un^{-4b_k}\right]^{2a_k\rho_1/(1 + 2a_k \rho_1)} \right),
\end{eqnarray*}
where the first relation is obtained by expressing the exponential term suitably as before, the second relation follows by dropping the $\un$ expression, which itself can be done because $\delta > 0;$ the third relation follows by multiplying and dividing by $n^d;$ the fifth relation is obtained by ignoring the $1 + 2a_k \delta$ term, which can be done since $1 + a_k \delta > 1;$ the sixth relation is due to \eqref{eqn:DeltaCond}; while we get the last relation by artificially introducing $\left[n^d \un^{-4b_k}\right]^{2a_k\rho_1/(1 + 2a_k \rho_1)},$ which can done since it grows to $\infty$ on account of the exponent being positive and $\lim_{n \to \infty} a_k \un^2/(d \log n)$ being $1.$ From this last relation, it is easy to see that
\[
\Term_3(n) = O\left(n^{-d(1 - 2a_k \rho_1)/(1 + 2a_k \rho_1)} \; \un^{-8b_k a_k \rho_1/(1 + 2a_k \rho_1)} \right),
\]
as desired.

We finally deal with the fourth term. As before, using \eqref{eqn:upBdRho} and Lemma~\ref{lem:CovBounds}.\ref{itm:StBd},
\begin{eqnarray*}
\Term_4(n) & = & O\left((2n + 1)^d \sum_{q: \rho_q < \delta, \; (2n + 1)^\delta < q \leq 2dn + 2} q^{d - 1} \rho_q \, \un^{- (4b_k - 2)} \, \exp\left[\frac{-2a_k \un^2}{1 + 2a_k\rho_q}\right] \right) \\
& = & O\Bigg(n^{2 d - 1}  \sum_{q: \rho_q < \delta, \; (2n + 1)^\delta < q \leq 2dn + 2} \rho_q \, \un^{- (4b_k - 2)} \, \exp\left[\frac{-2a_k \un^2}{1 + 2a_k\rho_q}\right] \Bigg).
\end{eqnarray*}
Due to $\Cr{a:PoiG},$ there exists a positive constant $\Cl[Const]{c:CovDecay}$ so that $\rho_q \log q \leq \Cr{c:CovDecay}$ for each $q \geq 1.$ Hence, for any $q > (2n + 1)^\delta,$ we have
\begin{equation}
\label{eqn:RhoiLognBound}
\rho_q \log(2n + 1) \leq \frac{\Cr{c:CovDecay}}{\delta}.
\end{equation}
Now, observe that
\begin{eqnarray*}
& & \un^{- (4b_k - 2)} \, \exp\left[\frac{-2a_k \un^2}{1 + 2a_k \rho_q}\right] \\
& = & O\left((2n + 1)^{-2d/(1 + 2a_k \rho_q)} \;  \un^2 \; \un^{-8b_k a_k \rho_q/(1 + 2a_k \rho_q)}\right)\\
& = & O\left((2n + 1)^{d[4a_k\rho_q/(1 + 2a_k \rho_q) - 2]} \;  \un^2 \; \un^{-8b_k a_k \rho_q/(1 + 2a_k \rho_q)}\right)\\
& = & O\left((2n + 1)^{d[4a_k\rho_q/(1 + 2a_k \rho_q) - 2]} \;  \un^2\right)\\
& = & O\left((2n + 1)^{-2d} [(2n + 1)^{\rho_q}]^{4a_k d} \; \un^2 \right) \\
& = & O\left(n^{-2d} \; \un^2 \right),
\end{eqnarray*}
where the first relation follows by suitably modifying the exponential term as before; the third relation follows by dropping the last $\un$ expression as the exponent is negative; the fourth relation follows by dropping the $1 + 2a_k \rho_q$ term; while the last relation follows from \eqref{eqn:RhoiLognBound} which shows that $(2n + 1)^{\rho_q}$ is a constant with respect to $q$ and $n.$ Consequently, we have
\[
\Term_4(n) = O\left(\frac{\un^2}{n} \sum_{q = 1}^{2dn} \rho_q \right).
\]
The desired result now follows from the bounds on $\Term_1(n), \ldots, \Term_4(n)$ obtained above, and the following reasons: first, $a_k \un^2/[d \log n] \to 1;$ second,
\[
\frac{\vartheta_k}{2b_k} = \frac{[1 - 2a_k\rho_1]^2}{2a_k[1 - 2a_k \rho_1^2]} < \frac{1 - 2a_k \rho_1}{1 + 2a_k \rho_1};
\]
the latter holds because $\rho_1 \in (0, 1),$ $\rho_2 \in [0, 1)$ and $1 - 2a_k \rho_1 > 0,$ which itself are true on account of $\Cr{a:r0Asm}, \Cr{a:r1Asm}, \Cr{a:ComC}(k)$ and $\Cr{a:r2LAsm}.$
\end{itemize}

This completes the proof.
\end{proof}

\label{Pf:VarianceEstimates}
We now proceed towards proving Lemma~\ref{lem:VarianceEstimates}. As the first step, we establish Lemma~\ref{lem:CovBdsNL} which, at a loose level, provides covariance bounds between indicators in \eqref{Defn:Nnk} and \eqref{Defn:Lnk}. Before stating it, we introduce few notations.

From Remark~\ref{rem:NonUniPartition}, recall that the vertex set of each $G \equiv (V, E) \in \cN_k,$ $1 \leq k < d - 1,$ has a uniquely associated partition ($V_1$ and $V_2$) so that the properties mentioned in Lemma~\ref{lem:PwL1Dist} hold. Keeping this in mind, let
\begin{equation}
\label{Defn:Nnkj}
N_{n, k, j}(u) :=
\begin{cases}
0, & \text{ if $k = 0, d - 1,$ } \\
\sum_{\tG \in \IdxnCrNk: |V_1| = j} \indc[\vecXtG \geq u \ones_{2b_k}] & \text{ if $1 \leq k < d - 1,$}
\end{cases}
\end{equation}
where $1 \leq j \leq b_k$ and $u \in \pReal.$ It is easy to see that
\begin{equation}
\label{eqn:NnkDecom}
\Nnk(\un) = \sum_{j = 1}^{b_k} N_{n, k, j}(\un).
\end{equation}

For $1 \leq k < d - 1;$ $\vec{t}, \vecp{t} \in \Gamma_n;$ and $G \equiv (V, E), \Gp \equiv (\Vp, \Ep) \in \cN_k;$ let
\begin{equation}
\label{Defn:rhoaN}
\rhoaN_{\tG, \tpGp} := \max\left\{\rho_{\|[\vec{t} + \vec{v}] - [\vecp{t} + \vecp{v}]\|_1} : \vec{v} \in V, \vecp{v} \in \Vp \right\};
\end{equation}
this mimics the definition in \eqref{Defn:rhoa}. In the same spirit, for $0 \leq k < d; \vec{t}, \vecp{t} \in \Gamma_n;$ and $G, \Gp \in \cP_k;$ define $\rhoaL_{\tG, \tpGp}.$ Recall that $|\cG_k|,$ which is used in \eqref{Defn:cNk}, \eqref{Defn:cPk}, \eqref{Defn:cP1}, and \eqref{Defn:cP0}, is finite for each $k \geq 0.$ Thus, by repeating the arguments that were used to derive \eqref{eqn:upBdRho}, it is not difficult to see that, for each $\tG \in \dGrid \times \cN_k$ and $q \geq 1,$
\begin{equation}
\label{eqn:upBdRhoN}
|\{\tpGp \in \dGrid \times \cN_k : \rhoaN_{\tG, \tpGp} = \rho_q\}| \leq {\Cl[Const]{c:upBdRhoN}}_{, k} \; q^{d - 1},
\end{equation}
where ${\Cr{c:upBdRhoN}}_{,k}$ is a positive constant (depending on $k$). Similarly, for each $\tG \in \dGrid \times \cP_k$ and $q \geq 1,$
\begin{equation}
\label{eqn:upBdRhoL}
|\{\tpGp \in \dGrid \times \cP_k : \rhoaL_{\tG, \tpGp} = \rho_q\}| \leq {\Cl[Const]{c:upBdRhoL}}_{, k} \; q^{d - 1},
\end{equation}
where ${\Cr{c:upBdRhoL}}_{,k}$ is another positive constant. Separately, let
\begin{equation}
\label{Defn:Nxitk}
\Nxitk(u) \equiv \indc[\vec{X}_{\vec{t}, G^*} \geq u \ones_{2b_k}], \quad \text{ if $1 \leq k < d - 1,$}
\end{equation}
where $G^* \equiv (V, E) \in \cN_k$ is some arbitrary but fixed graph such that, for a partition $V_1$ and $V_2$ of $V$ as in Lemma~\ref{lem:PwL1Dist}, we have $|V_1| = 1,$ and $\vec{X}_{\vec{t}, G^*}$ is as defined above \eqref{Defn:Nnk}. Also, let
\begin{equation}
\label{Defn:Lxitk}
\Lxitk(u) \equiv
\begin{cases}
\indc[(X_{\vec{t}}, X_{\vec{t} + \vec{e}_1})\geq u \ones_{2b_0 + 1}], & \text{ if $k = 0,$ } \\[1ex]
\indc[(\vec{X}_{\vec{t}, \set{k + 1}}, X_{\vec{t}}) \geq u \ones_{2b_k + 1}], &  \text{ if $1 \leq k < d,$ }
\end{cases}
\end{equation}
where $\vec{X}_{\vec{t}, \set{b_k}}$ is as defined in \eqref{Defn:vecX}.

\begin{lemma}
\label{lem:CovBdsNL}
The following statements are true.
\begin{enumerate}
\setlength\itemsep{1em}

\item \label{itm:NCovBdk} Suppose $\Cr{a:r0Asm}, \Cr{a:r2Asm},$ and $\Cr{a:rqAsm}$ hold. Fix $k$ so that $1 \leq k < d - 1$ and let $\{\un\} \in \pSeq$ be such that $\lim_{n \to \infty} \un = \infty.$ Then, there is a constant $\delN > 0$ and a family Gaussian random vectors $\{\vec{Z}_\mu \in \Real^{4b_k} : \mu \in [0, \delN]\}$ such that
\begin{equation}
\label{eqn:NCovBd}
0 \leq \Pr\{\vec{Z}_\mu \geq \un \ones_{4b_k}\} - \left[\ExP[\Nxik(\un)]\right]^2  = O \left(
\mu \; \un^{-(4b_k - 2)} \exp\left[\frac{-2\varrho_k \un^2}{1 + 2 \mu \varrho_k}\right]\right),
\end{equation}
where the hidden constants are independent of $\mu,$ while $\varrho_k$ is as in \eqref{Defn:VarRho}. Moreover, if $\tG, \tpGp \in \dGridCrNk$ is such that $\rhoaN_{\tG, \tpGp} = \mu \leq \delN$ and $|V_1| = |\Vp_1|,$ then for all sufficiently large $\un$
\begin{equation}
\label{eqn:NdomByY}
\Pr\{(\vecXtG, \vecXtpGp) \geq \un \ones_{4b_k} \} \leq \Pr\{\vec{Z}_{\mu} \geq \un \ones_{4b_k}\}.
\end{equation}
Above, $V_1$ (resp. $\Vp_1$) is the first subset in the uniquely associated partition of the vertex set of $G$ (resp. $\Gp$); see Remark~\ref{rem:NonUniPartition} and Lemma~\ref{lem:PwL1Dist}.

\item \label{itm:LCovBdk} Suppose $\Cr{a:r0Asm}, \Cr{a:r1Asm}, \Cr{a:ComC}(k),$ and $\Cr{a:r2LAsm}$ hold. Fix $k$ so that $0 \leq k < d$ and let $\{\un\} \in \pSeq$ be such that $\lim_{n \to \infty} \un = \infty.$ Then, there is a constant $\delL > 0$ and a family of Gaussian random vectors $\{\vec{Z}_\mu : \mu \in [0, \delL]\}$ such that
\begin{equation}
\label{eqn:LCovBd}
0 \leq \Pr\{\vec{Z}_\mu \geq \un \ones_{4b_k + 2}\} - \left[\ExP[\Lxik(\un)]\right]^2 =  O\left(\mu \; \un^{-4b_k} \exp\left[\frac{-2 \varphi_k \un^2}{1 + 2 \mu \varphi_k}\right]\right),
\end{equation}
where the hidden constants are independent of $\mu,$ while $\varphi_k$ is as in \eqref{Defn:VarPhi}. Moreover, if $\tG, \tpGp \in \dGridCrPk$ is such that $\rhoaL_{\tG, \tpGp} = \mu \leq \delL,$ then
\begin{equation}
\label{eqn:LDom}
\Pr\{(\vecXtG, \vecXtpGp) \geq \un \ones_{|V| + |\Vp|}\} \leq \Pr\{\vec{Z}_{\mu} \geq \un \ones_{4 b_k + 2}\}.
\end{equation}
Above, $V$ and $\Vp$ are the vertex sets of $G$ and $\Gp,$ respectively.
\end{enumerate}
\end{lemma}

The proof for this result is given in Section~\ref{sec:ConclResults}, p.~\pageref{Pf:CovBdsNL}.

\begin{proof}[Proof of Lemma~\ref{lem:VarianceEstimates}]
We discuss each statement individually.
\begin{itemize}[leftmargin=*]
\item Arguments for Statement~\ref{itm:SnkVarEst}: We first consider the case $1 \leq k < d.$ We claim that
\begin{equation}
\label{eqn:SnkVarSuffCond}
\left|\frac{\Var[\Snk(\un)]}{\lbnk(\un)} - \frac{\ExP[\Snk(\un)]}{\lbnk(\un)}\right| = O\left(\un^{-1} \exp[-[\varphi_k - a_k] \un^2]\right).
\end{equation}
Using the triangle inequality, Lemma~\ref{lem:limRates}.\ref{itm:ExPSnku}, and the fact that $\varphi_k > a_k$ (see Lemma~\ref{lem:limRates}.\ref{itm:ExPLnku}),
it is then easy to see that the desired result holds.

It remains to establish \eqref{eqn:SnkVarSuffCond}. From \eqref{Defn:Snk}, we have
\[
\ExP[\Snk(\un)] = \sum_{\tOm \in \IdxnCrDir} \ExP[\xitOm(\un)]
\]
and
\[
\Var[\Snk(\un)] = \sum_{\tOm \in \IdxnCrDir} \Var[\xitOm(\un)] +  \sum_{\substack{\tOm, \tpOmp \in \IdxnCrDir \\ \tOm \neq \tpOmp}} \cov[\xitOm(\un), \xitpOmp(\un)].
\]
Hence, using Lemma~\ref{lem:CovBounds}.\ref{itm:NonNegCov} and the fact that $\Var[\xitOm(\un)] = \ExP[\xitOm(\un)] - [\ExP[\xitOm(\un)]]^2,$ we get
\begin{eqnarray}
& & \left|\frac{\Var[\Snk(\un)]}{\lbnk(\un)} - \frac{\ExP[\Snk(\un)]}{\lbnk(\un)}\right| \label{eqn:VarExpDiff} \\
& \leq & \frac{\sum_{\tOm \in \IdxnCrDir} [\ExP[\xitOm(\un)]]^2}{\lbnk(\un)} + \frac{\sum_{\substack{\tOm, \tpOmp \in \IdxnCrDir \\ \tOm \neq \tpOmp}} \cov[\xitOm(\un), \xitpOmp(\un)]}{\lbnk(\un)}. \nonumber
\end{eqnarray}

Observe that $\Cr{a:r0Asm}$ and $\Cr{a:r2Asm}$ hold; the latter being a consequence of $\Cr{a:r2LAsm}$ and $\Cr{a:r1Asm}.$ Arguing as we did to obtain \eqref{eqn:IntSqSnkBd} and then substituting $\eqref{Defn:Lambdanku},$ it then follows that
\begin{eqnarray}
\frac{\sum_{\tOm \in \IdxnCrDir} [\ExP[\xitOm(\un)]]^2}{\lbnk(\un)} & = & O\left(\un^{-2b_k} \exp[-a_k \un^2] \left(1 + \frac{1}{\un^2} + \frac{1}{\un^4}\right) \right) \nonumber \\
& = & O\left(\un^{-2b_k} \exp[-a_k \un^2]\right). \label{eqn:FracSqMeanSumLb}
\end{eqnarray}

With $\rhoa_{\tOm, \tpOmp}$ as in \eqref{Defn:rhoa}, now write the covariance sum in \eqref{eqn:VarExpDiff} as $\Term_1(n) + \Term_2(n),$ where $\Term_1(n)$ is exactly as in \eqref{eqn:vCloseRangeSum}, while
\[
\Term_2(n) := \frac{1}{\lbnk(\un)} \; \sum_{q = 1}^{2dn + 2}\sum_{\substack{\tOm, \tpOmp \in \IdxnCrDir: \\[0.5ex] \rhoa_{\tOm, \tpOmp} \, = \,\rho_q}}  \cov[\xitOm(\un), \xitpOmp(\un)].
\]
Note that $\Cr{a:r0Asm}, \Cr{a:r1Asm}, \Cr{a:r2LAsm},$ and $\Cr{a:ComC}(k)$ hold. Therefore, making the same arguments as those that were used to derive \eqref{eqn:T1Bd}, and, additionally, using the fact that $\lbnk(\un)$ cannot be ignored unlike there, it follows from \eqref{Defn:Lambdanku} that
\begin{equation}
\label{eqn:FracInitCovSumLb}
\Term_1(n) = O\left(\un^{-1} \exp[-[\varphi_k - a_k]\un^2]\right).
\end{equation}

On the other hand, we have
\begin{eqnarray}
\Term_2(n) & = & \frac{1}{\lbnk(\un)} \; O\left((2n + 1)^d \sum_{q = 1}^{2dn + 2} \rho_q q^{d - 1} \un^{-(4b_k - 2)} \exp\left[\frac{-2a_k\un^2}{1 + 2a_k \rho_q}\right]\right) \nonumber \\
& = & \frac{1}{\lbnk(\un)} O\left((2n + 1)^d \un^{-(4b_k - 2)} \exp\left[\frac{-2a_k\un^2}{1 + 2a_k \rho_1}\right]  \sum_{q = 1}^{2dn + 2} \rho_q q^{d - 1} \right) \nonumber \\
& = & \frac{1}{\lbnk(\un)} O\left((2n + 1)^d \un^{-(4b_k - 2)} \exp\left[\frac{-2a_k\un^2}{1 + 2a_k \rho_1}\right]\right) \nonumber \\
& = & O\left(\un^{-2(b_k - 1)} \exp\left[\frac{-a_k(1 - 2a_k \rho_1)\un^2}{1 + 2a_k \rho_1}\right] \right), \label{eqn:FracEventualCovSumLb}
\end{eqnarray}
where the first relation holds due to \eqref{eqn:upBdRho}, Lemma~\ref{lem:CovBounds}.\ref{itm:StBd}, and the stationarity property of the field $X;$ the second one is obtained by using the fact that $\rho_q \leq \rho_1$ which itself holds on account of $\Cr{a:r2LAsm},$ the third is true to due to $\Cr{a:CLTG},$ while the last one is got by substituting \eqref{Defn:Lambdanku}.

On account of $\Cr{a:r0Asm}, \Cr{a:r1Asm}, \Cr{a:ComC}(k),$ and $\Cr{a:r2LAsm},$ note that $\rho_1 \in (0, 1),$ $\rho_2 \in [0, 1)$ and $1 - 2a_k \rho_1 > 0.$ This shows that the coefficient before $\un^2$ in the exponential terms  in \eqref{eqn:FracSqMeanSumLb} and \eqref{eqn:FracEventualCovSumLb} is negative; while the same holds in \eqref{eqn:FracInitCovSumLb} due to $\varphi_k > a_k$ (see Lemma~\ref{lem:limRates}.\ref{itm:ExPLnku}). Therefore, the terms in \eqref{eqn:FracSqMeanSumLb}, \eqref{eqn:FracInitCovSumLb}, and \eqref{eqn:FracEventualCovSumLb}, all decay to zero with $n;$ it remains the identify the one that decays the slowest. Also, from the conditions on $\rho_1, \rho_2$ discussed above, $(1 - 2a_k \rho_1)/(1 + 2a_k \rho_1) < 1,$ which shows that the coefficient in the exponential term in \eqref{eqn:FracEventualCovSumLb} is larger than that in \eqref{eqn:FracSqMeanSumLb}. On the other hand, additionally using \eqref{eqn:DiffPhiA}, one can easily see that \eqref{eqn:FracInitCovSumLb} decays slower than \eqref{eqn:FracEventualCovSumLb}. Thus, our claim in \eqref{eqn:SnkVarSuffCond} is true which gives the desired result.

The case $k = 0$ follows similarly by using Lemma~\ref{lem:CovBounds0}.

\item Arguments for Statement~\ref{itm:SqNnkEst}: The result is trivially true for $k = 0$ and $k = d - 1$ cases. So, suppose that $1 \leq k < d - 1.$

Since $\Cr{a:r0Asm}, \Cr{a:r2Asm},$ and $\Cr{a:rqAsm}$ hold, we have $0 \leq \rho_3 < \rho_2 < 1.$ Further, from Lemma~\ref{lem:limRates}.\ref{itm:ExPNnku}, we have $\varrho_k > a_k,$ where $\varrho_k$ is as in \eqref{Defn:VarRho}. Combining the two, it is then easy to see that $\varrho_k \in (0, \infty).$

Let $\delN > 0$ be as in Lemma~\ref{lem:CovBdsNL}.\ref{itm:NCovBdk} and let $\delta \in (0, 1)$ be such that $\delta \leq \min\{\delN, \frac{1}{2\varrho_k}\}.$ From \eqref{eqn:NnkDecom} and since $k$ is finite, we get
\[
\ExP[\Nnk^2(\un)] = O\left(\sum_{j = 1}^{b_k} \ExP[N^2_{n, k, j}(\un)]\right).
\]
Now, fix an arbitrary $j$ so that $1 \leq j \leq b_k.$ Then, $\ExP[N_{n, k, j}^2(\un)]/\lbnk^2(\un)$ can be written as $\Term_1(n) + \Term_2(n) + \Term_3(n),$ where
\[
\Term_1(n) = \frac{1}{\lbnk^2(\un)} \sum_{\tG \in \IdxnCrNk: |V_1| = j} \Pr\{\vecXtG  \geq \un \ones_{2b_k}\},
\]
\[
\Term_2(n) = \frac{1}{\lbnk^2(\un)} \sum_{\substack{q: \rho_q > \delta, \\[0.5ex] q \leq Q_n}} \; \sum_{\substack{\tG, \tpGp \in \IdxnCrNk: |V_1| = |\Vp_1| = j\\[0.5ex] \tG \neq \tpGp, \;  \rhoaN_{\tG, \tpGp} = \rho_q}} \Pr\{(\vecXtG, \vecXtpGp) \geq \un \ones_{4b_k}\},
\]
and
\[
\Term_3(n) := \frac{1}{\lbnk^2(\un)} \sum_{\substack{q: \rho_q \leq \delta, \\[0.5ex] q \leq Q_n}} \; \sum_{\substack{\tG, \tpGp \in \IdxnCrNk: |V_1| = |\Vp_1| = j\\[0.5ex] \tG \neq \tpGp, \;  \rhoaN_{\tG, \tpGp} = \rho_q}} \Pr\{(\vecXtG, \vecXtpGp) \geq \un \ones_{4b_k}\}.
\]
Here, $Q_n = O(2n + 1);$ this bound for $q$ follows from \eqref{Defn:rhoaN}, the fact that $\|\vec{t} - \vecp{t}\| \leq 2dn$ for any $\tG, \tpGp \in \IdxnCrNk,$  and since each $G \in \cN_k$ is a subset of a bounded graph $\cG_0$ (defined above \eqref{Defn:cPk}). This bound is in similar spirit to $2dn + 2$ that was used in the proof of Theorem~\ref{thm:PoiResSnk} (see the discussion above \eqref{eqn:vCloseRangeSum}). In the remainder of this proof, we bound each of the above three terms.

Consider $\Term_1(n).$ We have
\begin{eqnarray*}
\Term_1(n) & \leq & \frac{1}{\lbnk^2(\un)} (2n + 1)^d |\cN_k| \un^{-2b_k} \exp[-\varrho_k \un^2] \\
& = & O\left(\frac{\exp[-(\varrho_k - a_k)\un^2]}{\lbnk(\un)} \right),
\end{eqnarray*}
where the first relation follows from Lemma~\ref{lem:ProbEst}.\ref{itm:NnkInd}, and by making use of the facts that the field $X$ is stationary and $|\{G \in \cN_k: |V_1| = j\}| \leq |\cN_k|;$ while the next one follows from \eqref{Defn:Lambdanku}.

With regards to $\Term_2(n),$ note that
\begin{eqnarray*}
\Term_2(n) & = & \frac{1}{\lbnk^2(\un)} \sum_{\substack{q: \rho_q > \delta, \\[0.5ex] q \leq Q_n}} \; \sum_{\substack{\tG, \tpGp \in \IdxnCrNk: |V_1|= |\Vp_1| = j \\[0.5ex] \tG \neq \tpGp, \, \rhoaN_{\tG, \tpGp} = \rho_q}} \Pr\{\vecXtG \geq \un \ones_{2b_k}\} \\
& = & O\left(\frac{1}{\lbnk^2(\un)}  \left[\sum_{q:\rho_q > \delta} q^{d - 1} \right] \sum_{\tG \in \IdxnCrNk: |V_1| = j} \Pr\{\vecXtG \geq \un \ones_{2b_k}\} \right)\\
& = & O\left(\frac{1}{\lbnk^2(\un)}  \sum_{\tG \in \IdxnCrNk: |V_1| = j} \Pr\{\vecXtG \geq \un \ones_{2b_k}\} \right)\\
& = & O\left(\frac{\exp[-(\varrho_k - a_k)\un^2]}{\lbnk(\un)} \right),
\end{eqnarray*}
where the first relation follows by dropping variables, the second one follows on account of \eqref{eqn:upBdRhoN}, the third one holds due to fact that $|\{q: \rho_q > \delta\}|$ is finite which itself is true since $\lim_{q \to \infty} \rho_q = 0$ on account of $\Cr{a:CLTG},$ while the truth of the last one can be seen using arguments similar to those used to bound $\Term_1(n)$ above.

Moving onto $\Term_3(n),$ observe that
\begin{eqnarray*}
\Term_3(n) & = & O\Bigg(\frac{1}{\lbnk^2(\un)} \sum_{q: \rho_q \leq \delta, \; q \leq Q_n} \; \sum_{\substack{\tG, \tpGp \in \IdxnCrNk: |V_1| = |\Vp_1| = j\\[0.5ex] \tG \neq \tpGp, \;  \rhoaN_{\tG, \tpGp} = \rho_q}} \Pr\{\vec{Z}_{\rho_q} \geq \un \ones_{4b_k}\}\Bigg) \\
& = & O\left( \frac{(2n + 1)^d}{\lbnk^2(\un)} \sum_{q: \rho_q \leq \delta, \; q \leq Q_n} q^{d - 1} \Pr\{\vec{Z}_{\rho_q} \geq \un \ones_{4b_k}\} \right)\\
& = & O\left(\frac{(2n + 1)^d}{\lbnk^2(\un)}  \sum_{q: \rho_q \leq \delta, \; q \leq Q_n} q^{d - 1} \left[\Pr\{\vec{Z}_{\rho_q} \geq \un \ones_{4b_k}\} - \left[\ExP[\Nxik(\un)]\right]^2 \right]\right) \\
& & + \; O\left(\frac{(2n + 1)^d}{\lbnk^2(\un)}  \sum_{q: \rho_q \leq \delta, \; q \leq Q_n} q^{d - 1} \left[\ExP[\Nxik(\un)]\right]^2\right) \\
& = & O\left(\frac{(2n + 1)^d}{\lbnk^2(\un)} \un^{-(4b_k - 2)}  \sum_{q: \rho_q \leq \delta, \; q \leq Q_n}  \rho_q q^{d - 1} \exp\left[\frac{-2\varrho_k \un^2}{1 + 2\rho_q \varrho_k}\right] \right) \\
& & + \; O\left(\frac{(2n + 1)^d}{\lbnk^2(\un)} \left[\ExP[\Nxik(\un)]\right]^2 \sum_{q: \rho_q \leq \delta, \; q \leq Q_n} q^{d - 1}\right)\\
& = &  O\left(\frac{(2n + 1)^d}{\lbnk^2(\un)} \un^{-(4b_k - 2)} \exp\left[\frac{-2\varrho_k \un^2}{1 + 2 \delta \varrho_k}\right] \right) \\
& & + \;  O\left(\frac{(2n + 1)^d}{\lbnk^2(\un)} \left[\ExP[\Nxik(\un)]\right]^2 \sum_{q: \rho_q \leq \delta, q \leq Q_n} q^{d - 1}\right)\\
& = & O\left(\frac{(2n + 1)^d}{\lbnk^2(\un)} \un^{-(4b_k - 2)} \exp\left[-\varrho_k \un^2\right] \right) +
O\left(\frac{(2n + 1)^{2d}}{\lbnk^2(\un)} \left[\ExP[\Nxik(\un)]\right]^2 \right)\\
& = & O\left(\un^{-2(b_k - 1)} \dfrac{\exp[-(\varrho_k - a_k) \un^2]}{\lbnk(\un)}\right) + O\left(\exp[-2(\varrho_k - a_k) \un^2]\right),
\end{eqnarray*}
where the first relation follows from \eqref{eqn:NdomByY} and the fact that $\delta \leq \delN;$ the second one holds due to \eqref{eqn:upBdRhoN}, and since  $|\Idx_n| = (2n + 1)^d$ and $|\cN_k|$ is finite; the third one follows by adding and subtracting $\left[\ExP[\Nxik(\un)]\right]^2;$ the fourth one is true due to \eqref{eqn:NCovBd}, the fifth one holds due to $\Cr{a:CLTG}$ and since $\rho_q \leq \delta;$ the sixth one follows because $2 \delta \varrho_k \leq 1$ and $\sum_{q \leq Q_n} q^{d - 1} = O((2n + 1)^d);$ while the last one is got by using Lemma~\ref{lem:ProbEst}.\ref{itm:NnkInd} and then substituting \eqref{Defn:Lambdanku}.

Since $j$ was arbitrary and $b_k \geq 1$ for $1\leq k < d - 1,$ the desired result is now easy to see.

\item Arguments for Statement~\ref{itm:SqLnkEst}: The discussion here is similar to that in the proof for Statement~\ref{itm:SqNnkEst} above. So, we only highlight the major differences.

Since $\Cr{a:r0Asm}, \Cr{a:r1Asm}, \Cr{a:ComC}(k),$ and $\Cr{a:r2LAsm}$ hold, we have $0 \leq \rho_2 < \rho_1 < 1;$ also, $1 + (2b_k - 1) \rho_2 > 2b_k \rho_1.$ Further, we have from Lemma~\ref{lem:limRates}.\ref{itm:ExPLnku} that $\varphi_k > a_k,$ where $\varphi_k$ is as in \eqref{Defn:VarPhi}. Combining the two, it is easy to see that $\varphi_k \in (0, \infty).$

Let $\delL > 0$ be as in Lemma~\ref{lem:CovBdsNL}.\ref{itm:LCovBdk} and let $\delta \in (0, 1)$ such that $\delta \leq \min\{\delL, \frac{1}{2\varphi_k}\}.$ Now, from \eqref{Defn:Lnk}, it is easy to see that one can write $\ExP[\Lnk^2(\un)]/\lbnk^2(\un)$ as $\Term_1(n) + \Term_2(n) + \Term_3(n),$ where
\[
\Term_1(n) = \frac{1}{\lbnk^2(\un)} \sum_{\tG \in \IdxnCrPk} \Pr\{\vecXtG \geq \un \ones_{|V|}\},
\]
\[
\Term_2(n) = \frac{1}{\lbnk^2(\un)} \sum_{\substack{q: \rho_q > \delta, \\[0.5ex] q \leq Q_n}} \; \sum_{\substack{\tG, \tpGp \in \IdxnCrPk: \\[0.5ex] \tG \neq \tpGp, \, \rhoaL_{\tG, \tpGp} = \rho_q}} \Pr\{(\vecXtG, \vecXtpGp) \geq \un \ones_{|V| + |\Vp|}\},
\]
and
\[
\Term_3(n) = \frac{1}{\lbnk^2(\un)} \sum_{\substack{q: \rho_q \leq \delta, \\[0.5ex] q \leq Q_n}} \; \sum_{\substack{\tG, \tpGp \in \IdxnCrPk: \\[0.5ex] \tG \neq \tpGp, \, \rhoaL_{\tG, \tpGp} = \rho_q}} \Pr\{(\vecXtG, \vecXtpGp) \geq \un \ones_{|V| + |\Vp|}\}.
\]
Above, $V$ and $\Vp$ are the vertex sets of $G$ and $\Gp,$ respectively. Since each $G \in \cP_k$ is a subset of bounded graph $\cG_k$ (see \eqref{Defn:cPk}, \eqref{Defn:cP1}, and \eqref{Defn:cP0}), it follows, as in our discussion for Statement~\ref{itm:SqNnkEst} above, that the upper bound for $q$ satisfies $Q_n = O(2n + 1).$

By making use of Lemma~\ref{lem:ProbEst}.\ref{itm:LnkInd}, we then have
\[
\Term_1(n) = O\left(\frac{\un^{-1} \exp[-(\varphi_k - a_k) \un^2]}{\lbnk(\un)} \right).
\]
Similarly,
\[
\Term_2(n) = O\left(\frac{\un^{-1} \exp[-(\varphi_k - a_k) \un^2]}{\lbnk(\un)} \right).
\]
Lastly, by using \eqref{eqn:LDom}, then adding and subtracting $\left[\ExP[\Lxik(\un)]\right]^2$ and, finally, using Lemmas~\ref{lem:CovBdsNL}.\ref{itm:LCovBdk} and \ref{lem:ProbEst}.\ref{itm:LnkInd}, we get
\[
\Term_3(n) = O\left(\frac{\un^{-2b_k} \exp[-(\varphi_k - a_k)\un^2]}{\lbnk(\un)}\right) + O\left(\un^{-2} \exp[-2(\varphi_k - a_k)\un^2]\right).
\]
Since $2b_k \geq 1$ for $0 \leq k < d,$ the desired result is now easy to see.

\item Arguments for Statement~\ref{itm:SqDnkEst}: Our discussion here mimics the proofs of Statements~\ref{itm:SqNnkEst} and \ref{itm:SqLnkEst} above. So, our arguments are brief. The $k = 0$ case is trivial. So, suppose that $1 \leq k < d.$

Set $\delta = 1.$ From \eqref{Defn:Dnk}, observe that $\ExP[\Dnk^2(\un)]/\lbnk^2(\un)$ can be written as $\Term_1(n) + \cdots + \Term_3(n),$ where
\[
\Term_1(n) := \frac{1}{\lbnk^2(\un)} \sum_{\substack{\tOm \in \IdxnCrDir : \\[0.5ex] \|\vec{t}\|_\infty = n}} \Pr\{\vecXtOm \geq \un \ones_{2b_k}\},
\]
\[
\Term_2(n) := \frac{1}{\lbnk^2(\un)} \sum_{\substack{\tOm, \tpOmp \in \IdxnCrDir: \tG \neq \tpGp \\[0.5ex] \|\vec{t}\|_\infty = \|\vecp{t}\|_\infty = n, \, \rhoa_{\tOm, \tpOmp} = \delta }} \Pr\{(\vecXtOm, \vecXtpOmp) \geq \un \ones_{4b_k}\},
\]
\[
\Term_3(n) = \frac{1}{\lbnk^2(\un)} \sum_{\substack{q: \rho_q < \delta, \\[0.5ex] q \leq 2dn + 2}} \; \sum_{\substack{\tOm, \tpOmp \in \IdxnCrDir: \\[0.5ex] \|\vec{t}\|_\infty = \|\vecp{t}\|_\infty = n, \, \rhoa_{\tOm, \tpOmp} = \rho_q }} \Pr\{(\vecXtOm, \vecXtpOmp) \geq \un \ones_{4b_k}\}.
\]

Using Lemma~\ref{lem:ProbEst}.\ref{itm:SnkInd} and the fact that $|\{\tG \in \IdxnCrDir: \|\vec{t}\|_\infty = n\}| = O((2n + 1)^{d - 1}),$ we have
\[
\Term_1(n) = O\left(\frac{1}{n \lbnk(\un)}\right),
\]
and
\[
\Term_2(n) = O\left(\frac{1}{n \lbnk(\un)}\right).
\]
Consider $\Term_3(n).$ We first make the following observation. As in \eqref{eqn:upBdRho}, for any $\tOm \in \dGridCrDir$ with $\|\vec{t}\| = n,$
\[
|\{\tpOmp \in \dGridCrDir: \|\vecp{t}\| = n, \rhoa_{\tOm, \tpOmp} = \rho_q\}| \leq {\Cl[Const]{c:upBdRhon}}_{,k} \; q^{d - 2}
\]
for some $k-$dependent constant ${\Cr{c:upBdRhon}}_{, k} \geq 0.$ Because $\Cr{a:CLTG}$ holds, we trivially have that
\[
\sum_{q = 0}^{\infty} q^{d - 2} \rho_q < \infty.
\]
Separately, observe that, for any $\tOm, \tpOmp \in \dGridCrDir$ with $\rhoa_{\tOm, \tpOmp} = \rho_q,$
\begin{eqnarray*}
& & \Pr\{(\vecXtOm, \vecXtpOmp) \geq \un \ones_{4b_k}\} - \left[\ExP[\indc_{\vec{0}, \set{b_k}}]\right]^2 \\
& = & \Pr\{(\vecXtOm, \vecXtpOmp) \geq \un \ones_{4b_k}\} - \ExP[\xitOm(\un)] \ExP[\xitpOmp(\un)] \\
& = & O \left(\rho_q \; \un^{-(4b_k - 2)} \exp\left[\frac{-2a_k\un^2}{1 + 2a_k \rho_q}\right]\right) \\
& = & O \left(\rho_q \; \un^{-(4b_k - 2)} \exp\left[\frac{-2a_k\un^2}{1 + 2a_k \rho_1}\right]\right)
\end{eqnarray*}
where the first relation follows since $X$ is stationary, the next one holds due to Lemma~\ref{lem:CovBounds}.\ref{itm:StBd}, while the last holds since $\rho_q \leq \rho_1$ which itself is true due to $\Cr{a:r2LAsm}.$ Thus, by adding and subtracting $\left[\ExP[\indc_{\vec{0}, \set{b_k}}]\right]^2$ to the summand in $\Term_3(n),$ it then follows, as in the proof of Statement~\ref{itm:SqNnkEst} above, that
\[
\Term_3(n) = O\left(\frac{\un^{-(2b_k - 2)}}{n\lbnk(\un)}\exp\left[\frac{-a_k(1 - 2a_k \rho_1)\un^2}{1 + 2a_k \rho_1}\right]\right) + O\left(\frac{1}{n^2}\right).
\]
The desired result is now easy to see.
\end{itemize}

This completes the proof.
\end{proof}

\begin{remark}
\label{rem:VarNL}
Consider $\Nxitk(u)$ and $\Lxitk(u)$ as in \eqref{Defn:Nxitk} and \eqref{Defn:Lxitk}, respectively. Then, using ideas from the above proof, it is not difficult to see that
\[
\lim_{n \to \infty} \frac{\Var[\sum_{\vec{t} \in \Idx_n} \Nxitk(\un)]}{\lbnk(\un)} = 0
\]
if $1 \leq k < d - 1$ and $\Cr{a:r0Asm}, \Cr{a:CLTG}, \Cr{a:r2Asm}$ and $\Cr{a:rqAsm}$ hold; while
\[
\lim_{n \to \infty} \frac{\Var[\sum_{\vec{t} \in \Idx_n} \Lxitk(\un)]}{\lbnk(\un)} = 0
\]
if $0 \leq k < d$ and $\Cr{a:r0Asm}, \Cr{a:r1Asm}, \Cr{a:ComC}(k), \Cr{a:CLTG},$ and $\Cr{a:r2LAsm}$ hold.
\end{remark}

Finally, we discuss the proof of Theorem~\ref{thm:SCLT}.

\begin{proof}[Proof of Theorem~\ref{thm:SCLT}] \label{Pf:SCLT}
The proof is essentially the same for $k = 0$ and $1 \leq k < d$ cases. Therefore, we detail the steps only for the latter.

As in the proof of Theorem~\ref{thm:PoiResSnk}, \eqref{eqn:dtvBd} holds. Bounds for the two terms on the RHS there have already been obtained in the proof of Lemma~\ref{lem:VarianceEstimates}.\ref{itm:SnkVarEst}. Therefore, we have
\begin{equation}
\label{eqn:tvBdCLT}
\tv{\Snk(\un)}{\Poi(\ExP[\Snk(\un)])} = O\left(\un^{-1} \exp[-(\varphi_k - a_k)\un^2]\right).
\end{equation}
Again, since $\Cr{a:r1Asm}$ and $\Cr{a:r2LAsm}$ imply $\Cr{a:r2Asm}$ and because $\nu_n \to -\infty$ is equivalent to $\lbnk(\un) \to \infty,$ it follows from Lemma~\ref{lem:limRates}.\ref{itm:ExPSnku} that $\ExP[\Snk(\un)] \to \infty.$ Therefore, we have
\[
\frac{\Poi(\ExP[\Snk(\un)]) - \ExP[\Snk(\un)]}{\sqrt{\ExP[\Snk(\un)]}} \CiD \Gau(0,1);
\]
this can be easily seen by using the characteristic function. Separately, using \eqref{eqn:dtvBd} and Definition~\ref{Defn:DistTV},
\begin{eqnarray*}
& & \left\|\left[\frac{\Snk(\un) - \ExP[\Snk(\un)]}{\sqrt{\ExP[\Snk(\un)]}}\right] - \left[\frac{\Poi(\ExP[\Snk(\un)]) - \ExP[\Snk(\un)]}{\sqrt{\ExP[\Snk(\un)]}}\right]\right\|_{\textnormal{TV}} \\
& = & \tv{\Snk(\un)}{\Poi(\ExP[\Snk(\un)])}\\
& = & O\left(\un^{-1} \exp[-(\varphi_k - a_k)\un^2]\right).
\end{eqnarray*}
Therefore, combining the above two relations, it follows that
\[
\frac{\Snk(\un) - \ExP[\Snk(\un)]}{\sqrt{\ExP[\Snk(\un)]}} \CiD \Gau(0, 1).
\]
Now, since Lemma~\ref{lem:limRates}.\ref{itm:ExPSnku} holds, we get the desired result.
\end{proof}

\begin{remark}
We note that the above result could have been proved using Berman's idea (\cite{Berman}, Chapter 8), where he showed a CLT for the sojourn time of a continuous time Gaussian process. In order to understand the technique, let us focus first on proving a CLT for the number of exceedances $S_{n,0}(u_n)=\sum_{\bar{t}\in \Gamma_n}1(X_{\bar{t}}\geq u_n)$. If the process $X$ is stationary with absolutely continuous spectral distribution function, we have the following representation for its covariance function:
$$\ExP[X_0X_{\bar{t}}]=\rho(\bar{t})=\sum_{m_1=-\infty}^\infty\sum_{m_2=-\infty}^\infty\cdots\sum_{m_d=-\infty}^\infty b(m_1,m_2,\cdots,m_d)b(t_1+m_1,t_2+m_2,\cdots,t_d+m_d)$$
for some $b\in l^2$ (can be chosen as inverse Fourier transform of the spectral density). This also means that we have the following representation for $X$:
$$X_{\bar{t}}=\sum_{m_1=-\infty}^\infty\sum_{m_2=-\infty}^\infty\cdots\sum_{m_d=-\infty}^\infty b(t_1+m_1,t_2+m_2,\cdots,t_d+m_d)\zeta_{m_1\cdots m_d}$$
with $\zeta$ being IID Gaussians.

The main idea is to show that the limiting distribution of exceedances of $X$ is the same as that of an auxiliary process $X_v$ for some $v\in\mathbb{R}$. To define this process, first let
$$b_{v}(\bar{t})=b_{\bar{t}},\,\,\,|t_i|\leq \left\lceil \frac{v}{2}\right \rceil\,\,\forall i,$$ and zero otherwise. Using this, define the auxiliary process by
$$X_v(\bar{t})=\frac{\sum_{m_1=-\infty}^\infty\sum_{m_2=-\infty}^\infty\cdots\sum_{m_d=-\infty}^\infty b(t_1+m_1,t_2+m_2,\cdots,t_d+m_d)\zeta_{m_1m_2\cdots m_d}}{\sqrt{\sum_{m_1=-\infty}^\infty\sum_{m_2=-\infty}^\infty\cdots\sum_{m_d=-\infty}^\infty b_v^2(m_1,\cdots,m_d)}}.$$
The core concept then is that the distribution of exceedances of this auxiliary process can be boiled down to studying limits of sums of i.i.d. random variables.

It is also worth noting that one can generalize the above to random vectors in order to study the limiting behaviour of $S_{n,k}(u_n)$ for $k\geq 1$.
\end{remark}

\section{Completing the Proofs}
\label{sec:ConclResults}
The technical results from Section~\ref{sec:IntResults} are proved here.

For an arbitrary $m \geq 1$ and $\rho \in \Real,$ let $\covM_m(\rho)$ be the $m \times m$ matrix whose diagonal entries are all $1$  and the off-diagonal entries are all  equal to $\rho.$ Similarly, for $m_1, m_2 \geq 1$ and $\rho, \mu \in \Real,$ let $W_{m_1, m_2}(\rho, \mu)$ be the $(m_1 + m_2) \times (m_1 + m_2)$ matrix given by
\begin{equation}
\label{Def:WMat}
W_{m_1, m_2}(\rho, \mu) :=
\begin{bmatrix}
\covM_{m_1}(\rho) & \mu \times \onesM_{m_1, m_2} \\
\mu \times \onesM_{m_2, m_1} & \covM_{m_2}(\rho)
\end{bmatrix},
\end{equation}
where $\onesM_{i, j}$ denotes the all ones matrix of dimension $i \times j.$ Further, for $m \geq 2,$ $1 \leq m_1, m_2 < m,$ and $\rho, \mup, \mu \in \Real,$ let $Q_{m, m_1, m_2}(\rho, \mup, \mu) \in \Real^{2m \times 2m}$ be given by
\begin{equation}
\label{Def:QMat}
Q_{m, m_1, m_2}(\rho, \mup, \mu) =
\begin{bmatrix}
W_{m_1, m - m_1}(\rho, \mup) & \mu \onesM_{m, m} \\
\mu \onesM_{m, m} & W_{m_2, m - m_2}(\rho, \mup)
\end{bmatrix} \in \Real^{2 m \times 2 m}.
\end{equation}
Separately, for a rational function $h$, let $\den(h)$ denote its denominator. Lastly, given vectors $\vec{u}_1 \in \Real^{m_1},$ and $\vec{u_2} \in \Real^{m_2},$ let $(\vec{u}_1, \vec{u}_2)$ denote the row vector of dimension $m_1 + m_2$ whose first $m_1$ entries are those of $\vec{u_1}$ and the last $m_2$ entries are those of $\vec{u}_2.$

We begin by deriving the probability estimates stated in Lemma~\ref{lem:ProbEst}.

\begin{proof}[Proof of Lemma~\ref{lem:ProbEst}]\label{Pf:ProbEst}
The truth of each statement can be seen from Lemma~\ref{lem:GaussianTailBounds} and the following lines of argument. Below, the notations $i,$ $M,$ and $\vec{\Delta}$ are as in Lemma~\ref{lem:GaussianTailBounds}, while $u$ is an arbitrary but fixed positive real number.

\begin{itemize}[leftmargin=*]

\item Arguments for Statement~\ref{itm:Sn0Ind}: Because $\Cr{a:r0Asm}$ holds, $X_{\vec{0}} \sim \Gau(0, 1).$ With respect to $\Pr\{X_{\vec{0}} \geq u\},$ we thus have $i = 1,$ $M = 1,$ and $\vec{\Delta} = u.$ As $u$ is positive, the Savage condition holds trivially. Hence, from Lemma~\ref{lem:GaussianTailBounds}, it follows that %
\[
1 - \frac{1}{u^2} \leq \dfrac{\Pr\{X_{\vec{0}} \geq u\}}{\tau_0 u^{-2b_0} \exp[-a_0 u^2]} \leq 1.
\]
This completes the argument.

\item Arguments for Statement~\ref{itm:SnkInd}: Since $\Cr{a:r0Asm}$ holds, $\vec{X}_{\vec{0}, \set{b_k}} \sim \Gau(\vec{0}, \covM_{2b_k}(\rho_2)).$ On the other hand, because $\Cr{a:r2Asm}$ holds, it follows from Lemma~\ref{lem:eigM} that $\covM_{2b_k}(\rho_2)$ is positive definite with eigenvalues $(1 - \rho_2),$ repeated $2b_k  - 1$ times, and $(1 + (2b_k - 1)\rho_2);$ also, $\ones_{2b_k}$ is the eigenvector with eigenvalue $1 + (2b_k - 1) \rho_2.$ With $i = 2b_k,$ and $M = W_{2b_k}(\rho_2),$ we then have
\[
|M| = (1 - \rho_2)^{2b_k - 1} (1 + (2b_k - 1) \rho_2),
\]
\[
[u \ones_{2b_k}] M^{-1} [u \ones_{2b_k}]^\tr =  2 a_k u^2,
\]
\text{ and }
\[
\vec{\Delta} = [u \ones_{2b_k}] M^{-1} = [u /(1 + (2b_k - 1) \rho_2)] \ones_{2k + 2} > 0.
\]
The latter verifies the Savage condition and also shows that
\[
\prod_{j = 1}^{2b_k} \vec{\Delta}(j) = \left[\frac{u}{1 + (2b_k - 1)\rho_2}\right]^{2b_k} \text{ and } \vec{\Delta}(j) \vec{\Delta}(\ell) = \left[\frac{u}{1 + (2b_k - 1)\rho_2}\right]^2.
\]
From Lemma~\ref{lem:GaussianTailBounds}, we then have
\[
1 - \left[\dfrac{[1 + (2b_k - 1)\rho_2]^2 \sum_{j, \ell = 1}^{i} M_{j\ell} (1 + \delta_{j \ell})}{2}\right] \frac{1}{u^2} \leq  \dfrac{\Pr\{\vec{X}_{\vec{0}, \set{b_k}} \geq u \ones_{2b_k}\}}{\left[\tau_k/\tbinom{d}{b_k}\right] u^{-2b_k} \exp[-a_k u^2]} \leq 1,
\]
Because $\sum_{j, \ell = 1}^{i} M_{j\ell} (1 + \delta_{j \ell})$ is a constant, we are done.

\item Arguments for Statement~\ref{itm:NnkInd}: We use divide and conquer strategy. In that, we first classify the elements of $\cN_k$ based on $|V_1|,$ where $V_1$ is the first subset in the partition uniquely associated with each $G \in \cN_k;$ see Remark~\ref{rem:NonUniPartition} and Lemma~\ref{lem:PwL1Dist}. We then obtain an upper bound for $\Pr\{\vecXtG \geq u\ones_{2b_k}\}$ that holds for every $G$ in a class. Thereafter, identifying the relationship between the bounds for each class, we obtain an universal bound.


To begin with, we derive some basic facts with regards to $W_{j, 2b_k - j}(\rho_2, \rho_3)$ for $1 \leq j \leq b_k.$ Because $\Cr{a:r0Asm}, \Cr{a:r2Asm},$ and $\Cr{a:rqAsm}$ hold, we have
\begin{equation}
\label{eqn:IneqRho320}
0 \leq \rho_3 < \rho_2 < \rho_0 = 1.
\end{equation}
This implies that $1 + (j - 1)\rho_2 - j\rho_3$ and $1 + (2b_k - 1 - j)\rho_2 - (2b_k - j) \rho_3$ are positive. From Lemma~\ref{lem:PDofW}, it then follows that $W_{j, 2b_k - j}(\rho_2, \rho_3)$ is positive definite. Further, from its definition, we also have that this matrix is symmetric. Both these facts, put together, show that $Y_j \sim \Gau(\vec{0}, W_{j, 2b_k - j}(\rho_2, \rho_3))$ is well defined.

Separately, for $1 \leq j \leq 2b_k - 1,$ let
\[
r_j = \frac{1 - \rho_3 + (2b_k - 1 - j)(\rho_2 - \rho_3)}{1 - \rho_3^2 + (2b_k - 2)(\rho_2 - \rho_3^2) + (j - 1)(2b_k - 1 - j)(\rho_2^2 - \rho_3^2)},
\]
and
\[
\varrho_{k, j} := j r_j + (2b_k  - j) r_{2b_k - j}.
\]
Then, using brute force, it is not difficult to see that
\[
\ones_{2b_k} \; W^{-1}_{j, 2b_k - j}(\rho_2, \rho_3)  = (r_j \ones_{j},  r_{2b_k - j} \ones_{2b_k - j})
\]
and
\[
\ones_{2b_k} \; [W_{j, 2b_k - j}(\rho_2, \rho_3)]^{-1} \;  \ones_{2 b_k}^\tr =  \varrho_{k, j}
\]
for $1 \leq j \leq b_k.$ Clearly, $\den(r_j) = \den(r_{2b_k - j});$ this implies that $\den(\varrho_{k, j}) = \den(r_j).$ Further, since \eqref{eqn:IneqRho320} holds,  both the numerator and denominator of each $r_j$ are positive. This implies that, for each $1 \leq j \leq b_k,$
\begin{equation}
\label{eqn:DelN_Pos}
\ones_{2b_k} \; W^{-1}_{j, 2b_k - j}(\rho_2, \rho_3)  > 0 \text{ and } \varrho_{k, j} > 0.
\end{equation}

Lastly, for $1 \leq j \leq b_k - 1,$  by using brute force or a symbolic calculator, we have
\[
\varrho_{k, j + 1} - \varrho_{k, j} =
\dfrac{2 (2(b_k - 1) - 2j + 1)(\rho_2 - \rho_3)(1 - \rho_2)(1 - \rho_3 + (b_k - 1) (\rho_2 -  \rho_3))}{\den(\varrho_{k, j}) \den(\varrho_{k, j + 1})};
\]
hence, and since \eqref{eqn:IneqRho320} holds, we get $\varrho_{k, j + 1} - \varrho_{k, j} > 0.$ In other words,
\begin{equation}
\label{eqn:DomNInd}
\varrho_{k , 1} = \ones_{2 b_k} [W^{-1}_{1, 2b_k - 1}(\rho_2, \rho_3)] \ones_{2b_k}^\tr < \ones_{2 b_k} [W^{-1}_{j, 2 b_k - j}(\rho_2, \rho_3)] \ones_{2b_k}^\tr = \varrho_{k , j}
\end{equation}
for each $j$ such that $2 \leq j \leq b_k.$

Now, fix some arbitrary $G \in \cN_k$ and suppose that $|V_1| = m.$ Lemma~\ref{lem:PwL1Dist} implies that $\|\vec{t} - \vecp{t}\|_1 \geq 3$ for $\vec{t} \in V_1$ and $\vecp{t} \in V \setminus V_1.$ Further, if $\vec{t}, \vecp{t}$ are both in $V_1$ or are both in $V_2,$ then it follows from Corollary~\ref{cor:DistVertMC} that $\|\vec{t} - \vecp{t}\|_1 \geq 2.$ Since $\Cr{a:r0Asm}$ and $\Cr{a:rqAsm}$ hold, after a permutation of the coordinates of $\vec{X}_{\vec{0}, G}$ if necessary, it now follows from Lemma~\ref{lem:SlepainInequality} that
\[
\Pr\{\vec{X}_{\vec{0}, G} \geq u \ones_{2b_k} \} \leq \Pr\{\vec{Y}_{m} \geq u \ones_{2b_k}\}.
\]
Consider $\Pr\{\vec{Y}_{m} \geq u \ones_{2b_k}\}.$ With $i = 2b_k$ and $M = W_{m, 2b_k - m}(\rho_2, \rho_3),$ \eqref{eqn:DelN_Pos} shows that
\begin{equation}
\label{eqn:DeltaNnk}
\vec{\Delta} := [u \ones_{2b_k}] \; M^{-1} > 0,
\end{equation}
which verifies the Savage condition. From Lemma~\ref{lem:GaussianTailBounds}, we then have
\begin{equation}
\label{eqn:YmBd}
\Pr\{\vec{Y}_{m} \geq u \ones_{2b_k}\} \leq  \frac{1}{(2 \pi)^{b_k}|M|^{1/2}\prod_{l = 1}^{2b_k} \vec{\Delta}(l)} \; \exp\left[-\frac{\varrho_{k, m} u^2}{2} \right].
\end{equation}
Combining this with \eqref{eqn:DomNInd}, we finally have
\begin{equation}
\label{eqn:YmBd2}
\Pr\{\vec{Y}_{m} \geq u \ones_{2b_k}\} \leq {\Cl[Const]{c:PrNnk}}_{,k} \; u^{-2b_k} \exp\left[-\frac{\varrho_{k, 1} u^2}{2} \right],
\end{equation}
where
\[
{\Cr{c:PrNnk}}_{,k} := \frac{1}{(2\pi)^{b_k}}\max_{1\leq j \leq b_k} \left[\dfrac{1}{|W_{j, 2b_k - j}(\rho_2, \rho_3)|^{1/2} \prod_{l = 1}^{2b_k} \vecp{\Delta}_j(l)}\right]
\]
with $\vecp{\Delta}_j = \ones_{2b_k} W^{-1}_{j, 2b_k - j}(\rho_2, \rho_3).$ Note that $\varrho_{k , 1}/2 = \varrho_k,$ where the latter is as in \eqref{Defn:VarRho}. Since $G$ was arbitrary, this line of argument is now complete.

\item Arguments for  Statement~\ref{itm:Ln0Ind}:  Consider $\Pr\{(X_{\vec{0}}, X_{\vec{e}_1}) \geq u \ones_2\}.$ With regards to this expression, $i = 2$ and $M = \covM_{2}(\rho_1),$ where the latter is on account of $\Cr{a:r0Asm}.$ Clearly,
\[
|M| = (1 - \rho_1)(1 + \rho_1),
\]
\[
[u \ones_{2}]M^{-1}[u \ones_{2}]^\tr = 2u^2/(1 + \rho_1),
\]
and
\[
\vec{\Delta} = [u/(1 + \rho_1)] \ones_{2} > 0.
\]
The latter verifies the Savage condition and also shows that $\prod_{j = 1}^{2} \vec{\Delta}(j) = \frac{u^2}{(1 + \rho_1)^2}.$ Hence, from Lemma~\ref{lem:GaussianTailBounds}, we have
\[
\Pr\{(X_{\vec{0}}, X_{\vec{e}_1}) \geq u \ones_2\} \leq \frac{(1 + \rho_1)^2}{2 \pi \sqrt{1 - \rho_1^2}} u^{-(2b_0 + 1)} \exp[-\varphi_0 u^2],
\]
where $\varphi_0$ is as in \eqref{Defn:VarPhi}. We are now done.

\item Arguments for Statement~\ref{itm:SnkIndWC}: Here, $i = 2b_k + 1$ and $M = W_{2b_k, 1}(\rho_2, \rho_1).$ As $\Cr{a:r1Asm}$ and $\Cr{a:ComC}(k)$ hold, it follows from Lemma~\ref{lem:PDofW} that $M$ is positive definite. Now, by brute force,
\[
[u \ones_{2b_k + 1}] M^{-1} [u \ones_{2b_k + 1}]^\tr = u^2 \, [(2b_k) r_{2b_k} + r_1] = 2 u^2 \varphi_k
\]
and
\[
\vec{\Delta} = [u \ones_{2b_k + 1}] M^{-1} = (r_{2b_k} \ones_{2b_k}, \;  r_{1} ) \, u,
\]
where $\varphi_k$ is as in \eqref{Defn:VarPhi} and
\[
r_j = \frac{1 + (2b_k - j) \rho_2 - (2b_k + 1 - j) \rho_1}{[1 + (2b_k - 1) \rho_2] [1 - \rho_1] +  \rho_1 [1 + (2b_k - 1) \rho_2 - (2b_k) \rho_1]}
\]
for $j = 1, 2b_k.$ Since $\Cr{a:r1Asm}$ and $\Cr{a:ComC}(k)$ hold, both the numerator and denominator of $r_1$ and $r_{2b_k}$ are positive. This shows that $\vec{\Delta} > 0,$ which verifies the Savage condition. Hence, from Lemma~\ref{lem:GaussianTailBounds},
\[
\Pr\{(\vec{X}_{\vec{0}, \set{b_k}}, X_{\vec{0}}) \geq u \ones_{2b_k + 1}\} \leq \frac{1}{(2 \pi)^{(2b_k + 1)/2}|M|^{1/2} \prod_{j = 1}^{2b_k + 1}\vecp{\Delta}(j)} u^{-(2b_k + 1)} \exp[-\varphi_k u^2].
\]
With this, we are done.

\item Arguments for Statement~\ref{itm:LnkInd}: Let $k \geq 1.$ Fix an arbitrary $G \equiv (V, E) \in \cP_k.$ From \eqref{Defn:cPk} and \eqref{Defn:cP1}, $|V| \geq 2b_k + 1;$ also, there is further subset of $2b_k$ vertices whose pairwise $\|\cdot\|_1-$distances are at least $2.$ Retaining these latter $2b_k$ vertices and choosing any one of the remaining vertices, it is easy to see using Lemma~\ref{lem:SlepainInequality} and $\Cr{a:r2LAsm}$ that
\[
\Pr\{\vec{X}_{\vec{0}, G} \geq u \ones_{|V|}\} \leq \Pr\{(\vec{X}_{\vec{0}, \set{b_k}}, X_{\vec{0}}) \geq u \ones_{2b_k + 1}\}.
\]
The desired result now follows from Statement~\ref{itm:SnkIndWC}, proved above.

Now consider the case $k = 0.$ Let $G \equiv (V, E) \in \cP_0.$ From \eqref{Defn:cP0}, note that $|V| = 2.$ Because $\Cr{a:r2LAsm}$ holds, it follows from Lemma~\ref{lem:SlepainInequality} that
\[
\Pr\{\vec{X}_{\vec{0}, G} \geq u \ones_{2}\} \leq \Pr\{(X_{\vec{0}}, X_{\vec{e}_1}) \geq u \ones_2\}.
\]
The desired result now follows from Statement~\ref{itm:Ln0Ind}, proved above.
\end{itemize}

This completes the proof.
\end{proof}

We next derive the covariance bounds discussed in Lemma~\ref{lem:CovBounds}.

\begin{proof}[Proof of Lemma~\ref{lem:CovBounds}] \label{Pf:CovBounds}
We handle each statement separately. In each case, the idea is to first approximate $\vec{X}_{\tOm, \tpOmp}$ by a simpler Gaussian vector. In that regard, we first introduce a notation. For $k$ such that $1 \leq k < d$ and $\mu \in [0, \rho_1],$ let $\vec{Z} _\mu \sim \Gau(\vec{0}, W_{2b_k, 2b_k}(\rho_2, \mu)),$ where $W_{2b_k, 2b_k}(\rho_2, \mu)$ is as in \eqref{Def:WMat}. Because $\mu \in [0, \rho_1]$ and $\Cr{a:ComC}(k)$ holds, we have $1 + (2b_k - 1) \rho_2 - 2b_k\mu \geq 1 + (2b_k - 1) \rho_2 - 2b_k\rho_1 > 0.$ Using Lemma~\ref{lem:PDofW}, it then follows that the matrix $W_{2b_k, 2b_k}(\rho_2, \mu)$ is positive definite; by definition, it is also symmetric. Consequently, $\vec{Z}_\mu$ is well defined for each $\mu \in [0, \rho_1].$
\begin{itemize}[leftmargin=*]
\item Arguments for Statement~\ref{itm:NonNegCov}: If $\tOm = \tpOmp,$ then the non-negativity trivially follows as the covariance then is simply the variance. Suppose $\tOm \neq \tpOmp.$ Now, let
\[
\hat{\rho}_{\tOm, \tpOmp} := \min\{\rho_{\|(\vec{t} + \alpha \vec{e}_{\omega}) - (\vecp{t} + \alphap \vec{e}_{\omegap})\|_1}: \alpha, \alphap \in \{-1, + 1\}, \; \omega \in \vec{\Omega}, \omegap \in \vecp{\Omega}\}.
\]
As $\tOm \neq \tpOmp,$ we have $\hat{\rho}_{\tOm, \tpOmp} \neq \rho_0;$ in fact, due to $\Cr{a:r1Asm}$ and $\Cr{a:r2LAsm},$ $\hat{\rho}_{\tOm, \tpOmp} \in [0, \rho_1].$
Hence, the random vector $\vec{Z}_{\hat{\rho}_{\tOm, \tpOmp}}$ is well defined. Separately, note that
\begin{equation}
\label{eqn:Z0}
\Pr\{\vec{Z}_{0} \geq \un \ones_{4b_k}\} = \Pr\{\vecXtOm \geq \un \ones_{2b_k}\}\Pr\{\vecXtpOmp \geq \un \ones_{2b_k}\}.
\end{equation}

In terms of the above notations, we have
\begin{eqnarray*}
& & \cov[\xitOm(\un), \xitpOmp(\un)] \\
& = & \Pr\{(\vecXtOm, \vecXtpOmp) \geq \un \ones_{4b_k}\} - \Pr\{\vecXtOm \geq \un \ones_{2b_k}\} \Pr\{\vecXtOm \geq \un \ones_{2b_k}\} \\
& = & [\Pr\{(\vecXtOm, \vecXtpOmp) \geq \un \ones_{4b_k}\} - \Pr\{\vec{Z}_{\hat{\rho}_{\tOm, \tpOmp}} \geq \un \ones_{4b_k} \}]\\
& &  + \; [\Pr\{\vec{Z}_{\hat{\rho}_{\tOm, \tpOmp}} \geq \un \ones_{4b_k} \} - \Pr\{\vec{Z}_{0} \geq \un \ones_{4b_k} \}].
\end{eqnarray*}
Since the last two differences are positive on account of Lemma~\ref{lem:SlepainInequality}, we have the desired result.

\item Arguments for Statement~\ref{itm:WkstBd}: As $\tOm \neq \tpOmp,$ there exists $\alphap \in \{-1, +1\}$ and $\omegap \in \vecp{\Omega}$ such that
\begin{equation}
\label{eqn:distTerm}
\vecp{t} + \alphap \vec{e}_{\omegap} \neq \vec{t}  + \alpha \vec{e}_{\omega}
\end{equation}
for all $\alpha \in \{-1, + 1\}$ and $\omega \in \vec{\Omega}.$ Hence, by dropping some variables, we have
\[
\Pr\{(\vecXtOm, \vecXtpOmp) \geq \un \ones_{4b_k}\} \leq \Pr\{(\vecXtOm, X_{\vecp{t} + \alphap \vec{e}_{\omegap}}) \geq \un \ones_{2b_k + 1}\}.
\]
Again, as \eqref{eqn:distTerm} holds, $\Cr{a:r2LAsm}$ and $\Cr{a:r1Asm}$ imply that $\cov[(X_{\vec{t} + \alpha \vec{e}_{\omega}}, X_{\vecp{t} + \alphap \vec{e}_{\omegap}})] \leq \rho_1$ for all $\alpha \in \{-1, +1\}$ and $\omega \in \vec{\Omega}.$ Using Lemma~\ref{lem:SlepainInequality}, we then have
\[
\Pr\{(\vecXtOm, X_{\vecp{t} + \alphap \vec{e}_{\omegap}}) \geq u \ones_{2b_k + 1}\} \leq \Pr\{(\vec{X}_{\vec{0}, \set{b_k}}, X_{\vec{0}}) \geq u \ones_{2b_k + 1}\}.
\]
The desired result now follows from Lemma~\ref{lem:ProbEst}.\ref{itm:SnkIndWC}.

\item Arguments for Statement~\ref{itm:WkBd}: Since $q \geq 1$ and $\Cr{a:r2LAsm}$ holds, we have $\rhoa_{\tOm, \tpOmp} = \rho_q \in [0, \rho_1].$ Hence, as discussed above, the random vector $\vec{Z}_{\rho_q}$ is well defined.

From Lemma~\ref{lem:SlepainInequality}, observe that
\[
\cov[\xitOm(\un),\xitpOmp(\un)] \leq \Pr\{\vec{Z}_{\rho_q} \geq \un \ones_{4b_k}\}.
\]
With regards to the expression on the right, the notations from Lemma~\ref{lem:GaussianTailBounds} have the following values: $i = 4b_k$ and $M = W_{2b_k, 2b_k}(\rho_2, \rho_q).$ From Lemma~\ref{lem:eValW}, it is easy to see that $\ones_{4b_k}$ is an eigenvector of $M$ with eigenvalue $1 + (2b_k - 1) \rho_2 + 2b_k \rho_q.$ Hence,
\[
\vec{\Delta} =  \un \ones_{4b_k} \; M^{-1}  =
\frac{\un}{1 + (2b_k -1)\rho_2 + 2b_k \rho_q} \ones_{4 b_k} > 0,
\]
where the positivity holds due to $\Cr{a:r2LAsm}.$ Having verified the Savage condition, we have
\begin{eqnarray*}
& & \Pr\{\vec{Z}_{\rho_q} \geq \un \ones_{4b_k)}\} \\
& \leq & \frac{[1 + (2b_k - 1)\rho_2 + 2b_k\rho_q]^{4b_k}}{(2\pi)^{2b_k} (1 - \rho_2)^{2b_k - 1} \sqrt{[1 + (2b_k - 1)\rho_2]^2 - 4b_k^2 \rho_q^2}} \un^{-4b_k} \exp\left[\frac{ -2 a_k \un^2}{1 + 2a_k \rho_q}\right] \\
& \leq & \frac{[1 + (2b_k - 1)\rho_2 + 2b_k\rho_1]^{4b_k}}{(2\pi)^{2b_k} (1 - \rho_2)^{2b_k - 1} \sqrt{[1 + (2b_k - 1)\rho_2]^2 - 4b_k^2 \rho_1^2}} \un^{-4b_k} \exp\left[\frac{ -2 a_k \un^2}{1 + 2a_k \rho_q}\right],
\end{eqnarray*}
where the first relation is due to Lemma~\ref{lem:GaussianTailBounds} and the determinant formula in Lemma~\ref{lem:eValW}, while the last relation holds since $\rho_q \leq \rho_1.$ The desired result is now easy to see.

\item Arguments for Statement~\ref{itm:StBd}: Observe that
\begin{eqnarray*}
& & \cov[\xitOm(\un), \xitpOmp(\un)] \\
& = & \Pr\{(\vecXtOm, \vecXtpOmp) \geq \un \ones_{4b_k}\} - \Pr\{\vecXtOm \geq \un \ones_{2b_k}\} \Pr\{\vecXtpOmp \geq \un \ones_{2b_k}\} \\
& \leq & \Pr\{\vec{Z}_{\rho_q} \geq \un \ones_{4b_k} \} - \Pr\{\vec{Z}_{0} \geq \un \ones_{4b_k}\} \\
& = & \int_0^{\rho_q}  \frac{\partial }{\partial \mu} \left[\Pr\{\vec{Z}_\mu \geq \un \ones_{4b_k}\}\right] \d{\mu} \\
& \leq & \rho_q \; \max_{\mu \in (0, \rho_q)} \left|\frac{\partial }{\partial \mu} \left[\Pr\{\vec{Z}_\mu \geq \un \ones_{4b_k}\}\right]\right|,
\end{eqnarray*}
where the second relation follows from Lemma~\ref{lem:SlepainInequality} and \eqref{eqn:Z0} and the third relation follows from the second fundamental theorem of Calculus. The latter applies since $\rho_q \leq \rho_1,$ which holds due to $\Cr{a:r2LAsm},$  and because $\Pr\{\vec{Z}_\mu \geq \un \ones_{4b_k} \},$ as a function of $\mu,$ is continuous over $[0, \rho_1]$ and differentiable in $(0, \rho_1),$ which itself follows from Lemma~\ref{lem:ZmuTailBeh}. This latter result also shows that, for $\mu \in (0, \rho_q),$
\[
\left|\frac{\partial}{\partial \mu} \left[\Pr\{\vec{Z}_\mu \geq \un \ones_{4b_k}\}\right] \right| \leq O\left( \un^{-(4b_k - 2)} \exp\left[\frac{-2a_k \un^2}{1 + 2 a_k \rho_q}\right]\right).
\]
The desired result is now easy to see.
\end{itemize}
This completes the proof.
\end{proof}

We now obtain the bounds discussed in Lemma~\ref{lem:CovBdsNL}.

\begin{proof}[Proof of Lemma~\ref{lem:CovBdsNL}] \label{Pf:CovBdsNL}
As we have done so far, we talk about each statement separately.
\begin{itemize}[leftmargin=*]
\item Arguments for Statement~\ref{itm:NCovBdk}: Our approach here is a mixture of ideas from the proofs of Lemmas~\ref{lem:ProbEst}.\ref{itm:NnkInd} and ~\ref{lem:CovBounds}.\ref{itm:StBd}.

For $1 \leq j \leq b_k,$ let $W_{j, 2b_k - j}(\rho_2, \rho_3)$ and $Q_{2b_k, j, j}(\rho_2, \rho_3, \mu)$ be defined as in \eqref{Def:WMat} and \eqref{Def:QMat}, respectively. For brevity, and since $\rho_2, \rho_3,$ and $k$ are constants, we shall denote these matrices as $W_j$ and $Q_j(\mu),$ respectively. These simplified notations apply only in this proof; the reader should not  confuse $W_j$ here with $W_m(\rho)$ defined above \eqref{Def:WMat}.

We claim that there exist $\delN > 0$ such that, for all $\mu \in [0, \delN]$ and $1 \leq j \leq b_k,$ the following conditions hold:
\begin{enumerate}
\item \label{rel:ImplHold} $Q_j(\mu)$ is positive definite, hence, invertible; further, its maximum eigenvalue is bounded from above by $\kappa \geq 0,$ a constant which is independent of $\mu$ and $j;$

\item \label{rel:dom} for all large enough $u,$ $\Pr\{\vec{Y}_{\mu, j} \geq u \ones_{4 b_k}\}  \leq \Pr\{\vec{Y}_{\mu, 1} \geq u \ones_{4 b_k}\},$ where $\vec{Y}_{\mu, j} \sim \Gau(\vec{0}, Q_j(\mu));$ this random vector is well-defined since $Q_j(\mu)$ is a symmetric positive define matrix.
\end{enumerate}
Let $\vec{Z}_{\mu} = \vec{Y}_{\mu, 1},$ $\mu \in [0, \delN].$ Assuming the above claim to be true, we now prove the desired result.

We first prove \eqref{eqn:NCovBd}. Using \eqref{Defn:Nxitk}, note that
\[
\Pr\{\vec{Z}_\mu \geq \un \ones_{4b_k}\} - \left[\ExP[\Nxik(\un)]\right]^2 = \Pr\{\vec{Z}_\mu \geq \un \ones_{4b_k}\} - \Pr\{\vec{Z}_0 \geq \un \ones_{4b_k}\}.
\]
The non-negativity is now a simple consequence of Lemma~\ref{lem:SlepainInequality}. To establish the decay rate, set $m = 2b_k,$ $\rho = \rho_2,$ and $\mup = \rho_3.$  Because $\Cr{a:r0Asm}, \Cr{a:r2Asm}$ and $\Cr{a:rqAsm}$ hold, we have \eqref{eqn:IneqRho320}. Further, from condition~\ref{rel:ImplHold} above, it follows that implications of Lemma~\ref{lem:PDofQ} hold in relation to the matrix $Q_j(\mu)$ for $\mu \in [0, \delN]$ and $1 \leq j \leq b_k.$ Now, by arguing as in the proof of Lemma~\ref{lem:CovBounds}.\ref{itm:StBd} and making use of Lemma~\ref{lem:GenContDiff}, it is easy to see that the desired result holds.

Let $\tG, \tpGp$ be as in the statement with $|V_1| = |\Vp_1| = j.$ Also, let $V$ and $\Vp$ be the vertex sets of $G$ and $\Gp,$ respectively. From Lemma~\ref{lem:PwL1Dist}, recall that for each $\vec{t}_{i_1} \in V_1$ (resp. $\Vp_1$) and $\vec{t}_{i_2} \in V \setminus V_1$ (resp. $\Vp \setminus \Vp_1$), we have $\|\vec{t}_{i_1} - \vec{t}_{i_2}\|_1 \geq 3.$ Further, whenever $\vec{t}_{i_1}, \vec{t}_{i_2} \in V_1$ (resp. $\Vp_1$) or $\vec{t}_{i_1}, \vec{t}_{i_2} \in V \setminus V_1$ (resp. $\Vp \setminus \Vp_1$), we have from Corollary~\ref{cor:DistVertMC} that $\|\vec{t}_{i_1} - \vec{t}_{i_2}\| \geq 2.$ Using \eqref{Defn:rhoaN} and the fact that $\Cr{a:r0Asm}$ and $\Cr{a:rqAsm}$ hold, after permuting the coordinates of $\vecXtG$ and $\vecXtpGp$ if necessary, it then follows from Lemma~\ref{lem:SlepainInequality} that
\[
\Pr\{(\vecXtG, \vecXtpGp) \geq \un \ones_{4b_k} \} \leq \Pr\{ \vec{Y}_{\mu, j} \geq \un \ones_{4b_k}\}.
\]
Now, by using condition~\ref{rel:dom} above, we get \eqref{eqn:NdomByY}.

It remains to show the claim above. Towards this, we first prove the alternate claim: there exists $\delN > 0$ so that, for $\mu \in [0, \delN]$ and $1 \leq j \leq b_k,$ (a) condition~\ref{rel:ImplHold} holds; (b) there exist suitable constants (with respect to $\mu$ and $j$) which bound, from below and above, each of $\vecp{\Delta}_{\mu, j} := \ones_{4b_k} Q_j^{-1}(\mu),$ $\ones_{4 b_k} [Q_j^{-1}(\mu)] \ones_{4b_k}^\tr,$ and $|Q_j(\mu)|;$ and, finally, (c) if $j \geq 2,$ then
\[
\ones_{4 b_k} [Q_j^{-1}(\mu)] \ones_{4b_k}^\tr - \ones_{4 b_k} [Q_1^{-1}(\mu)] \ones_{4b_k}^\tr \geq \conN > 0,
\]
where $\conN$ is some constant independent of $\mu$ and $j.$

Because $\Cr{a:r0Asm},$ $\Cr{a:r2Asm},$ and $\Cr{a:rqAsm}$ hold, recall that we have \eqref{eqn:IneqRho320}. Hence, it follows from Lemma~\ref{lem:PDofQ} that there exists some $\delta \equiv \delta(2b_k, \rho_2, \rho_3) > 0$ and $\kappa \equiv \kappa(2b_k, \rho_2, \rho_3) > 0$ such that, for $\mu \in [0, \delta]$ and $1 \leq j \leq b_k,$  the matrix $Q_j(\mu)$ is positive definite; further, its maximum eigenvalue is bounded from above by $\kappa.$

Separately, observe that
\[
\vecp{\Delta}_{0, j} = \ones_{4b_k} Q_j^{-1}(0)
= (\ones_{2b_k} W_j^{-1},  \ones_{2b_k} W_j^{-1})
> 0,
\]
where the second relation follows from the definition of $Q_j(0),$ while the last relation holds as in \eqref{eqn:DelN_Pos} which itself is true due to \eqref{eqn:IneqRho320}.

Also, note that, for $2 \leq j \leq b_k,$
\begin{eqnarray*}
\ones_{4b_k} \, [Q^{-1}_j(0)] \, \ones_{4 b_k}^\tr & = & 2 \, [\ones_{2b_k} \, W_j^{-1} \, \ones_{2 b_k}^\tr] \\
& > &  2 \, [\ones_{2b_k} \, W_1^{-1} \, \ones_{2 b_k}^\tr] \\
& = & \ones_{4b_k} \, [Q^{-1}_1(0)] \, \ones_{4 b_k}^\tr \\
& > & 0,
\end{eqnarray*}
where the first and the third relation follows from the definition of $Q_j(0),$  the second one holds as in \eqref{eqn:DomNInd}, while the last one is true as in \eqref{eqn:DelN_Pos}; recall, both \eqref{eqn:DelN_Pos} and \eqref{eqn:DomNInd} are themselves true due to \eqref{eqn:IneqRho320}.

Further, for $1 \leq j \leq b_k,$ we have
\[
|Q_j(0)| = |W_j|^2 > 0,
\]
where the first relation follows by definition, while the last one follows from Lemma~\ref{lem:PDofW} in which necessary conditions holds due to \eqref{eqn:IneqRho320}.

Lastly, since $\rho_2$ and $\rho_3$ are constants satisfying \eqref{eqn:IneqRho320}, we also have from their definitions that each of $\vecp{\Delta}_{0, j},$ $\ones_{4 b_k} Q_j^{-1}(0) \ones_{4b_k}^\tr,$ and $|Q_j(0)|$ is finite for all $j.$

Using the above arguments and then invoking continuity of rational functions, it is now easy to see that there exists some positive number smaller than $\delta,$ which we denote by $\delN$ henceforth, such that our alternate claim holds. With this $\delN,$ we now establish condition~\ref{rel:dom}.

Let $u > 0,$ $\mu \in [0, \delN],$ and $j$ be so that $2 \leq j \leq b_k.$ Also, let $\vec{\Delta}_{\mu, j} = u \vecp{\Delta}_{\mu, j}.$ Clearly, $\vec{\Delta}_{\mu, j} > 0$ for each $j.$ Hence, it follows from Lemma~\ref{lem:GaussianTailBounds} that
\begin{eqnarray*}
\Pr\{\vec{Y}_{\mu, j} \geq u \ones_{4 b_k}\} & \leq & (2\pi)^{-2b_k} |Q_j(\mu)|^{-1/2} \exp\left[- u^2 [\ones_{4b_k} Q_j^{-1}(\mu) \ones_{4b_k}^\tr]/2 \right] \left[\textstyle \prod_{l = 1}^{4b_k} \vec{\Delta}_{\mu, j}(l)\right]^{-1} \\
& = & (2\pi)^{-2b_k} u^{-4b_k} |Q_j(\mu)|^{-1/2} \exp\left[- u^2 [\ones_{4b_k} Q_j^{-1}(\mu) \ones_{4b_k}^\tr]/2 \right] \left[\textstyle \prod_{l = 1}^{4b_k} \vecp{\Delta}_{\mu, j}(l)\right]^{-1}.
\end{eqnarray*}
Similarly, by using the lower bound in Lemma~\ref{lem:GaussianTailBounds}, we have
\begin{multline*}
\Pr\{\vec{Y}_{\mu, 1} \geq u \ones_{4 b_k}\} \geq
\left[1 - \frac{1}{2} \sum_{l_1, l_2 = 1}^{4b_k} \frac{M_{l_1 l_2}(\mu) [1 + \indc_{l_1 l_2}]}{u^2 \vecp{\Delta}_{\mu, 1}(l_1) \vecp{\Delta}_{\mu, 1}(l_2)}\right] \\ (2\pi)^{-2b_k} u^{-4b_k} |Q_1(\mu)|^{-1/2} \exp\left[- u^2 [\ones_{4b_k} Q_1^{-1}(\mu) \ones_{4b_k}^\tr]/2 \right] \left[\textstyle \prod_{l = 1}^{4b_k} \vecp{\Delta}_{\mu, 1}(l)\right]^{-1},
\end{multline*}
where $M_{l_1 l_2}(\mu)$ is the $l_1 l_2-$th entry of $Q_1^{-1}(\mu).$ Since $\indc_{l_1 l_2} \leq 1,$ note that
\[
1 - \frac{1}{2} \sum_{l_1, l_2 = 1}^{4b_k} \frac{M_{l_1 l_2}(\mu) [1 + \indc_{l_1 l_2}]}{u^2 \vecp{\Delta}_{\mu, 1}(l_1) \vecp{\Delta}_{\mu, 1}(l_2)} \geq 1 - \frac{\ones_{4 b_k} Q_1^{-1}(\mu) \ones_{4 b_k}^\tr}{u^2 \vecp{\Delta}_{\mu, 1}(l_1) \vecp{\Delta}_{\mu, 1}(l_2)}.
\]
Cancelling off the common terms and using part~(c) of our alternate claim, it now follows that
\[
\frac{\Pr\{\vec{Y}_{\mu, j} \geq u \ones_{4 b_k}\}}{\Pr\{\vec{Y}_{\mu, 1} \geq u \ones_{4 b_k}\}} \leq \frac{h_j(\mu) \exp\left[- u^2 \conN / 2 \right]}{\left[1 - g(\mu)/ u^2\right] h_1(\mu)}
\]
for some suitably defined continuous functions $g, h_j: [0, \delN] \to \Real,$ each of which is bounded from below and above by constants (with respect to $\mu$ and $j$) in $(0, \infty).$ Consequently, and because $\conN > 0,$ it follows that for all sufficiently large enough $u,$
\[
\frac{\Pr\{\vec{Y}_{\mu, j} \geq u \ones_{4 b_k}\}}{\Pr\{\vec{Y}_{\mu, 1} \geq u \ones_{4 b_k}\}} \leq 1.
\]
This verifies Condition~\ref{rel:dom}, as desired.

\item Arguments for Statement~\ref{itm:LCovBdk}: Let $W_{1, 2b_k}(\rho_2, \rho_1)$ and $Q_{2b_k + 1, 1, 1}(\rho_2, \rho_1, \mu)$ be defined as in \eqref{Def:WMat} and \eqref{Def:QMat}, respectively. Set $m = 2b_k + 1,$ $\rho = \rho_2,$ and $\mup = \rho_1.$ Because $\Cr{a:r0Asm}, \Cr{a:r1Asm},$ and $\Cr{a:r2LAsm}$ hold, we have $\rho_2, \rho_1 \in [0, 1).$ Since $\Cr{a:ComC}(k)$ also holds, it follows from Lemma~\ref{lem:PDofQSp} that there exists some $\delL = \delta(2b_k + 1, \rho_2, \rho_1)> 0$ and $\kappa \equiv \kappa(2b_k + 1, \rho_2, \rho_1) > 0$ such that the matrix $Q_{2b_k + 1, 1, 1}(\rho_2, \rho_1, \mu)$ is positive definite and, hence, invertible; further, its maximum eigenvalue is bounded from above by $\kappa.$ Now, since $Q_{2b_k + 1, 1, 1}(\rho_2, \rho_1, \mu)$ is also symmetric, it follows that the random variable $\vec{Z}_\mu \sim \Gau(\vec{0}, Q_{2b_k + 1, 1, 1}(\rho_2, \rho_3, \mu))$ is well defined for all $\mu \in [0, \delL].$

Using \eqref{Defn:Lxitk}, observe that
\[
\Pr\{\vec{Z}_\mu \geq \un \ones_{4b_k + 2}\} - \left[\ExP[\Lxik(\un)]\right]^2 = \Pr\{\vec{Z}_\mu \geq \un \ones_{4b_k + 2}\} - \Pr\{\vec{Z}_0 \geq \un \ones_{4b_k + 2}\}.
\]
Now, using Lemma~\ref{lem:GenContDiff} and arguing as in the proof of Lemma~\ref{lem:CovBounds}.\ref{itm:StBd}, we get that \eqref{eqn:LCovBd} holds.

We now prove \eqref{eqn:LDom}. Let $k \geq 1.$ Consider $\tG, \tpGp$ as in the statement. As argued in the proof of Lemma~\ref{lem:ProbEst}.\ref{itm:LnkInd}, it is easy to see from \eqref{Defn:cPk} and \eqref{Defn:cP1} that $|V| \geq 2k + 3$ (resp. $|\Vp| \geq 2k + 3$); further, there exist $2k + 2$ vertices in $V$ (resp. $\Vp$) such that their pairwise $\|\cdot\|_1-$distance is at least $2.$ Retaining only these $2k + 2$ vertices and an additional vertex from remaining ones both in $V$ and $\Vp,$  and then permuting them if necessary, it is not difficult to see using the definition of $\rhoaL_{\tG, \tpGp}$ (given below \eqref{Defn:rhoaN}) and Lemma~\ref{lem:SlepainInequality} that \eqref{eqn:LDom} holds.

Using arguments similar to those above and the ones used in the proof of the $k = 0$ case of Lemma~\ref{lem:ProbEst}.\ref{itm:LnkInd}, it is easy to see that the desired result holds for the $k = 0$ case here as well.
\end{itemize}
This completes the proof.
\end{proof}

\section{Discussion}
\label{sec:Discussion}
Here we expand on some on the remarks given in Section~\ref{sec:intro} and also provide several directions for future research.

We begin by elaborating on Remark~\ref{rem:L1Iso}. Recall that the indicator associated with the random vector given in \eqref{Defn:vecX} approximates the presence of minimal subcomplexes with non-trivial Betti numbers. The $\|\cdot\|_1-$isotropy assumption significantly simplified the covariance matrix associated with this vector. On the other hand, if we had assumed $\|\cdot\|_2-$isotropy, then instead of $W_{2b_k}(\rho_2)$ we would have ended with a matrix whose off-diagonal entries were made up of the two distinct entries $\rho_2$ and $\rho_{\sqrt{2}}.$ This would have made our subsequent computations slightly more involved.

Next, to understand the restriction mentioned in Remark~\ref{rem:LocRest}, let us consider a Gaussian field on the continuum with a smooth covariance function, say $\cov[X_{\vec{t}}, X_{\vecp{t}}] = e^{-\|\vec{t} - \vecp{t}\|_2}.$ Then, note that fine sampling would result in $\rho_1$ and $\rho_2$ being close to each other and also being close to $1,$ thereby reversing the inequality in $\Cr{a:ComC}(k).$ In this sense, this latter covariance condition puts a lower bound on the sampling distances for our results to hold.

Moving on, recall that a key ingredient across our proofs are the covariance bounds given in Lemma~\ref{lem:CovBounds}. These extend the bounds given in \cite[Lemma 3.4 (i), (ii)]{holst1990poisson} to the multivariate case. A multivariate bound similar to the one in Lemma~\ref{lem:CovBounds}.\ref{itm:StBd} was also obtained in \cite[Lemma 1]{arcones1994limit} and plays an important role across \cite{Estrade,kratz2001central, kratz2016central}. While that bound is for more general functions of Gaussian vectors, it is weaker for the specific indicator function that we work with.

Under the independence assumptions mentioned in Remark~\ref{rem:Ind}, note that the tighter covariance bound obtained in Lemma~\ref{lem:CovBounds}.\ref{itm:StBd} would have been no longer necessary, thereby drastically simplifying our calculations. In fact, one could have then used a weaker Stein-Chen bound, as in \cite{kahle2013} for example, to establish the distributional convergences. Another notable difference would have been that the events $\{\vecXtOm \geq u \ones_{2b_k}\}$ and $\{\vecXtG \geq u \ones_{2b_k}\}$ for $G \in \cN_k$ would then have become equiprobable. Recall that these events are associated with $\|\cdot\|_1-$isometric and non-isometric minimal subcomplexes having non-trivial Betti numbers. Due to this, the leading constant in $\lbnk(\un)$ would have changed.

We now point to several interesting questions that arise from the present work. Since Assumption $\Cr{a:ComC}(k)$ was central to all of our theorems, one of the first questions to ask is: `What happens when this condition is not satisfied?'. When $\Cr{a:ComC}(k)$ holds, then it is easy to see from Lemmas~\ref{lem:BettiApprox}.\ref{itm:BdSnku} and \ref{lem:limRates} that $\lim_{n \to \infty} \ExP[\cSnk(\un)]/\ExP[\Snk(\un)] = 1.$ On the other hand, if this condition is not satisfied, the first consequence is that, whenever $\un \to \infty,$
\begin{equation}
\label{eqn:ExPSpnk}
\frac{\ExP[\Spnk(\un)]}{\ExP[\Snk(\un)]}  \to 0,
\end{equation}
where
\[
\Spnk(\un) := \sum_{\tOm \in \IdxnCrDir} \indc[\vecXtOm \geq \un 2b_k] \indc[X_{\vec{t}} < \un].
\]
This can be shown using the bounds for $\Pr\{\vecXtOm \geq \un \ones_{2b_k}, X_{\vec{t}} \leq \un\}$ given in \cite[Theorem 2.2]{dai2001identification}. Since $\Cr{a:ComC}(k)$ does not hold anymore, note that the above result cannot be shown directly using Lemma~\ref{lem:GaussianTailBounds}. This is because one cannot obtain bounds for $\Pr\{(\vecXtOm, X_{\vec{t}}) \geq \un \ones_{2b_k + 1}\}$ as the Savage condition associated with this expression no longer holds. Using \eqref{eqn:ExPSpnk}, it is now easy to see that
\[
\frac{\ExP[\cSnk(\un)]}{\ExP[\Snk(\un)]} \leq \frac{\ExP[\Spnk(\un)]}{\ExP[\Snk(\un)]} \to 0.
\]
This implies that $\Snk(\un)$ can no longer be used to study $\cSnk(\un)$ and, consequently, $\Bnk(\un).$ In other words, one would need to use better approximators for $\Bnk(\un).$ At present, we are unsure if $\cSnk(\un)$ will once again dictate the behaviour of $\Bnk(\un).$ In particular, it is not clear if $\ExP[\Spnk(\un)]$ will dominate the expected count of the more intricate subcomplexes having non-trivial Betti numbers. But if this behaviour indeed holds, then, perhaps, $\Bnk(\un)$ can be studied using $\Spnk(\un).$

As we noted in Remark \ref{rem:CLTlimit}, another obvious work for the future is to show the CLT for Betti numbers in the entire non-vanishing regime. One way to proceed could be to build upon the ideas from the CLT proof in \cite{kahle2013}. There, the Betti number are approximated by the contribution from all those isolated components whose vertex support is bounded by some sufficiently large $m.$ It is then shown that if the regime is so chosen that components of size $m+1$ cannot occur, then a CLT for the approximated Betti number implies one for the actual Betti number itself. Presently, the difficulty in adopting this approach to our setup is in proving the CLT for the isolated components. It appears that the isolation condition, which requires dealing simultaneously with Gaussian random variables exceeding and being below some threshold, entails use of a modified Stein-Chen approach with negative covariances. It is not clear how to generalize Theorem~\ref{thm:ProbSpaceExistence} to handle this case. We expect the computations to be a bit involved. Nevertheless, this should be quite interesting since we would then be able to obtain a better estimate of the variance of $\beta_{n,k}(\un)$ itself; see Remark~\ref{rem:VarEstBnk}.

In a sense, this work studies the distribution of the number of holes for all sufficiently large but fixed excursion level $u$ and window size $n.$ An alternate way to look at this setup would be fix $n$ and only vary $u,$ and then ask questions about the statistics of the range of $u-$values over which each hole persists; note that holes can appear and disappear as $u$ is varied. A formal way to record such birth and death times is via what is known as the `persistence diagram'. In studying the persistence diagram associated with Gaussian excursions, ideas from \cite{skraba2017randomly} should be of help.

Since $\ExP[\beta_{n,k}(\un)]/ \ExP[\beta_{n,0}(\un)]\to 0$ for each $k\geq 1$, and also since there are only finitely many different Betti numbers in any given dimension $d$, we believe that the EPC of $\cK(n; u)$ should also exhibit trivial, Poisson, and CLT behaviour with regimes being determined by those of $\beta_{n,0}(\un)$. In fact, one should also be able to easily prove limit theorems for the LKCs using the results already established in this paper.

Recall that our theorems are proved in the sparse regime. The next logical step would be to derive limiting results in the thermodynamic regime, but without assuming the fast decay conditions as in \cite{Reddy2018}. Here, we believe that the ideas from \cite{yogeshwaran2015topology, yogeshwaran2017random} may turn out be very useful. The resulting theorems would supplement those in \cite{Estrade}, \cite{kratz2016central}, \cite{muller2017central}, etc.

As has been mentioned in \cite{holst1990poisson}, and as is demonstrated by \cite{husler1995rate}, one could use similar techniques as in this paper to establish Poisson approximation theorems for number of exceedances of nonstationary Gaussian sequences as well. This suggests that our results on Betti numbers should also be generalizable to the nonstationary scenario.

Considering $m-$dependent stationary Gaussian sequences, \cite{raab1999compound} and \cite{hashorva2002remarks} obtained compound Poisson approximations for the number of exceedances by suitably modifying the Stein-Chen method used in \cite{holst1990poisson}. In such cases when there are strong local covariances, it is interesting to ask whether we can extend these results to Betti numbers of $m-$dependent Gaussian fields. Since the basic ideas involved are similar, we believe that this should be possible. This work assumes significance from the perspective of ARMA models.

Regarding extending our results to general fields, \cite{raab1997poisson} suggests that our theorems can be also be shown for $\chi^2-$fields. For other fields, we first note that even though Theorem \ref{thm:SuffCondPoissonConv} is quite general, an analogue of Theorem \ref{thm:ProbSpaceExistence} needs to be established. Also, since we relied on precise multivariate Gaussian tail estimates, we will be required to estimate the corresponding tail probabilities for the particular field under consideration.

Lastly, it would be interesting to extend the ideas in this paper and those discussed above to the dynamic setup where the random field of interest also evolves with time. In this direction, ideas from \cite{thoppe2016evolution} should be of help.

\appendix

\section{}

\begin{lemma}
\label{lem:eigM}
Let $\covM_m(\rho)$ be as defined above \eqref{Def:WMat}. Then, the eigenvalues of the matrix $\covM_m(\rho)$ are $(1 - \rho),$ repeated $(m - 1)$  times, and $(1 + (m - 1) \rho).$ The corresponding linearly independent eigenvectors are $\vec{w}_1, \ldots, \vec{w}_m,$ where $\vec{w}_m = \ones_m$ and, for $j = 1, \ldots, m - 1,$
\[
\vec{w}_j(\ell) =
\begin{cases}
1, & \text{ if $\ell = 1,$}\\
-1, & \text{ if $\ell = j + 1,$}\\
0, & \text{otherwise.}
\end{cases}
\]
\end{lemma}
\begin{proof}
This is immediate.
\end{proof}

\begin{lemma}
\label{lem:eValW}
Let $W_{m_1, m_2}(\rho, \mu)$ be as defined in \eqref{Def:WMat} with $m_1, m_2 \geq 1.$ Then, $m_1 + m_2 - 2$ eigenvalues of $W_{m_1, m_2}(\rho, \mu)$ are equal to $(1 - \rho).$ The corresponding linearly independent eigenvectors are $ \vec{w}_1^\prime, \ldots, \vec{w}_{m_1 - 1}^\prime,$ $\vec{w}_1^{\prime\prime}, \ldots, \vec{w}_{m_2 - 1}^{\prime\prime},$ where
\[
\vec{w}_j^\prime(\ell) =
\begin{cases}
1, & \text{ if $\ell = 1,$}\\
-1, & \text{ if $\ell = j + 1,$}\\
0, & \text{ otherwise;}
\end{cases}
\]
and
\[
\vec{w}_j^{\prime\prime}(\ell) =
\begin{cases}
1, & \text{ if $\ell =  m_1 + 1,$}\\
-1, & \text{ if $\ell = m_1 + j + 1,$}\\
0, & \text{ otherwise.}
\end{cases}
\]
Additionally, if $\rho \neq 1$ and if at least one of $(1 + (m_1 - 1)\rho)$ and $(1 + (m_2 - 1)\rho)$ is non-zero, then
\[
|W_{m_1, m_2}(\rho, \mu)| = (1 - \rho)^{m_1 + m_2 - 2} [(1 + (m_1 - 1) \rho) (1 + (m_2 - 1) \rho) - m_1 m_2 \mu^2].
\]
Lastly, the following statements hold when $m_1, m_2$ satisfy some special conditions:
\begin{enumerate}
\item If $m_1 = m_2,$ then the remaining two eigenvalues of $W_{m_1, m_2}(\rho, \mu)$ are $1 + (m_1 - 1)\rho \pm m_1\mu$ with the corresponding eigenvectors being $\ones_{2 m_1}$ and $(\ones_{m_1},  - \ones_{m_1}).$

\item If $m_2 = 1,$ then the remaining two eigenvalues are $[2 + (m_1 - 1)\rho \pm \sqrt{(m_1 - 1)^2 \rho^2 + 4m_1 \mu^2}]/2.$
\end{enumerate}
\end{lemma}
\begin{proof}
It is straightforward to verify that $ \vec{w}_1^\prime, \ldots, \vec{w}_{m_1 - 1}^\prime,$ $\vec{w}_1^{\prime\prime}, \ldots, \vec{w}_{m_2 - 1}^{\prime\prime}$ are indeed independent eigenvectors with eigenvalue $1 - \rho.$

We now establish the determinant formula. Without loss of generality, let $(1 + (m_1 - 1)\rho)$ be non-zero. This, combined with the fact that $\rho \neq 1,$ shows that all the eigenvalues of $W_{m_1}(\rho),$ given by Lemma~\ref{lem:eigM}, are non-zero; hence, it is invertible. Then, from matrix theory concerning determinant of block matrices,
\begin{equation}
\label{eqn:BlockDet}
|W_{m_1, m_2} (\rho, \mu)| = |\covM_{m_1}(\rho)| \; |\covM_{m_2}(\rho) - \mu^2 \onesM_{m_2, m_1} \; \covM_{m_1}^{-1}(\rho) \; \onesM_{m_1, m_2}|.
\end{equation}
From Lemma~\ref{lem:eigM} again, observe that
\begin{equation}
\label{eqn:Det1}
|\covM_{m_1}(\rho)| = (1 - \rho)^{m_1 - 1} (1 + (m_1 - 1)\rho),
\end{equation}
and that $\ones_{m_1}$ is an eigenvector of $W_{m_1}(\rho).$ This latter fact and that $\onesM_{m_2, m_2} = W_{m_2}(1)$ show
\[
\onesM_{m_2, m_1} \; \covM_{m_1}^{-1}(\rho) \; \onesM_{m_1, m_2} = \frac{m_1}{1 + (m_1 - 1) \rho} W_{m_2}(1).
\]
Using this calculation and Lemma~\ref{lem:eigM} one last time, it is now easy to see that the eigenvectors of $\covM_{m_2}(\rho) - \mu^2 \onesM_{m_2, m_1} \; \covM_{m_1}^{-1}(\rho) \; \onesM_{m_1, m_2}$ are exactly as those of $\covM_{m_2}(\rho)$ and the corresponding eigenvalues are $(1 - \rho),$ repeated $(m_2 - 1)$ times, and $(1 + (m_2 - 1) \rho) - m_1 m_2 \mu^2/(1 + (m_1  - 1) \rho).$ The last eigenvalue is well defined since $(1 + (m_1  - 1) \rho)$ is non-zero. Therefore,
\begin{equation}
\label{eqn:Det2}
|\covM_{m_2}(\rho) - \mu^2 \onesM_{m_2, m_1} \; \covM_{m_1}^{-1}(\rho) \; \onesM_{m_1, m_2}| = (1 - \rho)^{m_2 - 1} \left[1 + (m_2 - 1) \rho) - \frac{m_1 m_2 \mu^2}{1 + (m_1  - 1) \rho}\right].
\end{equation}
Substituting \eqref{eqn:Det1} and \eqref{eqn:Det2} in \eqref{eqn:BlockDet}, the desired result follows.

It remains to show the statements on the remaining eigenvalues. The first one is trivially true. So, consider the case that $m_2 = 1.$ From our earlier calculations, we have
\[
|W_{m_1, 1}(\rho, \mu)| = (1 - \rho)^{m_1 - 1}(1 + (m_1- 1) \rho - m_1 \mu^2)
\]
and that $m_1 - 1$ eigenvalues of $W_{m_1, 1}(\rho, \mu)$ are $1 - \rho.$ Separately, the trace of $W_{m_1, 1}(\rho, \mu)$ is $m_1 + 1.$ If we let $\kappa_1$ and $\kappa_2$ be the remaining two eigenvalues, it then follows that
\[
\kappa_1 + \kappa_2 = 2 + (m_1 - 1) \rho \text{ and } \kappa_1 \kappa_2 = 1 + (m_1 - 1) \rho - m_1 \mu^2.
\]
From this, the desired result is easy to see.
\end{proof}

\begin{lemma}
\label{lem:PDofW}
Fix $m_1, m_2 \geq 1.$ Let $\rho, \mu \in [0,1)$ be such that both $1 + (m_1 - 1) \rho - m_1 \mu$ and $1 + (m_2 - 1)\rho - m_2 \mu$ are positive. Then, $W_{m_1, m_2}(\rho, \mu)$ from \eqref{Def:WMat} is positive definite.
\end{lemma}
\begin{proof}
From Lemma~\ref{lem:eValW}, $m_1  + m_2 - 2$ eigenvalues of $W_{m_1, m_2}(\rho, \mu)$ are $1 - \rho;$ these are already positive as $\rho < 1.$ Let $\kappa_1$ and $\kappa_2$ be the remaining two eigenvalues. Clearly, the trace of $W_{m_1, m_2}(\rho, \mu)$ is $m_1 + m_2.$ Consequently, we have $\kappa_1 + \kappa_2 = 2(1 - \rho) + (m_1 + m_2) \rho;$ this is positive since $\rho \in [0, 1).$ Thus, to prove that both $\kappa_1$ and $\kappa_2$ are positive and hence show that $W_{m_1, m_2}(\rho, \mu)$ is positive definite, it suffices to show that $\kappa_1 \kappa_2 > 0.$

Now observe that
\begin{eqnarray*}
\kappa_1 \kappa_2 & = & (1 + (m_1 - 1) \rho)(1 + (m_2 - 1)\rho) - m_1 m_2 \mu^2 \\
& = & (1 + (m_1 - 1)\rho - m_1 \mu)(1 + (m_2 - 1)\rho - m_2 \mu) \\
& & + \; m_1 \mu (1 + (m_2 - 1)\rho - m_2\mu) + m_2 \mu (1 + (m_1 - 1) \rho - m_1 \mu)  \\
& > & 0,
\end{eqnarray*}
where the first relation follows from the determinant formula given in Lemma~\ref{lem:eValW} and the fact that the remaining $m_1 + m_2 - 2$ eigenvalues of $W_{m_1, m_2}(\rho, \mu)$ are $1 - \rho,$ while the last relation follows from the given conditions on $\rho$ and $\mu.$ The desired result is now easy to see.
\end{proof}

\begin{lemma}
\label{lem:ZmuTailBeh}
Fix $m \geq 1$ and $u \in \pReal.$ Let $\rho, \rhop \in [0, 1)$ be such that $1 + (m - 1) \rho - m \rhop > 0.$ For $\mu \in [0, \rhop],$ let $\vec{Z}_\mu \sim \Gau(\vec{0}, W_{m, m}(\rho, \mu)),$ where $W_{m, m}(\rho, \mu)$ is as in \eqref{Def:WMat}. Let $f_{\vec{Z}_\mu}$ denote the density of $\vec{Z}_\mu$ and let $g(\mu) := \Pr\{\vec{Z}_\mu \geq u \ones_{2m}\}.$  Then, the following statements are true.
\begin{enumerate}

\item \label{itm:Cont} $g(\mu)$ is continuous in $[0, \rhop].$

\item \label{itm:DerBd} $g(\mu)$ is differentiable on $(0, \rhop).$ Further, for any $\mu \in (0, \rhop),$
\begin{equation}
\label{eqn:gDerBd}
\left|\frac{\partial}{\partial \mu}[g(\mu)] \right|  \leq  \Cl[Const]{c:IB} u^{-2(m - 1)} \exp\left[\frac{-mu^2}{1 + (m - 1) \rho + m \mu}\right]
\end{equation}

for some constant $\Cr{c:IB} \geq 0$ which depends on $m.$

\end{enumerate}
\end{lemma}

\begin{remark}
Because of Lemma~\ref{lem:PDofW}, $W_{m, m}(\rho, \mu)$ is a symmetric positive definite matrix and, hence, $\vec{Z}_{\mu}$ is well defined.
\end{remark}

\begin{proof}[Proof of Lemma~\ref{lem:ZmuTailBeh}]
We provide arguments for each statement separately.

\begin{itemize}[leftmargin=*]
\item Arguments for Statement~\ref{itm:Cont}: By definition, for any $\vec{z} \in \Real^{2m}$ and $\mu \in [0, \rhop],$ we have
\begin{equation}
\label{eqn:denDf}
f_{\vec{Z}_\mu}(\vec{z}) = \frac{1}{(2\pi)^m  \sqrt{|W_{m, m}(\rho, \mu)|}} \exp\left[\frac{-\vec{z} \, W^{-1}_{m, m}(\rho, \mu)\,  \vec{z}^\tr}{2}\right].
\end{equation}

For each $\mu \in [0, \rhop],$ observe that
\begin{eqnarray}
|W_{m, m}(\rho, \mu)| & = & (1 - \rho)^{2m - 2} [(1 + (m - 1)\rho)^2 - m^2 \mu^2] \label{eqn:DetW}\\
& \geq & (1 - \rho)^{2m - 2} [(1 + (m - 1)\rho)^2 - (m\rhop)^2] \label{eqn:DetWBd}\\
& > & 0 \nonumber,
\end{eqnarray}
where the first relation follows from Lemma~\ref{lem:eValW}, the second holds since $\mu \in [0, \rhop],$ while the last one holds because $\rho, \rhop  \in [0, 1)$ and $1 + (m - 1) \rho - \rhop > 0.$ Again, from Lemma~\ref{lem:eValW} and because $\rho, \mu \geq 0,$ note that the largest eigenvalue of $W_{m, m}(\rho, \mu)$ is $1 + (m - 1) \rho + m \mu;$ since $\mu \leq \rhop,$ it is bounded from above by $1 + (m - 1) \rho + m \rhop.$

Using these calculations, it then follows that
\begin{multline}
\label{eqn:denBd}
|f_{\vec{Z}_\mu}(\vec{z})| \\ \leq \frac{1}{(2\pi)^m (1 - \rho)^{m - 1} \sqrt{[1 + (m - 1)\rho]^2 - [m \rhop]^2}} \exp\left[\frac{-\vec{z}\vec{z}^\tr}{2(1 + (m - 1) \rho + m \rhop)}\right].
\end{multline}
Note that the RHS is integrable and does not depend on $\mu.$

Now, consider any sequence $\{\mu_n\} \subset [0, \rhop]$ with  $\lim_{n \to \infty} \mu_n = \mu^*.$ As $f_{\vec{Z}_\mu}$ is continuous in $\mu$ on $[0, \rhop],$ we have $\lim_{n \to \infty} f_{\vec{Z}_{\mu_n}}(\vec{z}) = f_{\vec{Z}_{\mu^*}}(\vec{z}).$ Using this and the fact that the RHS of \eqref{eqn:denBd} is integrable, it follows from the dominated convergence theorem that $\lim_{n \to \infty} g(\mu_n) = g(\mu^*).$ This shows that $g(\mu)$ is continuous on $[0, \rhop],$ as desired.

\item Arguments for Statement~\ref{itm:DerBd}: For $\hat{h}_{ij}(\mu) = h_{ij}(\mu) / |W_{m, m}(\rho, \mu)|,$ where $\{h_{ij}(\mu)\}$ are some suitable polynomials in $\mu,$ it is easy to see that \eqref{eqn:denDf} can be rewritten as
\[
f_{\vec{Z}_\mu}(\vec{z}) =  \frac{1}{(2\pi)^m \sqrt{|W_{m, m}(\rho,\mu)|}} \exp\left[\frac{-\sum_{i,j=1}^{2m}z(i)\; z(j) \; \hat{h}_{ij}(\mu)}{2}\right].
\]
Hence, for $\mu \in (0, \rhop),$
\[
\frac{\partial}{\partial \mu} \left[f_{\vec{Z}_\mu}(\vec{z})\right] = -\frac{f_{\vec{Z}_\mu}(\vec{z})}{2}\left[\frac{1}{|W_{m,m}(\rho,\mu)|} \frac{\partial}{ \partial \mu}[|W_{m,m}(\rho,\mu)|] +\sum_{i,j}z(i) \; z(j)\frac{\partial}{\partial \mu} \hat{h}_{ij}(\mu)\right].
\]
Like $h_{ij}(\mu),$ note from \eqref{eqn:DetW} that $|W_{m, m}(\rho, \mu)|$ is also a polynomial in $\mu.$ Consequently, we have that $\frac{1}{|W_{m, m}(\rho, \mu)|}\frac{\partial}{ \partial \mu}[|W_{m,m}(\rho,\mu)|]$ and $\frac{\partial}{\partial \mu} \hat{h}_{ij}(\mu)$ are rational functions in $\mu,$ whose denominators are $|W_{m, m}(\rho, \mu)|$ and $|W_{m, m}(\rho, \mu)|^2,$ respectively. Because $[0, \rhop]$ is a compact set and since \eqref{eqn:DetWBd} holds, it then follows that
\[
\sup_{\mu \in (0, \rhop)} \left|\frac{1}{|W_{m, m}(\rho, \mu)|}\frac{\partial}{ \partial \mu}[|W_{m,m}(\rho,\mu)|]\right| = \sup_{\mu \in [0, \rhop]} \left| \frac{1}{|W_{m, m}(\rho, \mu)|}\frac{\partial}{ \partial \mu}[|W_{m,m}(\rho,\mu)|]\right| < \infty;
\]
similarly,
\[
\sup_{\mu \in (0, \rhop)} \left|\frac{\partial}{\partial \mu} \hat{h}_{ij}(\mu)\right| < \infty.
\]
Combining these observations with \eqref{eqn:denBd}, it then follows that there exists a constant $\Cr{c:IB}^{\prime} \geq 0$ depending on $m$ such that
\begin{equation}
\label{eqn:DerDenBd}
\left|\frac{\partial}{\partial \mu} \left[f_{\vec{Z}_\mu}(\vec{z})\right]\right| \leq
\Cr{c:IB}^{\prime} \left[1 + \vec{z} \vec{z}^\tr \right] \exp\left[\frac{-\vec{z} \vec{z}^\tr}{2(1 + (m - 1) \rho + m \rhop)}\right]
\end{equation}
for each $\mu \in (0, \rhop).$ As the RHS is independent of $\mu$ and integrable and since 
\[
\Pr\{\vec{Z}_{\mu} \geq u \ones_{2m}\} = \int_{[u, \infty)^{2m}} f_{\vec{Z}_\mu}(\vec{z})\d{} \vec{z} = \int_{(u, \infty)^{2m}} f_{\vec{Z}_\mu}(\vec{z})\d{} \vec{z},
\]
the dominated convergence theorem shows that, for each $\mu \in (0, \rhop),$
\begin{equation}
\label{eqn:DiffISign}
\frac{\partial}{\partial \mu} \Pr\{\vec{Z}_{\mu} \geq u \ones_{2m}\} = \int_{(u, \infty)^{2m}} \frac{\partial}{ \partial \mu} \left[f_{\vec{Z}_\mu}(\vec{z})\right] \d{\vec{z}}.
\end{equation}
This shows that $g(\mu)$ is differentiable over $(0, \rhop),$ as desired.

It remains to prove \eqref{eqn:gDerBd}. Let $\psi: \Real \to \Real$ be a monotonically increasing $C^\infty-$function satisfying
\[
\psi(z) =
\begin{cases}
0 & \text{ if } z \leq 0,\\
1 & \text{ if } z \geq 1,\\
\end{cases}
\]
and $|\frac{d \psi(z)}{dz}| < \infty.$ Let $\Cr{c:IB}^{\prime \prime} := \sup_{z} |\frac{d \psi(z)}{dz}|.$ For $\epsilon \in (0, 1],$ let $\psi_{\epsilon}(z; u) := \psi\left(\frac{z - u}{\epsilon}\right).$ Clearly, if $\epsilon_1 < \epsilon_2,$ then $\psi_{\epsilon_1}(z; u) \geq \psi_{\epsilon_2}(z; u)$ for all $z.$ Also, $\lim_{\epsilon \to 0}\psi_{\epsilon}(z; u) = \indc[z > u].$ Thus, it follows that $\psi_\epsilon(z; u)$ monotonically increases to $\indc [z > u]$ as $\epsilon \to 0.$ For $\vec{z} \equiv (z(1), \ldots, z(2m)) \in \Real^{2m}$ and $\epsilon \in (0,1],$ let
\[
\Psi_{\epsilon}(\vec{z}; u) := \prod_{i = 1}^{2m} \psi_{\epsilon}(z(i); u).
\]
Then, from \eqref{eqn:DiffISign}, we have
\begin{eqnarray*}
\frac{\partial}{\partial \mu} \Pr\{\vec{Z}_{\mu} \geq u  \ones_{2m}\} &  = & \int_{\Real^{2m}} \left[\prod_{i = 1}^{2m} \indc[\vec{z}(i) > u] \right] \frac{\partial}{\partial \mu} \left[f_{\vec{Z}_\mu}(\vec{z})\right] \d{\vec{z}} \\
& = & \int_{\Real^{2m}} \left[\lim_{\epsilon \to 0} \Psi_{\epsilon}(\vec{z}; u)\right] \frac{\partial}{\partial \mu} \left[f_{\vec{Z}_\mu}(\vec{z})\right] \d{\vec{z}}.
\end{eqnarray*}
As $|\Psi_{\epsilon}(\vec{z}; u)| \leq 1,$ we have, for each $\epsilon > 0,$
\[
\left|\Psi_{\epsilon}(\vec{z}; u)  \frac{\partial}{\partial \mu}  \left[f_{\vec{Z}_\mu}(\vec{z})\right] \right| \leq \left|  \frac{\partial}{\partial \mu} \left[f_{\vec{Z}_\mu}(\vec{z})\right] \right|.
\]
The RHS is integrable for each $\mu \in (0, \rhop),$ as discussed above. Hence, by dominated convergence theorem, it follows that
\begin{equation}
\label{eqn:limitOutside}
\frac{\partial }{\partial \mu} \Pr\{\vec{Z}_{\mu} \geq u \ones_{2m}\} = \lim_{\epsilon \to 0} \int_{\Real^{2m}} \Psi_{\epsilon}(\vec{z}; u)\frac{\partial}{\partial \mu} \left[f_{\vec{Z}_\mu}(\vec{z})\right] \d{\vec{z}}.
\end{equation}

Let $\sigma_{ij}$ denote the $(i, j)-$th entry of $W_{m, m}(\rho, \mu).$ Since $\sigma_{ij} = \mu$ only if either $i \in \{1, \ldots, m\}$ and $j \in \{m + 1, \ldots, 2m\},$ or $i \in \{m + 1, \ldots, 2m\}$ and $j \in \{1, \ldots, m\},$ we have
\begin{equation}
\label{eqn:SumOfPD}
\frac{\partial}{\partial\mu}\left[f_{\vec{Z}_\mu}(\vec{z})\right] = \sum_{i = 1}^{m} \sum_{j = m + 1}^{2m}\frac{\partial}{\partial \sigma_{ij}} \left[f_{\vec{Z}_\mu}(\vec{z})\right] + \sum_{i = m + 1}^{2m} \sum_{j = 1}^{m}\frac{\partial}{\partial \sigma_{ij}} \left[f_{\vec{Z}_\mu}(\vec{z})\right].
\end{equation}
Pick $(i, j)$ where $\sigma_{ij} = \mu.$ Then, by some standard algebra, it is easy to see that
\[
\frac{\partial f_{\vec{Z}_\mu}(\vec{z})}{\partial \sigma_{ij}} = \frac{\partial^2 f_{\vec{Z}_\mu}(\vec{z})}{\partial z(i) \; \partial z(j)}.
\]
From this, using integration by parts twice, it follows that
\[
\int_{\Real^{2m}} \Psi_{\epsilon}(\vec{z}; u)\frac{\partial}{\partial \sigma_{ij}} \left[f_{\vec{Z}_\mu}(\vec{z})\right] \d{\vec{z}} = \int_{\Real^{2m}} \frac{\partial^2\left[\Psi_{\epsilon}(\vec{z}; u) \right]}{\partial z(i)\; \partial z(j)}  f_{\vec{Z}_\mu}(\vec{z}) \d{\vec{z}}.
\]
For some references on the above two steps, see the proof of \cite[Theorem~2.3]{adler1990introduction}.

Now, as $\left|\frac{ \partial \psi_\epsilon(z; u)}{\partial z}\right| \leq \frac{\Cr{c:IB}^{\prime\prime}}{\epsilon}$ for $z \in (u, u + \epsilon)$ and $0$ otherwise and, since, $\psi_\epsilon(z; u) \leq 1$ for $z \in (u, \infty)$ and $0$ otherwise, the above relation shows that
\begin{multline}
\label{eqn:TermijBd1}
\left|\int_{\Real^{2m}} \Psi_{\epsilon}(\vec{z}; u)\frac{\partial}{\partial \sigma_{ij}} \left[f_{\vec{Z}_\mu}(\vec{z})\right] \d{\vec{z}}\right| \\
\leq
[\Cr{c:IB}^{\prime\prime}]^2 \max_{\vec{z}_{ij} \in (u, u + \epsilon)^2} \left[f_{ij}(\vec{z}_{ij}) \; \Pr\{\vec{Z}_{\hat{i}\hat{j}} \geq u \ones_{2m - 2}| \vec{Z}_{ij} = \vec{z}_{ij}\} \right],
\end{multline}
where $\vec{Z}_{ij}$ is a $2-$dimensional vector made up of the $i-$th and $j-$th component of $\vec{Z}_\mu,$ $\vec{Z}_{\hat{i} \hat{j}}$ is $\vec{Z}_\mu$ with $i-$th and $j-$th components deleted, $f_{ij}$ is the density of $\vec{Z}_{ij},$ and $\vec{z}_{ij}$ is the $2-$dimensional vector made up of $z(i)$ and $z(j).$

Let $f_{\hat{i}\hat{j}}(\cdot|\vec{Z}_{ij} = \vec{z}_{ij})$ denote the conditional density of $\vec{Z}_{\hat{i}\hat{j}}.$ Clearly, $\vec{Z}_{\hat{i} \hat{j}} | \vec{Z}_{ij} = \vec{z}_{ij}$ is Gaussian with mean $\mathscr{M}(\vec{z}_{ij}) := \vec{z}_{ij}  \left[\var[\vec{Z}_{ij}]\right]^{-1} \cov[\vec{Z}_{ij}, \vec{Z}_{\hat{i}\hat{j}}] $ and covariance matrix
\begin{equation}
\label{eqn:condCov}
\var[\vec{Z}_{\hat{i}\hat{j}} | \vec{Z}_{ij}] =  \var[\vec{Z}_{\hat{i}\hat{j}}] - \cov[\vec{Z}_{\hat{i}\hat{j}}, \vec{Z}_{ij}] \left[\var[\vec{Z}_{ij}]\right]^{-1} \cov[\vec{Z}_{ij}, \vec{Z}_{\hat{i}\hat{j}}].
\end{equation}
Since $\var[\vec{Z}_{\hat{i}\hat{j}} | \vec{Z}_{ij}]$ is the Schur-complement of $\Var[\vec{Z}_{ij}],$ and since $\Var[\vec{Z}_\mu]$ is positive definite for each $\mu \in [0, \rhop],$ it follows that $\var[\vec{Z}_{\hat{i}\hat{j}} | \vec{Z}_{ij}]$ itself is positive definite for each $\mu \in [0, \rhop].$

From Lemma~\ref{lem:eValW} and, since $\rho, \mu > 0,$ note that
\begin{equation}
\label{eqn:EVRel}
\ones_{2m} \, [W_{m,m}(\rho, \mu)]^{-1}  = \frac{1}{1 + (m - 1) \rho + m\mu} \ones_{2m} >0.
\end{equation}
Hence, it follows that, for any permutation matrix $P,$
\begin{equation}
\label{eqn:PermuationSavage}
\ones_{2m} \, [P \, W_{m, m}(\rho, \mu) \, P^\tr]^{-1} = \frac{1}{1 + (m - 1) \rho + m\mu} \ones_{2m}  > 0.
\end{equation}
%

Now, pick a permutation matrix $P$ so that
\[
P \, W_{m, m} (\rho, \mu) \, P^\tr =
\begin{bmatrix}
\var[\vec{Z}_{\hat{i}\hat{j}}] &  \cov[\vec{Z}_{\hat{i}\hat{j}}, \vec{Z}_{ij}] \\
\cov[\vec{Z}_{ij}, \vec{Z}_{\hat{i}\hat{j}}] & \var[\vec{Z}_{ij}]
\end{bmatrix}.
\]
From \eqref{eqn:condCov}, \eqref{eqn:PermuationSavage}, and the block matrix inversion formula, we then have
\[
\ones_{2m}
\begin{bmatrix}
[\var[\vec{Z}_{\hat{i}\hat{j}}| \vec{Z}_{ij}]]^{-1} \\[0.5ex] - \left[\var[\vec{Z}_{ij}]\right]^{-1} \cov[\vec{Z}_{ij}, \vec{Z}_{\hat{i}\hat{j}}]    [\var[\vec{Z}_{\hat{i}\hat{j}}| \vec{Z}_{ij}]]^{-1}
\end{bmatrix} = \frac{1}{1 + (m - 1) \rho + m\mu} \ones_{2m - 2}
> 0.
\]
Clearly, the above relation also holds if the vector $\ones_{2m}$ is replaced by $u \ones_{2m}.$ Hence,
\begin{equation}
\label{eqn:MeanAtu}
[u \ones_{2m - 2} - \mathscr{M}(u \ones_{2})] \, [\var[\vec{Z}_{\hat{i}\hat{j}}| \vec{Z}_{ij}]]^{-1} = \frac{1}{1 + (m - 1) \rho + m\mu} \ones_{2m - 2} > 0.
\end{equation}
Consequently, by continuity of affine functions, for sufficiently small $\epsilon > 0$ and each $\vec{z}_{ij} \in (u, u + \epsilon)^2,$ the Savage condition
\begin{equation}
\label{Defn:Deltazij}
\vec{\Delta}_{\vec{z}_{ij}} := [u \ones_{2m - 2} - \mathscr{M}(\vec{z}_{ij})] \, [\var[\vec{Z}_{\hat{i}\hat{j}}| \vec{Z}_{ij}]]^{-1} > 0
\end{equation}
holds. Therefore, from Lemma~\ref{lem:GaussianTailBounds}, we have
\begin{eqnarray*}
& & \Pr\{\vec{Z}_{\hat{i}\hat{j}} \geq u \ones_{2m -2} |\vec{Z}_{ij} = \vec{z}_{ij} \}\\
& \leq & \Pr\{\vec{Z}_{\hat{i}\hat{j}} - \mathscr{M}(\vec{z}_{ij}) \geq u \ones_{2m -2}  - \mathscr{M}(\vec{z}_{ij})|\vec{Z}_{ij} = \vec{z}_{ij}\}\\
& \leq & \frac{1}{\prod_{\ell = 1}^{2m - 2} \vec{\Delta}_{\vec{z}_{ij}}(\ell)} f_{\hat{i}\hat{j}}(u \ones_{2m - 2}|Z_{ij} = z_{ij}).
\end{eqnarray*}
Substituting this in \eqref{eqn:TermijBd1}, we get
\[
\left|\int_{\Real^{2m}} \Psi_{\epsilon}(\vec{z}; u)\frac{\partial}{\partial \sigma_{ij}} \left[f_{\vec{Z}_\mu}(\vec{z})\right] \d{\vec{z}}\right| \leq [\Cr{c:IB}^{\prime\prime}]^2 \max_{\vec{z}_{ij} \in (u, u + \epsilon)^2} \left[ \frac{f_{\vec{Z}_\mu}((\vec{z}_{ij}, u \ones_{2m - 2}) P)}{\prod_{\ell = 1}^{2m - 2} \vec{\Delta}_{\vec{z}_{ij}}(\ell)}  \right].
\]
Separately, from \eqref{eqn:MeanAtu} and the fact that $\mu \leq \rhop,$ observe that
\begin{equation}
\label{eqn:DeltaAtu}
\prod_{\ell = 1}^{2m - 2} \vec{\Delta}_{u \ones_{2m - 2}}(\ell) = \frac{u^{2m - 2}}{(1 + (m - 1) \rho + m \mu)^{2m - 2}} \geq \frac{u^{2m - 2}}{(1 + (m - 1) \rho + m \rhop)^{2m - 2}}.
\end{equation}
The above two relations and the fact that both $f_{\vec{Z}_\mu}$ and $\vec{\Delta}$ are continuous functions now show
\begin{multline}
\limsup_{\epsilon \to 0} \left|\int_{\Real^{2m}} \Psi_{\epsilon}(\vec{z}; u)\frac{\partial}{\partial \sigma_{ij}} \left[f_{\vec{Z}_\mu}(\vec{z})\right] \d{\vec{z}} \right| \\
\leq [\Cr{c:IB}^{\prime\prime}]^2 \left[\frac{(1 + (m - 1)\rho + m \rhop)^{2m - 2}}{u^{2m - 2}} f_{\vec{Z}_\mu}(u \ones_{2m}) \right].
\label{eq:boundSd}
\end{multline}
Therefore, by substituting \eqref{eqn:denDf} in \eqref{eq:boundSd} and then making use of the determinant bound given in \eqref{eqn:DetWBd} and the fact that
\begin{equation}
\label{eqn:expArgAtu}
u \ones_{2m} W^{-1}_{m, m}(\rho, \mu) u \ones^\tr_{2m} = 2u^2 m/(1+ (m - 1) \rho + m \mu),
\end{equation}
we have
\begin{eqnarray*}
& & \limsup_{\epsilon \to 0} \left|\int_{\Real^{2m}} \Psi_{\epsilon}(\vec{z}; u)\frac{\partial}{\partial \sigma_{ij}} \left[f_{\vec{Z}_\mu}(\vec{z})\right] \d{\vec{z}} \right|\\
& \leq & [\Cr{c:IB}^{\prime\prime}]^2 \frac{[1 + (m - 1)\rho + m \rhop]^{2m - 2}}{(2\pi)^m (1 - \rho)^{m - 1} \sqrt{[1 + (m - 1)\rho]^2 - [m \rhop]^2}} u^{-2(m - 1)} \exp\left[\frac{-mu^2}{1 + (m - 1) \rho + m \mu}\right].
\end{eqnarray*}
The desired result now follows from \eqref{eqn:limitOutside} and \eqref{eqn:SumOfPD}.
\end{itemize}
This completes the proof.
\end{proof}

\begin{lemma}
\label{lem:detQ}
Let $m \geq 2,$ $1 \leq m_1, m_2 < m,$ and $Q_{m, m_1, m_2}(\rho, \mup, \mu)$ be as in \eqref{Def:QMat}. Let $\rho, \mup \in [0,1)$ be such that one of the following holds:
\begin{enumerate}
\item $\rho > \mup,$ or

\item $m_1 = m_2 = 1$ and $1 + (m - 2) \rho - (m - 1) \mup > 0.$
\end{enumerate}
Then,
\begin{multline*}
|Q_{m, m_1, m_2}(\rho, \mup, \mu)| = |W_{m_1, m - m_1}(\rho, \mup)|  |W_{m_2, m - m_2}(\rho, \mup)| \Big[1 - \\
\mu^2 [\ones_m \, W^{-1}_{m_1, m - m_1}(\rho, \mup) \, \ones_m^\tr] \, [\ones_m \, W^{-1}_{m_2, m - m_2}(\rho, \mup) \, \ones_m^\tr]\Big].
\end{multline*}
\end{lemma}
\begin{proof}
Due to the conditions on $\rho, \mup,$ it follows from Lemma~\ref{lem:PDofW} that both the matrices $W_{m_1, m - m_1}(\rho, \mup)$ and $W_{m_2, m - m_2}(\rho, \mup)$ are positive definite and, hence, invertible.

Now, observe that
\begin{eqnarray*}
& & |Q_{m, m_1, m_2}(\rho, \mup, \mu)|\\
& = & |W_{m_1, m - m_1}(\rho, \mup)| \, |W_{m_2, m - m_2}(\rho, \mup) - \mu^2 \onesM_{m, m} W^{-1}_{m_1, m - m_1}(\rho, \mup) \onesM_{m, m}|\\
& = & |W_{m_1, m - m_1}(\rho, \mup)| \, |W_{m_2, m - m_2}(\rho, \mup) - \mu^2 \ones_m^\tr \ones_m W^{-1}_{m_1, m - m_1}(\rho, \mup) \ones_m^\tr \ones_m|\\
& = & |W_{m_1, m - m_1}(\rho, \mup)| \, |W_{m_2, m - m_2}(\rho, \mup) - \mu^2 [\ones_m W^{-1}_{m_1, m - m_1}(\rho, \mup) \ones_m^\tr] \, \ones_m^\tr \, \ones_m|\\
& = & |W_{m_1, m - m_1}(\rho, \mup)| \, |W_{m_2, m - m_2}(\rho, \mup)| \, \Big|\eye_{m} - \\
& & \mu^2 [\ones_m W^{-1}_{m_1, m - m_1}(\rho, \mup) \ones_m^\tr] \, \ones_m^\tr \, \ones_m W^{-1}_{m_2, m - m_2}(\rho, \mup)\Big|\\
& = & |W_{m_1, m - m_1}(\rho, \mup)|  |W_{m_2, m - m_2}(\rho, \mup)| \Big[1 - \\
& & \mu^2 [\ones_m \, W^{-1}_{m_1, m - m_1}(\rho, \mup) \, \ones_m^\tr] \, [\ones_m \, W^{-1}_{m_2, m - m_2}(\rho, \mup) \, \ones_m^\tr]\Big],
\end{eqnarray*}
where the first relation follows from the determinant formula for block matrices, the second relation holds since $\onesM_{m, m} = \ones_{m}^\tr \ones_{m},$ the fourth relation follows by taking $W_{m_2, m - m_2}(\rho, \mup)$ common and then applying the product rule for determinants (here, $\eye_{m} \in \Real^{m \times m}$ is the identity matrix), while the final relation follows from Sylvester's determinant identity. This gives the desired result.
\end{proof}

For $m \geq 1,$ let $\hat{Q}_m(\mu)$ be the $2m \times 2m-$dimensional matrix given by
\begin{equation}
\label{Defn:hatQ}
\hat{Q}_m(\mu) =
\begin{bmatrix}
0 & \mu \onesM_{m, m} \\
\mu \onesM_{m, m} & 0
\end{bmatrix}.
\end{equation}

\begin{lemma}
\label{lem:eValhatQ}
The eigenvalues of $\hat{Q}_m(\mu)$ are $m \mu,$ $-m \mu,$ and $0$ repeated $2(m - 1)$ times. The corresponding eigenvectors are $\vec{v}_1 = \ones_{2m},$ $\vec{v}_2 = (\ones_m, -\ones_m),$ and $w^\prime_1, \ldots, w^\prime_{m - 1}, w^{\prime \prime}_1, \ldots, w^{\prime \prime}_{m - 1},$ with the latter defined as in Lemma~\ref{lem:eValW} (one needs to set $m_1 = m_2 = m$ there).
\end{lemma}
\begin{proof}
This is straightforward to see.
\end{proof}

\begin{lemma}
\label{lem:PDofQ}
Let $m \geq 2$ and $\rho, \mup \in [0, 1)$ be such that  $\rho > \mup.$ For $1 \leq j < m,$ let $Q_{m, j, j}(\rho, \mup, \mu)$ be defined as in \eqref{Def:QMat}. Then, there exists $\delta \equiv \delta(m, \rho, \mup)$ and $\kappa \equiv \kappa(m, \rho, \mup)$ such that, for any $\mu \in [0, \delta]$ and $1 \leq j < m,$ the matrix $Q_{m, j, j}(\rho, \mup, \mu)$ is positive definite and its maximum eigenvalue is bounded from above by $\kappa.$
\end{lemma}
\begin{proof}
On account of the conditions on $\rho, \mup,$ it follows from Lemma~\ref{lem:PDofW} that $W_{j, m - j}(\rho, \mup)$ is positive definite for each $1 \leq j < m.$ Let $\kmin_j \equiv \kmin_j(m, j, \rho, \mup)$ be the smallest eigenvalue of $W_{j, m - j}(\rho, \mup)$ and let
\[
\kmin \equiv \kmin(m, \rho, \mup) = \min_{1\leq j < m} \kmin_j.
\]
Similarly, define $\kmax$ by considering the largest eigenvalue of $W_{j, m - j}(\rho, \mup)$ for each $j.$ From positive definiteness and since $m$ is finite, we have $0 < \kmin \leq \kmax.$

Clearly, for each $1 \leq j < m,$ the eigenvalues of $Q_{m, j, j}(\rho, \mup, 0)$ lie between $\kmin$ and $\kmax.$ Also, as $\mu \geq 0,$ Lemma~\ref{lem:eValhatQ} shows that the eigenvalues of $\hat{Q}_m(\mu)$ lie between $-m \mu$ and $m\mu.$ Finally, observe that $Q_{m, j, j}(\rho, \mup, \mu) = Q_{m, j, j}(\rho, \mup, 0) + \hat{Q}_m(\mu)$ and that all the three matrices in this relation are symmetric. Therefore, by Weyl's inequality, it follows that all eigenvalues of $Q_{m, j, j}(\rho, \mu, \mup)$ lie between $\kmin - m \mu$ and $\kmax + m \mu$ for all $j.$

Now, set $\delta = \kmin/(2m)$ and $\kappa = \kmax + \kmin/2.$ The desired result is then easy to see.
\end{proof}

\begin{lemma}
\label{lem:PDofQSp}
Let $m \geq 2$ and $\rho, \mup \in [0, 1)$ be such that  $1 + (m - 2)\rho - (m - 1)\mup > 0.$ Then, there exists $\delta \equiv \delta(m, \rho, \mup) > 0$ and $\kappa \equiv \kappa(m, \rho, \mup) > 0$ such that, for any $\mu \in [0, \delta],$ the matrix $Q_{m, 1, 1}(\rho, \mup, \mu),$ defined as in \eqref{Def:QMat}, is positive definite and its maximum eigenvalue is bounded from above by $\kappa.$
\end{lemma}
\begin{proof}
Let $m_2 = m - 1$ and $W_{1, m_2}(\rho, \mup)$ be defined as in \eqref{Def:WMat}. Let $\kmin \equiv \kmin(m, \rho, \mup)$ and $\kmax \equiv \kmax(m, \rho, \mup)$ be the smallest and largest eigenvalues of $W_{1, m_2}(\rho, \mup).$ From the given conditions on $\rho, \mup,$ it is easy to see from Lemma~\ref{lem:PDofW} that the matrix $W_{1, m_2}(\rho, \mup)$ is positive definite; hence, $0 < \kmin \leq \kmax.$

Set $\delta = \kmin/(2m)$ and $\kappa = \kmax + \kmin/2.$ By an application of Weyl's inequality  as in the proof of Lemma~\ref{lem:PDofQ}, the desired result is easy to see.
\end{proof}

\begin{lemma}
\label{lem:GenContDiff}
Fix $m \geq 2$ and $u \in \pReal.$ Suppose one of the following conditions is true:
\begin{enumerate}
\item $\rho, \mup \in [0, 1)$ are such that $\rho > \mup;$ further, $\delta, \kappa > 0$ are such that the implications of Lemma~\ref{lem:PDofQ} hold.

\item $\rho, \mup \in [0, 1)$ are such that $1 + (m - 2)\rho - (m - 1)\mup > 0;$ further, $\delta, \kappa > 0$ are such that the implications of Lemma~\ref{lem:PDofQSp} hold.
\end{enumerate}
For $\mu \in [0, \delta],$ let $\vec{Z}_\mu \sim \Gau(\vec{0}, Q_{m, 1, 1}(\rho, \mup, \mu)),$ where $Q_{m, 1, 1}(\rho, \mup, \mu)$ is as defined in \eqref{Def:QMat}. Let $f_{\vec{Z}_\mu}$ be the density of $\vec{Z}_\mu$ and $g(\mu) := \Pr\{\vec{Z}_\mu \geq u \ones_{2m}\}.$ Then, the following statements are true.
\begin{enumerate}
\item \label{itm:ContG} $g(\mu)$ is continuous in $[0, \delta].$

\item \label{itm:DerBdG} $g(\mu)$ is differentiable on $(0, \delta).$ Further, for any $\mu \in (0, \delta),$
\[
\left|\frac{\partial}{\partial \mu} [g(\mu)]\right|  \leq  \Cl[Const]{c:GenB} u^{-2(m - 1)} \exp\left[\frac{- 2\phi_{m - 1}(\rho, \mup) u^2}{1 + 2\phi_{m - 1}(\rho, \mup) \mu}\right],
\]
where $\Cr{c:GenB} \geq 0$ is some constant which depends on $m$ and $\phi_{m - 1}(\rho, \mup)$ is defined as in \eqref{Defn:phiCons}.
\end{enumerate}
\end{lemma}
\begin{remark}
\label{rem:WellDefZ}
Since $\delta > 0$ is such that the implications of Lemmas~\ref{lem:PDofQ} or \ref{lem:PDofQSp} hold, the matrix $Q_{m, 1, 1}(\rho, \mup, \mu)$ is symmetric and positive definite and, hence, $\vec{Z}_\mu$ is well defined for all $\mu \in [0, \delta].$
\end{remark}

\begin{proof}[Proof of Lemma~\ref{lem:GenContDiff}]
As the line of reasoning is similar to that in the proof of Lemma~\ref{lem:ZmuTailBeh}, our arguments here are brief. We first recall the key intermediate steps from the proof of Lemma~\ref{lem:ZmuTailBeh} and then provide their corresponding variants here.
\begin{itemize}[leftmargin=*]
\item There, the continuity and differentiability of $g(\mu)$ were a simple consequence of the dominated convergence theorem once we obtained functions that dominated $|f_{\vec{Z}_{\mu}}|$ and $\left|\frac{\partial}{\partial \mu} \left[f_{\vec{Z}_\mu}(\vec{z})\right]\right|$ (see \eqref{eqn:denBd} and \eqref{eqn:DerDenBd}, respectively). These dominating functions were derived there by making use of a lower bound on the determinant and an upper bound on the maximum eigenvalue of $\Var[\vec{Z}_\mu]$ (see~\eqref{eqn:DetWBd} and the discussion below it).

\item Separately, after some algebra, the bound on $\left|\frac{\partial}{\partial \mu} [g(\mu)]\right|$ there followed by using the formula for $u \ones_{2m } W^{-1}_{m, m}(\rho, \mu) u \ones_{2m}$ given in \eqref{eqn:expArgAtu}, the fact that \eqref{eqn:EVRel} holds, and finally bounds on $\vec{\Delta}_{u \ones_{2m - 2}}$  and the determinant of $\Var[\vec{Z}_\mu]$ given in \eqref{eqn:DeltaAtu} and \eqref{eqn:DetWBd}, respectively.

\end{itemize}

We now give the equivalent statements here. For brevity, let $Q_m(\rho, \mup, \mu) \equiv Q_{m, 1, 1}(\rho, \mup, \mu).$ First, for each $\mu \in [0, \delta],$ we have
\begin{eqnarray*}
& & |Q_{m}(\rho, \mup, \mu)|\\
& = & |W_{1, m - 1}(\rho, \mup)|^2  \Big[1 - \mu^2 [\ones_m \, W^{-1}_{1, m - 1}(\rho, \mup) \, \ones_m^\tr]^2\Big] \\
& \geq & |W_{1, m - 1}(\rho, \mup)|^2  \Big[1 - \delta^2 [\ones_m \, W^{-1}_{1, m - 1}(\rho, \mup) \, \ones_m^\tr]^2\Big]\\
& = & |Q_m(\rho, \mup, \delta)| \\
& > & 0,
\end{eqnarray*}
where the first relation follows from Lemma~\ref{lem:detQ}, the second one holds due to the fact that $\mu \leq \delta,$ the third one holds due to Lemma~\ref{lem:detQ} again, while the last one follows since $Q_m(\rho, \mup, \delta)$ is positive definite, which itself holds due to Lemma~\ref{lem:PDofQ} or~\ref{lem:PDofQSp}.

Separately, for any $\mu \in [0, \delta],$ note that the largest eigenvalue of $Q_m(\rho, \mup, \mu)$ is bounded from above by some $\kappa > 0,$ independent of $\mu.$

Lastly, by brute force,
\[
(r_1, r_{m - 1} \ones_{m - 1},  r_1 , r_{m - 1} \ones_{m - 1}) \,  Q_m(\rho, \mup, \mu) = \ones_{2m},
\]
where
\[
r_j = \frac{1 + (m - 1 - j) \rho - (m - j)\mup}{1 +  (m - 2) \rho - (m - 1) [\mup]^2 + \mu [m + (m - 2)\rho - 2(m - 1)\mup]},
\]
for $j = 1, m - 1$ and $(r_1, r_{m - 1} \ones_{m - 1}, r_1, r_{m - 1} \ones_{m - 1})$ is the vector whose first entry is $r_1,$ the next $m - 1$ entries are $r_{m - 1}$ and so on. Note that, since $m + (m - 2)\rho - 2(m - 1)\mup = (m - 1)(1 - \mup) + [1 + (m - 2)\rho - (m - 1)\mup],$ the given conditions on $\rho, \mup$ ensure that both the numerators and denominators of $r_1$ and $r_{m - 1}$ are positive. Therefore,
\begin{eqnarray*}
\ones_{2m}  Q^{-1}_m(\rho, \mup, \mu) & \geq &  \ones_{2m} \frac{\min\{1 - \mup, 1 + (m - 2) \rho - (m - 1) \mup\}}{1 +  (m - 2) \rho - (m - 1) [\mup]^2 + \mu [m + (m - 2)\rho - 2(m - 1)\mup]} \\
& \geq & \ones_{2m} \frac{\min\{1 - \mup, 1 + (m - 2) \rho - (m - 1) \mup\}}{1 +  (m - 2) \rho - (m - 1) [\mup]^2 + \delta [m + (m - 2)\rho - 2(m - 1)\mup]}\\
& > & 0,
\end{eqnarray*}
where the second relation holds because $\mu \in [0, \delta],$ while the last one follows from the given conditions on $\rho$ and $\mup$ and since $\delta > 0.$ This shows that a condition equivalent to \eqref{eqn:EVRel} and, consequently, to \eqref{eqn:MeanAtu} holds here. Moving on, the above bound also shows that
\[
\prod_{\ell = 1}^{2m - 2} \vec{\Delta}_{u \ones_{2m - 2}}(\ell) \geq u^{2m - 2} \left[\frac{\min\{1 - \mup, 1 + (m - 2) \rho - (m - 1) \mup\}}{1 +  (m - 2) \rho - (m - 1) [\mup]^2 + \delta [m + (m - 2)\rho - 2(m - 1)\mup]}\right]^{2m - 2},
\]
where $\vec{\Delta}_{u \ones_{2m - 2}}$ is the term equivalent to the one defined in \eqref{Defn:Deltazij}. From this bound, it follows that a condition equivalent to \eqref{eqn:DeltaAtu} is true here. From the values of $r_1$ and $r_{m - 1}$ given above, one can also see that
\begin{eqnarray*}
u \ones_{2m} \, Q_{m}^{-1}(\rho, \mup, \mu) \, u \ones_{2m}^\tr  & = & 2u^2[r_1 +  (m - 1)r_{m - 1}] \\
& = & \frac{2u^2[m + (m - 2)\rho - 2(m - 1)\mup]}{1 +  (m - 2) \rho - (m - 1) [\mup]^2 + \mu [m + (m - 2)\rho - 2(m - 1)\mup]},
\end{eqnarray*}
which gives a formula similar to \eqref{eqn:expArgAtu}; in fact, they match when $\rho = \mup.$

From these relations, the desired bound is now easy to see. This completes the proof.
\end{proof}

\section*{Acknowledgements}
This problem was suggested to us by Robert Adler. He also gave us several key insights during the course of this work that benefitted us tremendously; we thank him for the same. We also sincerely thank D. Yogeshwaran and Primoz Skraba for several useful comments and suggestions.

\bibliographystyle{imsart-nameyear}
\bibliography{GaussianExcursion}

\end{document}